\pdfoutput=1
\documentclass{amsart}
\usepackage{colortbl}
\usepackage{amsmath,amsthm,amssymb}
\usepackage{graphicx}
\usepackage{harvard}

\numberwithin{equation}{section}

\def\C{\mathbb C}

\def\P{{\mathcal P}}

\def\A{{\mathcal A}}

\def\L{{\mathcal L}}
\def\H{{\mathcal H}}
\def\V{{\mathcal V}}

\def\ZZZ{{\mathbb Z}}
\def\LLL{{\mathcal L}}
\def\DDD{{\mathcal D}}
\def\RRR{{\mathbb R}}
\def\NNN{{\mathbb N}}

\def\lll{{\ell}}

\def\overarrow{\overrightarrow}

\def\RRR{{\Bbb R}}

\def\h{{\mathcal H}}

\def\1{\mathbf 1}

\def\A{{\mathcal A}}

\def\H{{\mathcal H}}

\def\L{{\mathcal L}}

\newtheorem{definition}{Definition}
\newtheorem{theorem}{Theorem}
\newtheorem{lemma}{Lemma}
\newtheorem{example}{Example}
\newtheorem{proposition}{Proposition}

\begin{document}

\title{Operator Methods, Abelian Processes and Dynamic
Conditioning}

\author{Claudio Albanese}

\email{claudio@level3finance.com}

\date{First version December 15th, 2006, last revision \today}

\thanks{I would like to thank my collaborators in the past 8 years
with whom this work would not have been possible. In particular
Joseph Campolieti, Peter Carr, Oliver Chen, Alexei Kuznetsov,
Sebastian Jaimungal, Paul Jones, Harry Lo, Stephan Lawi, Alex
Lipton, Alex Mijatovic, Adel Osseiran, Dmitri Rubisov, Stathis
Tompaidis, Manlio Trovato and Alicia Vidler. Special thanks go to
Paul Jones and Adel Osseiran for reading previous versions of this
manuscript and correcting errors. All remaining mistakes are the
solve responsibility of the author.
}

 \maketitle

\begin{abstract}

A mathematical framework for Continuous Time Finance based on
operator algebraic methods offers a new direct and entirely
constructive perspective on the field. It also leads to new
numerical analysis techniques which can take advantage of the
emerging massively parallel GPU architectures which are uniquely
suited to execute large matrix manipulations.

This is partly a review paper as it covers and expands on the
mathematical framework underlying a series of more applied articles.
In addition, this article also presents a few key new theorems that
make the treatment self-contained. Stochastic processes
with continuous time and continuous space variables are defined
constructively by establishing new convergence estimates for Markov
chains on simplicial sequences. We emphasize high precision
computability by numerical linear algebra methods as opposed to the
ability of arriving to analytically closed form expressions in terms
of special functions. Path dependent processes adapted to a given
Markov filtration are associated to an operator algebra. If this
algebra is commutative, the corresponding process is named Abelian,
a concept which provides a far reaching extension of the notion of
stochastic integral. We recover the classic
Cameron-Dyson-Feynman-Girsanov-Ito-Kac-Martin theorem as a
particular case of a broadly general block-diagonalization
algorithm. This technique has many applications ranging from the
problem of pricing cliquets to target-redemption-notes and
volatility derivatives. Non-Abelian processes are also relevant and
appear in several important applications to for instance snowballs
and soft calls. We show that in these cases one can effectively use
block-factorization algorithms. Finally, we discuss the method of
dynamic conditioning that allows one to dynamically correlate over
possibly even hundreds of processes in a numerically noiseless
framework while preserving marginal distributions.

\end{abstract}

 \tableofcontents

\section{Introduction}

The goal of this paper is to attempt to consolidate and present a
number of mathematical methods developed over several years by
myself and collaborators while addressing concrete problems in
derivative pricing theory. The results scattered across a number of
papers which are collected here have been complemented with a
rigorous {\it ab initio} treatment and a few key theorems which make
the framework mathematically self-contained. This results in a quite
comprehensive approach to the theory of Stochastic Processes and
Mathematical Finance which is novel in that it is fully constructive
and perhaps has applications beyond the realm of Financial
Engineering.

There are several traditions of Constructive Mathematics. One
attempts to re-derive classical results of real and functional
analysis based on a restrictive constructivist logic according to
which no mathematical object can be considered unless one specifies
explicitly how to construct it, see \cite{BR1987} and
\cite{Bishop1967}.  Along another tradition, Constructive Field
Theory, see \cite{GJ1987}, aimed at establishing the existence of
interacting quantum field theories by providing a constructive
procedure for computing $n$-point functions and demonstrating that
they satisfy a set of axioms. Measure theoretic probability and the
related theory of stochastic processes, \cite{Doob1953}, does not
seem to be understandable constructively. The PDE approach in
\cite{Feller} and the harmonic analysis approach in
\cite{Bochner1987} are instead essentially constructive but do not
delve into the theory of stochastic integrals and path dependent
processes and into lattice discretization schemes.

The main motivation that guided this research is the creation of an
engineering framework for exotic financial derivatives. Efficient
computability on current hardware has been and remains throughout
this article our key motivating concern. To this end, we work
towards an algebraization of Probability Theory that reduces all
calculations to matrix manipulations which can be performed
efficiently and in particular to matrix multiplications. Similarly
to the standard framework of algebraic topology, \cite{Spanier1966},
we consider processes taking values in separable topological spaces
and approximate continuous domains by means of simplicial sequences.
To establish convergence in the continuous limit, we directly
estimate convergence rates for probability transition kernels in the
continuous space limit following an approach similar in spirit to
Constructive Lattice Field Theory, see \cite{GJ1987}. Similarly to
constructive field theory, sets of axioms on $n$-point functions are
used to identify processes and renormalization group transformations
are used to control the continuous limit. Following
\cite{Naimark1959}, the approach is grounded upon the algebraic
theory of integration on locally compact Hausdorff separable
topological spaces.

Calculations with stochastic processes are carried out using
operator methods developed in Quantum Mechanics,
\cite{LandauLifshits} and systematized in Mathematical Physics
references such as \cite{ReedSimon}. In Finance, operator methods
have been developed along two independent and
 non-overlapping streams of research, one by Ait-Sahalia,
Hansen and Scheinkman who focused on econometric estimations in a
series of papers reviewed in \cite{AHS2005}, see also
\cite{Ait-Sahalia96}, \cite{HST98}, \cite{HS95}. The second stream
of research is by the author and collaborators who instead worked on
derivative pricing for path dependent and correlation derivatives,
see \cite{ACDV}, \cite{AChen4}, \cite{AChen5}, \cite{AKusnetsov5},
\cite{AL2004}, \cite{ALoMijatovic}, \cite{ATrovato3},
\cite{ATrovato2}, \cite{AVidler}, \cite{AJones} and
\cite{AOsseiran}. In this paper, we attempt to systematize the
mathematical framework of pricing theory in the operator formalism
from our own viewpoint, reserving to future work the task of
pursuing overlaps with the econometric literature.

The references quoted above are all relevant to our undertaking and
provided motivations on many levels. However, in an effort to keep
this writing self-contained, we are not going to assume any previous
knowledge of the reader.

To ground the mathematical framework, we obtain sharp pointwise
convergence estimates for probability kernels and its derivatives.
More precisely, we show that probability kernels converge pointwise
at rates of order $O(h^2)$, where $h$ is the lattice spacing. The
result applies to a large class of diffusion processes and
extensions thereof including smooth space inhomogeneities, regime
switching, finite activity jumps and some degree of time
inhomogeneities. We also prove similar convergence results for the
fast exponentiation method, our preferred numerical method for
exponentiating Markov generators, showing that errors in this scheme
are also of order $O(h^2)$, in the sense of pointwise convergence
for the probability kernel. In the particular case of Brownian
motion, we show that similar $O(h^2)$ pointwise error estimates
apply also to derivatives of the probability kernels and that
the power 2 in the $O(h^2)$ bounds is actually sharp.

The interest in convergence estimates was prompted by the desire of
understanding the mechanisms behind the empirically observed
smoothness and robustness in the calculation of price sensitivities
with the methods in our applied papers. We also observed empirically
that the fast exponentiation algorithm is stable under single
precision floating point arithmetics. We find that key to a high
precision numerical framework for sensitivities is handling the time
coordinate either as continuous or very finely discretized. A
sufficiently fine discretization is defined as one for which
explicit differentiation schemes are stable and typically correspond
to a hourly time scale in applications. Typically, weekly time steps
would permit stable implicit differentiation schemes but would not
allow for as much stability in the calculation of price
sensitivities and the probability kernels we require to evaluate and
manipulate.  This motivates us to avoid implicit differentiation
schemes on coarse time intervals.

Measure changes and time changes are defined constructively and a
version of the Fundamental Theorem of Finance is re-obtained. The
Cameron-Martin-Girsanov's theorem, see \cite{CameronMartin},  and
Ito's lemma, see \cite{Ito1951}, are proved twice in different ways
with operator methods. We also derive the Feynman-Kac formula, see
\cite{Feynman1948} and \cite{Kac1950}. One of the key results is an
extension of the Feynman-Kac-Ito formula in three different
directions. This formula concerns the characteristic function of a
stochastic integral over a diffusion process. In our formalism, this
formula becomes a block-diagonalization algorithm for large matrices
associated to path-dependent processes. The extension we discuss (i)
covers Markov processes more general than diffusions, (ii) allows
for a class of path-dependent processes we name {\it Abelian} which
extends the notion of stochastic integral and (iii) generalizes the
harmonic analysis framework to include extensions of trigonometric
Fourier transforms. The theory of Abelian processes finds numerous
practical applications to path dependent options and is applicable
to the great majority of path-dependent payoffs, from volatility
swaps to cliquets, range accruals, lookback options, target
redemption notes and more. We also give a version of Dyson's formula
to accelerate the pricing of path-dependent options given by Abelian
processes by means of a moment expansion. Non-Abelian processes are
more difficult to handle but we single out a class admitting
block-factorizations (as opposed to a block-diagonalizing
transformation) and which are also amenable to numerical analysis by
matrix algebra. Finally, we illustrate the method of dynamic
conditioning that allows one to correlate possibly numerous
processes by means of kernel manipulations while preserving
marginals and not incurring into dimensional explosion.

The mathematical methods in this article are particularly efficient
as they lend themselves to transparent hardware acceleration on the
emerging multi-core GPU hardware platforms. These massively parallel
architectures are based on low-cost technologies that have been
developed for the games and high definition markets and are uniquely
suited to implement BLAS Level 3 routines such as matrix-matrix
multiplications with high efficiency. See also \cite{GotoGeijn} for
a state of the art account on matrix-matrix multiplication software
on CPUs.

The paper is organized as follows. In Section 2, we introduce the
notion of simplicial sequence which is key to devising approximation
schemes for continuous valued process by means of a sequence of
Markov chains. In Section 3, we consider a general definition of
path functional and in Section 4 we give a general description of a
 stochastic process by means of an $n$-point function. Markov
Processes are introduced in Section 5, martingales and monotonic
processes in Section 6. In Section 7, we derive the Fundamental
Theorem of Arbitrage Free Pricing Theory. The classical results on
weak convergence of Markov generators by Bernstein, Bochner,
Kyntchine and Levy are re-obtained in Section 8. Time homogeneous
Markov Processes, fast exponentiation and spectral methods are
described in Section 9. In Section 10, we carry out a constructive
analysis of Brownian motion and prove convergence estimates. In
Section 11, we study the spectrum of diffusion generators. Sharp
pointwise kernel estimates are extended to general diffusion
processes in Section 12. In Section 13, we give estimates for the
convergence rate of time discretisation schemes of the type we
advocate for applications, i.e. based on fast exponentiation.
Section 14 reviews the derivation of hypergeometric Brownian motion
and particular cases such as the CEV model. In Section 15, we study
stochastic integrals and obtain Ito's formula for diffusion
processes and of Girsanov's theorem. Section 16 contains a
derivation of the Feynman-Kac formula for bridges over general
Markov processes. The general notion of Abelian process in
continuous time is introduced in Section 17. In Section 18, we
discuss the discrete time case. In this section, we introduce also
the notion of non-resonant block-diagonalisation scheme which
provides a numerically useful extension of Fourier analysis based on
trigonometric functions. Dyson's formula and moment expansions are
in Section 19, covering the uni-variate case, and Section 20,
covering the multivariate case. These two sections include
applications to exotic volatility derivatives. Block factorizations
and applications to snowballs and soft calls are in Section 21.
Dynamic conditioning and multi-factor correlation modeling is
discussed in Section 22. Conclusions end the paper.

\section{Measure Theory on Simplicial Sequences}

Let $d>0$ be an integer, consider the space $\RRR^d$ and the
sequence of lattices $h_m \ZZZ^d$ where $h_m = 2^{-m}$.

\begin{definition} {\bf (Lattices.)}
If $A\subset h_m\ZZZ^d$, the {\it convex hull} of $A$ in $\RRR^d$ is
denoted by ${\rm Hull}(A)$. The {\it interior} of $A \subset h_m
\ZZZ^d$ is denoted with ${\rm Int }(A)$ and is defined as the set of
all sites $x \in A$ contained in $A$ along with each one its
neighbors in $h_m \ZZZ^d$ at distance $h_m$.
\end{definition}

\begin{definition} {\bf (Simplicial Sequences.)}
A {\it bounded simplicial sequence} is given by an integer $m_0>0$,
and a sequence of subsets $A_m\in\ h_m \ZZZ^d$ defined for all
$m>m_0$ such that
\begin{itemize}
\item{} ${\rm Hull}(A_m) \subset {\rm Hull}(A_{m'})$  whenever $m < m'$.
\item{} for all $m>0$ and all internal points $x\in
{\rm Int }(A_m)$ there is a $M>0$ and an $\epsilon>0$ such that, for
all $m'>M$ and for all $y\in h_{m'}\ZZZ^d$ with $d( x, y) <
\epsilon$ we have that $y\in A_{m'}$ .
\end{itemize}
\end{definition}

We follow a constructivist logic paradigm according to which in
order to identify a set or a sequence of sets one has to explicitly
state how to construct it, possibly with a recursive algorithm, and
it must be possible to decide each step of the recursion in a finite
number of logical steps.

\begin{definition} {\bf (Lattice Functions.)}
Let $A = (A_m)$, $m\ge m_0$ be a bounded simplicial sequence. A {\it
real valued simplicial function}, denoted by $f : A \to \RRR$, is
defined as a sequence of functions $f_m: A_m \to \RRR$ such that
$f_{m'}(x) = f_m(x)$ for all $m' > m \ge m_0$ and all $x\in {\rm
Int}(A_m)$. The function $f : A \to \RRR$ is said {\it uniformly
bounded} if there is a constant $c> 0$ such that $\lvert
f_m(x)\lvert <c$ for all ${x \in A_m}$. The function $f : A \to
\RRR$ is {\it uniformly continuous} if it is uniformly bounded and
for all $\epsilon>0$ there is a $\delta>0$ such that if $x, y \in
h_m \ZZZ$ and $d(x, y)< \delta$ we have that $\lvert f(x)- f(y)
\lvert < \epsilon$.
\end{definition}

\begin{definition}{\bf (Equivalence of Lattice Functions.)}
Let $A = (A_m)$, $m\ge m_0$, be a bounded simplicial sequence and $f
= (f_m)$ and $ g = (g_m)$ are two uniformly continuous real valued
functions on $A$. We say that these series provide equivalent
representations of the same function in case, for all $m>0$ and all
all $x \in h_m \ZZZ$ we have that $\lim_{m'\to\infty} \lvert
f_{m'}(x) - g_{m'}(x) \lvert = 0$.
\end{definition}

Let ${\mathcal C}(A)$ be the set of all continuous functions on $A$
endowed with the natural structure of $C^*$ algebra given by the
operations of sum, multiplication by a scalar and pointwise
multiplication.

\begin{definition} {\bf (Integrals.)}
Let $A = (A_m)$, $m\ge m_0$, be a bounded simplicial sequence. An
{\it integral} is given by a sequence $I = (I_m)$ where $I_m$ is a
linear functional on the linear space of functions $f_m : A_m \to
\RRR$ such that
\begin{itemize}
\item{} $I_m(f_m)\ge 0$ whenever $f_m(x)\ge0$ for all $x\in A_m$.
\item{} The following limit exists for all continuous functions $f = (f_m) \in
{\mathcal C}(A)$:
\begin{equation}
I(f) \equiv \lim_{m\to\infty}  I_m( f_m)
\end{equation}
\item{} The functionals $I_m(f_m)$ defined above satisfies the following bound:
\begin{equation}
I_m(f_m) \leq c \lvert\lvert f_m \lvert\lvert_{\infty}
\end{equation}
for some constant $c>0$ and all functions $f_m$.
\end{itemize}
An integral is said to correspond to a probability measure if
$I_m(1) = 1$.
\end{definition}

\begin{definition}{\bf (Equivalent Integrals.)}
Let $A = (A_m)$, $m\ge m_0$, be a bounded simplicial sequence and
let $(I_m)$ and $ (J_m)$ be two integrals. We say that these series
provide equivalent representations of the same integral $I$ if
$\lim_{m\to\infty} \lvert I_m(f_{m}) - J_{m}(f_{m}) \lvert = 0$ for
all uniformly continuous functions $(f_m)$.
\end{definition}

Given an integral $I$ on the simplicial sequence $A$, for all
$p\ge1$, one defines the following semi-norm on the space of
continuous functions ${\mathcal C}(A)$:
\begin{equation}
\lvert\lvert f \lvert\lvert_{p} = (I(\lvert f \lvert^p ))^{1/p}.
\end{equation}
Let $L^p(A; I)$ be the linear space obtained by completing
${\mathcal C}(A)$ with respect to the semi-norm $\lvert\lvert f
\lvert\lvert_{I, p}$ and identifying equivalence classes of
functions at zero distance. More precisely, $L^p(A; I)$ is the
linear space of the Cauchy sequences $(f_m)$ in ${\mathcal C}(A)$
with respect to the norm $\lvert\lvert f \lvert\lvert_{I, p}$ modulo
the linear space of the Cauchy sequences converging to a limit of
zero $L^p$-norm.

\begin{definition} {\bf (Summable Functions.)}
A function is called summable if it is in $L^1(A; I)$ and square
summable if it is in $L^2(A; I)$.
\end{definition}

\begin{theorem}{\bf (Monotonic Sequences.)}
 Let $f_k$ be a non-decreasing sequence of summable
functions and consider the function $f = \lim_{k\to\infty} f_k$ and
the limit $\bar I = \lim_{k\to\infty} I(f_k)$. Then either $\bar I =
\infty$ or $f$ is summable and $\bar I =I(f)$.
\end{theorem}
\begin{proof}
If $\bar I < \infty$, then whenever $k>j$ we have that $\lvert\lvert
f_k - f_j \lvert\lvert_1 = I(f_k - f_j)$. Since the sequence
$I(f_k)$ is uniformly increasing, it is Cauchy. Hence the sequence
$f_k$ is Cauchy in $L^1(A; I)$. Its limit $f$ therefore belongs to
$L^1(A; I)$ since this space is complete.
\end{proof}

\begin{definition} {\bf (Measurable Functions.)}
A function is called measurable if it can be represented as the
limit of a non-decreasing sequence of summable functions.
\end{definition}

\begin{theorem} {\bf (Dominated Convergence.)} Let $f_k$ be a sequence of measurable
functions, suppose there is a summable function $g$ such that
$\lvert f_k(x) \lvert \leq \lvert g(x) \lvert$ and suppose also that
the limit $f(x) = \lim_{k\to\infty} f_k(x)$ exists in a pointwise
sense. Then $f$ is a summable function and $I(f) = \lim_{k\to\infty}
I(f_k)$.
\end{theorem}
\begin{proof}
Consider the sequence $u_k$ constructed iteratively so that $u_1 =
\min(u_1, f)$ and $u_k = \max(u_{k-1}, \min(f_k, f))$. The sequence
$u_k$ is uniformly non-decreasing and converges to $f(x)$.
Furthermore, $I(\lvert u_k \lvert ) \leq I(\lvert g\lvert)$ for all
$k$. Hence $f$ is summable and $I(f) = \lim_{k\to\infty} I(f_k)$.
\end{proof}

Let ${\mathcal V}$ be the linear space of all functions which can be
expressed as the limit of a non-decreasing sequence of continuous
functions $f(x) = \lim_{k\to\infty} f_k(x)$ such that the following
norm is finite
\begin{equation}
\lvert\lvert f \lvert\lvert_\infty = \lim_{k\to\infty} \lvert\lvert
f_k \lvert\lvert_\infty.
\end{equation}
$L^\infty(A; I)$ is defined as the completion of the linear space
${\mathcal V}$ with respect to the uniform norm.

\begin{definition} {\bf (Essentially Bounded Functions.)}
A function is called {\it essentially bounded} if it is in
$L^\infty(A; I)$.
\end{definition}

\begin{definition} {\bf (Absolute Continuity.)}
Let $I$ and $J$ be two integrals on the simplicial sequence $A =
(A_m)$. $I$ is said absolutely continuous with respect to $J$ if
these two integrals admit simplicial representations $I = (I_m)$ and
$J = (J_m)$ and there exists a $J-$summable function $g = (g_m)$
such that $I_m(f_m) = J_m(g_m f_m)$ for all $I-$summable functions
$f = (f_m)$.
\end{definition}

\section{Path Functionals}

To introduce processes we need to specify the notion of measure on a
set of paths. One could possibly introduce simplicial sequences in
space-time but we refrain from doing so and consider instead the
time coordinate as a continuous variable. In the time direction, we
are only going to be using Riemann integrals for piecewise smooth
functions, so that the underlying time discretisation one might
possibly imagine is quite straightforward and keeping track of it
would only lead to notational complexities.

Let us consider a bounded simplicial sequence $A = (A_m)$ and let
$[T, T']\subset\RRR$ be a fixed finite time interval. Let us
consider the sequence $\Lambda = (\Lambda_{m})$ where $\Lambda_{m} =
A_m \times [T, T']$.

Given an integer $q\ge0$, let $\h^q_{m}$ denote the set of all
functions $y_\cdot :[T, T']\to A_m$ which are constant on a family
of $q$ mutually disjoint sub-intervals of the form $[t_{k_i},
t_{k_{i+1}}), i=0, ..q$ with $t_{k_0} = 0$ and $t_{k_{q+1}} = T'$,
spanning the interval $[T, T']$. Notice that $(\h^q_{m})$ for $q$
fixed is itself a finite dimensional manifold with boundaries. In
fact, a function $y \in \h^q_{m}$ is characterized by the ordered
sequence $t_0=T \leq t_1 \leq .... \leq t_{q+1} = T'$ of time points
between which $y_t$ is constant and a set of values $y_0,
...y_{q-1}$ such that, $y_s = y_i$ for all $s\in[t_i, t_{i+1})$ and
for all $i=0,..q-1$.

The path spaces $\h^q_{m}$ for $q=0,1,...$ can be regarded as nested
into each other $\h_0(T, T') \subset \h_1(T, T') \subset ...$. In
fact, a path $y_t$ in $\h^q_{m}$ is also a path in $\h_{m'}(T, T')$
if $m<m'$. Let $\h(T, T') = \cup_{m=0,1,..} \h^q_{m}$ be the union
of all path spaces containing paths with a finite number of jumps.

\begin{definition}{\bf (Function Algebras.)}
Let us denote with $\C^q(\Lambda)$, the $C^*$ algebra of all
uniformly continuous sequences of simplicial functions $F = F_{m}:
\h^q_{m} \to \RRR$ endowed with the operations of sum,
multiplication and with respect to the uniform norm defined as
follows:
\begin{equation} \lvert\lvert F \lvert\lvert_\infty =
\lim_{m\to\infty} \sup_{y_\cdot \in \h^q_{m}} \lvert
F(y_\cdot)\lvert.
\end{equation}
\end{definition}

\begin{definition} {\bf (Path Functionals.)} A {\it path functional} is a sequence
$F = (F^q), q=0, 1, ...$ of continuous simplicial functions $F^q =
(F^q_{m})\in C^q(\Lambda)$ satisfying the following mutual
compatibility condition:
\begin{equation}
F^{q'}_{m}(y_\cdot) = F^q_{m}(y_\cdot), \;\;\;{\rm for\; all}
\;\;\;\; y_\cdot \in \h^q_{m}\;\;\;{\rm and\; all}\;\;\; q'>q.
\end{equation}
\end{definition}

\begin{definition} {\bf (Non-anticipatory Path Functionals.)}
Let $F(y_\cdot, t) = (F^q(y_\cdot, t)) \;\; t\in [T, T']$ be a
one-parameter family of path functionals. One says that this is a
{\it non anticipatory path functional} if
\begin{equation}
F^q_{m}(y_\cdot, t) = F^q_{m}(y_\cdot', t)
\end{equation}
whenever $y_s = y_{s}'$ for all $s\leq t$.
\end{definition}

Intuitively, non-anticipatory path functionals are indifferent to
information about the realization of the path $y_\cdot$ in the
argument at future times $s>t$. An elementary example of
non-anticipatory path functional is given by a function of two
arguments
\begin{equation}
F^q_{m}(y_\cdot, t) = F_0(y_t, t) \label{eq_defft}
\end{equation}
where $F_0: A_\infty \times [T, T'] \to \RRR$ is a piecewise smooth
one-parameter family of functions. Applications typically call for
functions $F_0(y, t)$ which are piecewise smooth in $t$ for each
$y\in\Lambda$ with possibly a discrete set of jump discontinuities.
We follow the usual convention according to which if jumps occur,
then the discontinuity is of {\it cadlag} type, i.e. right
continuous and with a left limit.

A less elementary example of non-anticipatory path functional one
often encounters is given by integrals of the form:
\begin{equation}
F^q_{m}(y_\cdot, t) \equiv \int_{T}^{t} d s_1 ... \int_{s_{k-1}}^{t}
d s_n a(t; s_1, .... s_k) F_1(y_{s_1}, s_1) ... F_k(y_{s_k}, s_k).
\label{eq_defft2}
\end{equation}
where the $F_j: A_\infty \times [T, T'] \to \RRR$, $j=1, ... k$, are
one-parameter families of lattice functions and also $a(t; s_0, ....
s_k)$ is a function of the time coordinates. Applications typically
call for functions $F_j(y, t)$ which are piecewise smooth in $t$ for
each $y\in\Lambda$ with possibly a discrete set of jump
discontinuities. Also in this case, we follow the convention
according to which at the points of jump discontinuity the function
is {\it cadlag}. The same regularity assumptions will be postulated
for the functions $a(t; s_0, .... s_k)$ with respect to each of its
arguments. Although the regularity assumption for these path
functionals is sufficient for applications, they are not strictly
needed and can be relaxed by taking limits.

\section{n-point Functions}

\begin{definition} {\bf (Filtered Probability Spaces.)}
Consider the $C^*$ algebra $\P$ generated by the functionals of form
(\ref{eq_defft}) and (\ref{eq_defft2}) by taking linear combinations
of finite products and completing the resulting normed space. A {\it
filtered probability space upon the lattice $\Lambda$} is defined as
a bounded linear functional on the $C^*-$algebra $\P$.
\end{definition}

A constructive definition of a stochastic process can be given with
various degrees of generality. We don't aim here at the utmost
generality but instead at pedagogical simplicity and at ultimately
explaining how to frame the theory of Markov processes. As a step
toward this goal, we introduce here a family of time ordered
$n-$point functions
\begin{equation}
\Phi_m(y_1, t_1; ... ;y_n, t_n \lvert y_0)
\end{equation}
defined for $y_1, ... y_n\in A_m$ and $t_1, ... t_n \in [T, T']$
which is assumed to be piecewise differentiable in the time
coordinates. In addition, we assume that the following properties
hold for all $T < t_1\leq ... \leq t_n \leq T'$:
\begin{align*}
&({\rm SP1}) \;\;\;\;\; \Phi(y_1, t_1; ... ;y_n, t_n \lvert y_0) \ge
0 \;\;\;\;\;
\forall y_0, y_1,.. y_n \in\Lambda, \\
&({\rm SP2}) \;\;\;\;\; \sum_{y_1, ... y_n \in\Lambda } \Phi(y_1,
t_1; ... ;y_n, t_n \lvert y_0) = 1
\;\;\;\;\; \forall y_0\in\Lambda, \; \\
&({\rm SP3}) \;\;\;\;\; \Phi(y_1, t_1; ...; y_i, t_i; y_{i+1},
t_{i+1}; ... ;y_n,
t_n \lvert y_0) = 0 \\
&\hskip4cm{\rm if} \; t_i = t_{i+1}
\; {\rm for \; some} \; i = 1 ... n-1 \;\;{\rm and \; if \;}  y_i\neq y_{i+1}\\
&({\rm SP4}) \;\;\;\;\; \sum_{y_i} \Phi(y_1, t_1; ... ;y_{i-1},
t_{i-1}; y_i, t_i;
y_{i+1}, t_{i+1}; ... ;y_n, t_n \lvert y_0) = \\
&\hskip4cm \Phi(y_1, t_1; ...;y_{i-1}, t_{i-1}; y_{i+1}, t_{i+1};
... ;y_n, t_n \lvert y_0)\;\;\;\; \forall y_1,.. y_n \in\Lambda.
\end{align*}

Notice that for each fixed starting point $y_0$ and fixed sequence
$T < t_1\leq ... \leq t_n\leq T'$, the function $\Phi_m(y_1, t_1;
... ;y_n, t_n \lvert y_0)$ is a probability distribution function in
the arguments $y_1,...y_n$. This is interpreted as the probability
distribution density for paths starting from the site $y_0$ at time
$T$ and achieving the values $y_1,...y_n$ at times $t_1, ... t_n$.

\begin{definition} {\bf (Stochastic Processes.)}
An adapted (stochastic) process is given by a non-anticipatory path
functional $F\in\P$ and a measure on $P$ defined by a family of
$n-$point functions $(\Phi)_{n=1,2..}$.
\end{definition}

Given an adapted process of the form $F_t \equiv F(y_t, t)\in\P$ and
given a $t>T$, the expectation subject to the initial condition $y_T
= y_0$ of the value attained by the process $F_t$ at time $t$ is
given by
\begin{equation}
E\big[F_t \lvert y_T = y_0\big] = \sum_{y_1\in\Lambda} \Phi(y_1, t
\lvert y_0) F(y_1, t).
\end{equation}
For an adapted process of the form (\ref{eq_defft2}) instead, the
conditional expectation is given by
\begin{equation}
E\big[F_t \lvert y_T = y_0\big] = \sum_{y_1,..y_n} \int_{T}^{t} d
s_1 ... \int_{s_{n-1}}^{t} d s_n \Phi(y_1, s_1; ... ;y_n, s_n \lvert
y_0) a(t; s_1, ... s_n) F_1(y_1, s_1) ...  F_n(y_n, s_n).
\end{equation}
Also higher moments can be computed. To keep expressions simple,
consider a path functional of the form
\begin{equation}
F_t \equiv \int_{T}^{t} d s_1 a(t; s_1) F_1(y_{s_1}, s_1).
\end{equation}
The variance of this process at a future point in time conditional
to the starting point at time $T$ is given by
\begin{equation}
E\big[F_t^2 \lvert y_T = y_0\big] =  2 \sum_{y_1, y_2}  \int_{T}^{t}
d s_1 \int_{s_1}^{t} d s_2 \Phi(y_1, s_1; y_2, s_2 \lvert y_0) a(t;
s_1) a(t; s_2) F_1(y_1, s_1) F_1(y_2, s_2).
\end{equation}
Notice that the factor 2 compensates for the time-ordering needed to
recast the expression in such a way that we can apply to it a
2-point function. This expression and its generalizations are used
extensively in the applications to path-dependent options discussed
below.

More generally, consider two path functionals $F_t$ and $G_t \in
L^1(\h(T, T'), \mu)$, where $F_t$ is defined as in (\ref{eq_defft2})
and
\begin{equation}
G_t \equiv \int_{T}^{t} d s_1 ... \int_{s_{q-1}}^{t} d s_q b(t; s_1,
.... s_q) G_1(y_{s_1}, s_1) ...  G_q(y_{s_q}, s_q).
\end{equation}
Then we can compute a mixed moment as follows:
\begin{align*}
& E\big[F_t G_t \lvert y_T = y_0\big] = \sum_{y_1,.. y_{n+q}}
\int_{T}^{t} d s_1 ... \int_{s_{n+q-1}}^{t} d s_{n+q} \;\;\; \Phi(
y_1, s_1; .... ; y_{n+q}, s_{n+q})
\\
& \hskip3.5cm \sum_\pi a(t; s_{\pi(1)}, .... s_{\pi(n)})
F_1(y_{\pi(1)}, s_{\pi(1)}) ...  F_n(y_{\pi(n)}, s_{\pi(n)})
\\
&\hskip3.5cm b(t; s_{\pi(n+1)}, .... s_{\pi(n+q)}) G_1(y_{\pi(n+1)},
s_{\pi(n+1)}) ... G_q(y_{\pi(n+q)}, s_{\pi(n+q)}),
\end{align*}
where the sum ranges over all the permutations $\pi$ of the time
ordered sequence $s_1\leq ... \leq s_{n+q}$ such that
$s_{\pi(1)}\leq ... \leq s_{\pi(n)}$ and $s_{\pi(n+1)}\leq ... \leq
s_{\pi(n+q)}$. Higher conditional moments of any path functional of
the form above can be evaluated in a similar way.

$n$-point functions $\Phi(y_1, s_1; ... ;y_n, s_n \lvert y_0)$ are
conditioned to the starting point $y_0$. It is straightforward to
define variations of these $n$-point functions which reflect
conditioning on an initial stretch of a path $y_s$ for $s\in[T, t)$,
where $t>T$. Conditioning to a past history is equivalent to
restricting the integral over each manifold $\H(T, T')$ to a
sub-manifold which is part of its boundary. In general, this results
in rather clumsy expressions which are difficult to compute. In the
next section we specialize to the Markovian case for the underlying
lattice process and in this case there are no memory effects and
conditioning is more straightforward.

\begin{definition} {\bf (Radon-Nykodim Derivative.)}
Consider two sequences of $n$-point functions on the simplicial
sequence $A_m$ $\Phi^1_m(y_1, t_1; ... ;y_n, t_n \lvert y_0)$ and
$\Phi^2_m(y_1, t_1; ... ;y_n, t_n \lvert y_0)$ , where $m=1,2,...$.
The Radon-Nykodim derivative of $\Phi^1$ with respect to $\Phi^2$ is
given by the path functional defined as follows:
\begin{equation}
\rho(y_\cdot) = \lim_{n\to\infty} { \Phi^1_m(y_{t_1}, t_1; ...
;y_{t_n},t_n\lvert y_0) \over \Phi^2_m(y_{t_1}, t_1; ...
;y_{t_n},t_n\lvert y_0)}
\end{equation}
where $t_i = T+{i\over n}(T'-T)$. Notice that the Radon-Nykodim
derivative might possibly be infinite on some paths.

The measure in path space given by $(\Phi^1_m)$ is said to be
absolutely continuous with respect to $(\Phi^2_m)$ if the
Radon-Nykodim derivative is finite and summable.
\end{definition}

Finally, a word on stopping times.
\begin{definition} {\bf (Stopping Times.)}
A stopping time is an adapted process $\tau_t = \tau(y_\cdot, t)$
which can take only two values, by convention 0 and 1.
\end{definition}
Stopping times are often used in conjunction with other adapted
processes $F_t$ to construct stopped versions of it. If $F_t =
F(y_\cdot, t)$ is an adapted process and $R_t = R(y_\cdot, t)$ is a
second process, then a stopped version of $F_t$ corresponds to the
adapted process of function $\bar F(y_\cdot, t)$, where $\bar
F(y_\cdot, t) = F(y_\cdot, t)$ if $\tau_t = 0$ and $\bar F(y_\cdot,
t) = R(y_\cdot, t)$ if $\tau_t = 1$.

\section{Markov Processes}

\begin{definition} {\bf (Markov Propagator.)}
A Markov propagator on the simplicial sequence $A_m, m\ge m_0$ is
defined as a sequence of functions $U_m(y_1, t_1; y_2, t_2)$ where
$y_1, y_2\in A_m$ and $T \leq t_1 \leq t_2 \leq T' $ satisfying the
following Chapman-Kolmogorov axioms:
\begin{align*}
&({\rm CK1}) \;\;\;\;\; U_m(y_1, t_1; y_2, t_2) \ge 0 \;\;\;\;\;
\forall y_1, y_2 \in\Lambda,
\;\;\;\forall t_1\leq t_2 \in \RRR,\\
&({\rm CK2}) \;\;\;\;\; U_m(y_1, t_1; y_2, t_1) = \delta_{y_1y_2}
\;\;\;\;\; \forall y_1, y_2 \in\Lambda,
\;\;\;\forall t_1 \in \RRR,\\
&({\rm CK3}) \;\;\;\;\; \sum_{y_2\in\Lambda } U_m(y_1, t_1; y_2,
t_2) U_m(y_2, t_2;
y_3, t_3) = U_m(y_1, t_1; y_3, t_3)  \\
&\hskip8cm \forall y_1, y_3 \in\Lambda,\;\;\forall t_1\leq t_2\leq
t_3 \in \RRR.
\end{align*}
\end{definition}

Given a Markov propagator $U_m(y_1, t_1; y_2, t_2)$, one can define
a sequence of $n-$point function having all the necessary properties
by setting
\begin{equation}
\Phi_m(y_1, t_1; ... ;y_n, t_n \lvert y_0) = \prod_{j=1..n}
U_m(y_{j-1}, t_{j-1}; y_j, t_j).
\end{equation}
where $t_0 = T$.

\begin{definition}{\bf (Markov Process.)}
A filtered probability space generated by a Markov propagator is
called Markovian and the corresponding stochastic process is called
{\it Markov process.}
\end{definition}

\begin{definition} {\bf (Markov Generator.)}
If the matrix elements $U_m(y_1, t_1; y_2, t_2)$ are differentiable
functions of the time parameter $t_2$ in a right neighborhood of
$t_1$, then one defines the Markov generator at time $t_1$ as the
following right derivative:
\begin{equation}
\LLL_m(y_1, y_2; t_1) = \lim_{t_2 \downarrow t_1} {d\over dt_2}
U_m(y_1, t_1; y_2, t_2) \label{eq_gen}
\end{equation}
\end{definition}

\begin{proposition}
If $\LLL_m(y_1, y_2; t_1)$ is a Markov generator, then for all pairs
$y_1, y_2 \in A_m$ the following two properties hold:
\begin{align}
&({\rm MG1}) \;\;\;\;\; \LLL_m(y_1, y_2; t) \ge 0 \;\;\;\;\; {\rm if}\;\; y_1 \neq y_2, \\
&({\rm MG2}) \;\;\;\;\; \LLL_m(y_1, y_1; t) = - \sum_{y_2\neq y_1}
\LLL_m(y_1, y_2; t).
\end{align}
Viceversa, if $\LLL_m(y_1, y_2; t)$ is a differentiable
one-parameter family of matrices satisfying conditions $(A)$ and
$(B)$ above, then the differential equation (\ref{eq_gen}) admits
one and only one solution satisfying the initial condition $U_m(y_1,
t; y_2, t) = \delta_{y_1 y_2}$.
\end{proposition}

The propagator $U_m(y_1, t_1; y_2, t_2)$ defined by the differential
equation in (\ref{eq_gen}) can be represented by means of a
so-called {\it path-exponential} defined as follows. Let $N>0$ be an
integer and let us consider the product
\begin{equation}
U^N_m(y_1, t_1; y_2, t_2) = \bigg( 1 + \delta t_N \LLL_m(t_1)\bigg)
\cdot .... \cdot \bigg( 1 + \delta t_N
\LLL_m(t_N)\bigg)\label{eq_pexp1}
\end{equation}
where $\delta t_N = {t_2-t_1\over N}$. If $N$ is so large that
\begin{equation}
(\delta t)_N < \min_{y\in\Lambda, t\in[t_1, t_2]} \LLL_m(y,y;
t)^{-1},
\end{equation}
then the operator product $U^N_m(y_1, t_1; y_2, t_2)$ in equation
(\ref{eq_pexp1}) is a probability kernel. Passing to the limit
$N\to\infty$, we find
\begin{equation}
\lim_{N\to\infty}U^N_m(y_1, t_1; y_2, t_2) = U_m(y_1, t_1; y_2,
t_2).
\end{equation}
By expanding the product in (\ref{eq_pexp1}) and passing to the
limit, we arrive at the following:
\begin{theorem} {\bf (Dyson expansion.)} The probability kernel can
be represented as the following convergent series:
\begin{equation}
U_m(y_1, t_1; y_2, t_2) = \bigg(1 + \sum_{n=1}^\infty
\int_{t_1}^{t_2} ds_1 \int_{s_1}^{t_2} ds_2 ... \int_{s_{n-1}}^{t_2}
ds_n \LLL_m(s_1) \LLL_m(s_2) .... \LLL_m(s_n)\bigg)(y_1, y_2).
\label{ma_dyson}
\end{equation}
\end{theorem}
By differentiating with respect to the two time coordinates in the
Dyson expansion, we also find the following two equations:
\begin{theorem} {\bf (Forward and backward equations.)} The probability kernel
satisfies the backward equation
\begin{equation}
{\partial\over\partial t_1} U_m(t_1; t_2) + \LLL_m(t_1) U_m(t_1;
t_2) = 0 \label{backward_eq}
\end{equation}
as well as the forward equation
\begin{equation}
{\partial\over\partial t_2} U_m(t_1; t_2) =  U_m(t_1;
t_2)\LLL_m(t_2).  \label{forward_eq}
\end{equation}
\end{theorem}

A handy notation for this expansion is given by the following:
\begin{definition} {\bf (Path-ordered Exponential.)}
The equation (\ref{ma_dyson}) is written as a
{\it path ordered exponential}
\begin{equation}
U_m(y_1, t_1; y_2, t_2) = {\rm Pexp}\bigg( \int_{t_1}^{t_2}
\LLL_m(s) ds \bigg)(y1, y2)
\end{equation}
where the operator $P$ formally acts as follows:
\begin{equation}
{\rm P} \bigg(\int_{t_1}^{t_2} \LLL_m(s) ds\bigg)^n = n!
\int_{t_1}^{t_2} ds_1 \int_{s_1}^{t_2} ds_2 ... \int_{s_{n-1}}^{t_2}
ds_n \LLL_m(s_1) \LLL_m(s_2) .... \LLL_m(s_n).
\end{equation}
\end{definition}

Using the Dyson expansion for the path-ordered exponential, one
finds a path-integral representation of the probability kernel. Let
us set the following definition:
\begin{definition} {\bf (Symbolic Path.)}
A symbolic path  $\gamma = \{\gamma_0, \gamma_1, \gamma_2, .... \}$
is an infinite sequence of sites in $h_m \ZZZ$ such that $\gamma_j
\neq \gamma_{j-1}$ for all $j=1, ...$. Let $\Gamma_m$ be the set of
all symbolic paths in $h_m \ZZZ$.
\end{definition}
\begin{theorem} {\bf (Path-Integral Representation.)}
The propagator admits the following representation:
\begin{align}
U_m(x, T; y, T') =& \sum_{q=1}^\infty \sum_{\gamma\in \Gamma_m :
\gamma_0 = x, \gamma_{q} = y}
 \int_T^{T'} ds_1 \int_{s_1}^{T'} ds_2 ...
\int_{s_{q-1}}^{T'} d s_{q} \\
& \exp\bigg({\int_{0}^{s_{1}} \LLL_m(\gamma_0, \gamma_0; v_0) d
v_0}\bigg) \prod_{j=1}^{q} \LLL_m(\gamma_{j-1}, \gamma_j; s_j)
\exp\bigg({\int_{s_{j}}^{s_{j+1}} \LLL_m(\gamma_j, \gamma_j; v_j) d
v_j}\bigg) \label{eq_upathint}
\end{align}
where $t_{q+1} = T$.
\end{theorem}

Notice that the total mass of the sector of path space
$\h^q_{m}(y_0, T; T')$ consisting of lattice paths originating from
$y_0\in\Lambda $ at time $T$ and attaining at most $q$ different
values by time $T'$ is given by
\begin{align*}
\mu(\h^q_{m}(y_0, T; T')) = \sum_{\gamma\in \Gamma_m : \gamma_0 = x}
\int_T^{T'} ds_1 \int_{s_1}^{T'} ds_2 ... \int_{s_{q-1}}^{T'} d
s_{q}
\exp\bigg({\int_T^{s_1} \LLL_m(\gamma_0, \gamma_0; v_0) d v_0}\bigg) \\
\prod_{j=1}^{q} \LLL_m(\gamma_{j-1}, \gamma_j; s_j)
\exp\bigg({\int_{s_{j}}^{s_{j+1}} \LLL_m(\gamma_j, \gamma_j; v_j) d
v_j }\bigg).
\end{align*}
where $s_{m+1} = T'$. If $\lambda_m = \max_{x\in A_m} \lvert
\LLL_m(x,x)\lvert$ and $c = \sup_{t\in[T, T']} \lvert\lvert
\LLL_m(t) \lvert\lvert$ is the operator norm of the matrix
$\LLL_m(t)$, then we have that
\begin{equation}
\mu\big[\h^q_{m}(y_0, T; T')\big] \leq {c^q (T'-T)^q \over q!}
e^{-\lambda_m (T'-T)} \label{eq_mest}
\end{equation}
This expression reaches a maximum as a function of $q$ at $q \approx
c (T'-T)$ and then declines at super-exponential rate. Hence, by
discretizing the space coordinate, we ensure that the probability
mass of the set of paths with $q$ changes decreases faster than any
exponential in the limit as $q\to\infty$. This convergence is
however not uniform as $m\to\infty$ and  $h_m\to0$.

\begin{definition}{\bf (Inverse lattice.)}
Let us consider the following lattice:
\begin{equation}
B_m = \left\{ -{2^{m-1} \pi \over L} + {k \pi \over L } , k = 0, ..
2^m-1 \right\} \label{brillouin}
\end{equation}
also called {\it Brillouin zone} or {\it inverse lattice} with
respect to $A_m$.
\end{definition}

\begin{definition}\label{def_symbol}{\bf (Pseudo-differential Symbols.)}
The symbol of a Markov generator is defined as follows:
\begin{equation}
\hat {\mathcal L}_m(x, p; t) = \sum_{y\in A_m} {\mathcal L}_m(x, y)
e^{ip(y-x)}
\end{equation}
where $p\in B_m$.
\end{definition}

Two particularly important special examples of Markov processes are
given by monotonic processes and diffusions.

\begin{definition}\label{def_monotonic}
{\bf (Monotonic process.)} A Markov process of generator ${\mathcal
L}_m(x, y)$ on the simplicial sequence $A_m$ is said monotonic
non-decreasing if
\begin{equation}
{\mathcal L}_m(x, y) = 0 \;\;\;\;{\rm whenever}\;\;\; y<x.
\end{equation}
and monotonic non-increasing if
\begin{equation}
{\mathcal L}_m(x, y) = 0 \;\;\;\;{\rm whenever}\;\;\; y>x.
\end{equation}
\end{definition}

\begin{definition} {\bf (Diffusion Process.)} A diffusion process is
a Markov process with generator of the form
\begin{equation}
{\mathcal L}_m(x, y; t) = \mu_m(x;t) \nabla_{h_m}(x, y)
+{\sigma_m(x; t)^2 \over 2} \Delta_{h_m}(x, y)
\end{equation}
where
\begin{equation}
\nabla_{h}(x, y) = {\delta_{y, x+h} - \delta_{y,x-h} \over 2 h}
\;\;\;\;{\rm and}\;\;\;\; \Delta_{h}(x, y) = {\delta_{y, x+h}+
\delta_{y, x-h} - 2 \delta_{y, x} \over h^2}. \label{eq_nabladelta}
\end{equation}
and $\mu(x;t) = (\mu_m(x;t))$ and $\sigma(x; t) = (\sigma_m(x; t))$
are two simplicial functions which, for simplicity, we assume smooth
in both arguments.
\end{definition}

\begin{proposition} \begin{itemize}
\item[(i)] The symbol of a diffusion process is given by
\begin{equation}
\hat {\mathcal L}_m(x, p; t) = \mu_m(x; t) {\sin ph \over i h} +
{\sigma_m(x; t)^2 \over 2} {\cos ph - 1 \over h^2}.
\end{equation}
\item[(ii)] The path integral representation for a diffusion process has the
form
\begin{align}
U_m(x, T; y, T') =& \sum_{q=1}^\infty 2 ^{-q} \sum_{
\begin{matrix}
\gamma\in \Gamma_m : \gamma_0 = x, \gamma_q = y \\
\lvert\gamma_j - \gamma_{j-1}\lvert = 1 \;\;\forall j\ge1
\end{matrix}}
W_m(\gamma, q; T, T') \label{eq_udiff1}
\end{align}
where
\begin{align}
&W_m(\gamma, q; T, T') = \int_T^{T'} ds_1 \int_{s_1}^{T'} ds_2 ...
\int_{s_{q-1}}^{T'} d s_{q} \notag \\
&\hskip3.5cm e^{ \int_{s_q}^{T'} {\mathcal L}_m(y, y, v_q) dv_q}
\prod_{j=0}^{q-1}  \bigg(  e^{ \int_{s_j}^{s_{j+1}} {\mathcal
L}_m(\gamma_j, \gamma_{j}, v_j) dv_j} 2 {\mathcal L}_m(\gamma_j,
\gamma_{j+1}) \bigg) \label{eq_wdiff}
\end{align}
and $s_0 = T$.
\end{itemize}
\end{proposition}

\section{Martingales and Monotonic Processes}

In this section we introduce the notion of piecewise smooth Markov
process which covers a large family of models useful for
applications. In this context, we define martingales and monotonic
Markov processes.

\begin{definition} {\bf (Piecewise Smooth Markov Processes.)}
Consider the time interval $[T, T']$, the simplicial sequence $A_m$,
$m\ge m_0$ and a finite number of time points $T = t_0 < t_1 < .. < t_n = T'$.
A piecewise smooth Markov process is given by
a family of Markov generators $\LLL^i_m(y_1, y_2; t)$ defined
on each half open time interval $[t_i, t_{i+1})$, where $i=0, 1, ... n-1 $.
In correspondence to each time point $t_i, i = 1, 2, ... n$, one also defines
a mapping operator $U^i_m(x, y)$ such that
\begin{align}
&({\rm MA1}) \;\;\;\;\; U^i_m(y_1, y_2) \ge 0\\
&({\rm MA2}) \;\;\;\;\; \sum_{y_2} U^i_m(y_1, y_2) = 1 \;\; \forall
y_1\in\Lambda.
\end{align}
The Markov propagator for any pair of time points $t_i \leq s < s' <
t_{i+1}$ is defined as follows:
\begin{equation}
U_m(y_1, s; y_2, s') = {\rm Pexp}\left(\int_s^{s'}\LLL^i_m( v)
dv\right)(y_1, y_2).
\end{equation}
Moreover, if $s' = t_{i+1}$, then
\begin{equation}
U_m(y_1, s; y_2, t_{i+1}) = \sum_{y_3\in\Lambda} {\rm
Pexp}\left(\int_s^{t_{i+1}}\LLL^i_m( v) dv\right)(y_1, y_3)
U^i_m(y_3, y_2).
\end{equation}
More general Markov propagators are obtained by taking products of the ones above.
\end{definition}

\begin{definition} {\bf (Attainable Sets.)}
Let $U_m$ be a piecewise smooth Markov process, $m\ge m_0$, $y\in
A_m$ and $t\in [T, T']$. The {\it attainable set} $\DDD_m(U, t, y)
\subset A_m$ is defined as follows: if $t\in (t_{i}, t_{i+1})$ for
some $i=0,..n-1$, then $\DDD_m(U, t, y)$ is the set of $\bar y
\in\Lambda$ such that $\LLL_i(y, \bar y; t)> 0$. If instead $t =
t_i$ for some $i=1,..n-1$, then $\DDD_m(U, t_i, y)$ is defined as
the set of the $\bar y \in\Lambda$ such that $U_i(y, \bar y; t)> 0$.
\end{definition}

\begin{definition} {\bf (Equivalent Markov Processes.)}
Two piecewise smooth Markov propagators $U$ and $U'$ are called
equivalent if their attainable sets $\DDD_m(U, t, y)$ and $\DDD_m(U', t, y)$ are
equal for all $t\in[T, T']$ and all $y\in\Lambda$. If $\DDD_m(U, t, y)$ is a
subset of $\DDD_m(U', t, y)$ for all $t\in[T, T']$ and all
$y\in\Lambda$, then one says that $U_m$ is absolutely continuous with
respect to $U_m'$.
\end{definition}

\begin{definition} \label{def_measure_change} {\bf (Measure Changes.)}
Let $U_m, m\ge m_0$
be a family of piecewise smooth Markov propagators. A measure change
is characterized by a family of positive, non-zero functions $G_m^{yt}(y')\ge 0$
indexed by $y\in A_m$ and $t\in[T, T']$,
which is strictly positive for all $y'\in\DDD_m(U, y, t)$ and zero otherwise. A measure
change function defines a transformation of a Markov generator into
an equivalent one according to the following formula:
\begin{equation}
\LLL'(y,y'; t) = {1\over G_m^{yt}(y)} \LLL(y,y';t) G_m^{yt}(y') -
{1\over G_m^{yt}(y)} (\LLL(t) G_m^{yt}) (y)\delta_{y y'}.
\end{equation}
\end{definition}

Notice that the specification of the function $G_m^{yt}(y')$ at the
point $y' = y$ is immaterial in the sense that it does not affect
the measure change transformation.

\begin{definition}{\bf (Time Changes.)}
A measure change is called time change if there is a function $\phi(y, t)$
such that $G_m^{yt}(y') = \phi(y, t)$
for all $y\in A_m$, $y'\in\DDD_m(U, y, t)$ and $t\in[T, T']$.
The time change is called {\it state
independent} if $\phi(y, t) \equiv \phi(t)$ is a function of the
time coordinate only.
\end{definition}

\begin{theorem} {\bf (Deterministic Time Changes.)} If $\phi(t)$ defines a state independent time change
corresponding to the measure change function $G_m^{yt}(y') =
\phi(t)$, then
\begin{equation}
U_m'(y, t; y', t') = U_m(y, \lambda(t); y', \lambda(t'))
\end{equation}
where
\begin{equation}
\lambda(t) = \int^t_T \phi(s) ds.
\end{equation}
\end{theorem}

The following is a particularly interesting special case of measure change:
\begin{definition}{\bf (Numeraire Changes.)}
Consider a smooth (as opposed to just piece-wise smooth) Markov process
 of generator $\LLL_m(y; t)$. Let $g_m(y'; t)$ be a function satisfying the equation
\begin{equation}
{\partial g_m(y; t)\over\partial t} + (\LLL_m g_m)(y; t) = 0.
\label{eq_numeraire}
\end{equation}
The measure change given by the function $G_m^{yt}(y') = g_m(y'; t)$
is called {\it numeraire change} and the Markov generator transforms
as follows:
\begin{equation}
\LLL_m'(y,y'; t) = {g_m(y'; t)\over g_m(y; t)} \LLL_m(y,y';t)  +
{1\over g_m(y; t)} {\partial g_m(y; t) \over \partial t} \delta_{y
y'}. \label{eq_numchange}
\end{equation}
\end{definition}

Notice that if $g_m(t)$ denotes the multiplication operator of
kernel $g_m(y; t)\delta_{y y'}$, then equation (\ref{eq_numchange})
can be written more compactly as follows:
\begin{equation}
\LLL_m'(t) = {1\over g_m(t)} \LLL_m(t) g_m(t)  + {1\over g_m(t)}
{\partial g_m(t) \over \partial t}.
\end{equation}

\begin{theorem} {\bf (Numeraire Changes.)} If $g_m$ satisfies equation (\ref{eq_numeraire}) and
defines a numeraire change, then
\begin{equation}
U_m'(y, t; y', t') = {g_m(y', t')\over g_m(y, t)} U_m(y, t; y', t').
\end{equation}
\end{theorem}

Let $F_t = F(y_\cdot, t)$ be an adapted process in the time interval
$[T, T'].$ Let us consider a fixed time $t\in [T, T')$. Recall that, since $F_t$ is
non-anticipatory, we have that $F(y_\cdot, t) = F(y_\cdot', t)$ for
all pairs of paths such that $y_s = y_s'$ for all $s \leq t$. If
$\bar y\in \DDD_m(U', t, y)$, let $\tilde y_\cdot = {\rm Ext_t}(y_\cdot, y_t)$
be the constant extension path such that $\tilde y_s = y_s$ for all $s<t$ and
$\tilde y_s = y_t$ for all $s\ge t$. With probability one, we have
that
\begin{equation}
\lim_{\delta t\downarrow 0}  F(y_\cdot, t+\delta t) = F({\rm Ext_t}(y_\cdot,
\bar y), t)
\end{equation}
for some $\bar y\in \DDD_m(U, t, y)$.

\begin{definition} {\bf (Monotonic Processes.)}
Let $U$ be a piecewise smooth Markov
propagator and let $F_t$ be an adapted process given by the
non-anticipatory path functional $F(y_\cdot, t)$. $F_t$ is said to
be {\it  increasing at time $t$} if
\begin{itemize}
\item[(i)] For all $\bar y \in \DDD_m(U, t, y)$ we have that
\begin{equation}
F({\rm Ext_t}(y_\cdot, \bar y), t) - F(y, t) \ge 0.
\end{equation}
\item[(ii)] We have that
\begin{equation}
{\partial F({\rm Ext_s}(y_\cdot, y_{t-0}), s) \over \partial s}
\bigg\lvert_{s=t} \ge 0
\end{equation}
where $y_{t-0} = \lim_{\delta t\downarrow 0} y_{t-\delta t}$.
\item[(iii)] For all $y\in A_m$ and all $t\in[T, T']$, either
the inequality in (i) holds in a strict sense for at least one
$\bar y\in \DDD_m(U, t, y)$ or the inequality in (ii) holds in a
strict sense.
\end{itemize}
In case the property (iii) fails but the other two still hold, the
process is called {\it non decreasing}.
\end{definition}

\begin{theorem} {\bf (Monotonic Processes.)}
Let $U_m$ and $U_m'$ be two piecewise smooth Markov propagators and
let $F(y_\cdot, t)$ be a non-anticipatory path functional. If $U_m$
and $U_m'$ are equivalent and if $F(y_\cdot, t)$ regarded as an
adapted process under the path measure generated by $U_m$ is
increasing (non-decreasing), then also $F(y_\cdot, t)$ regarded as
an adapted process under $U_m'$ is increasing (non-decreasing).
\end{theorem}

\begin{proof}
This theorem descends from the fact that the definition
of monotonicity depends only on the attainable sets.
\end{proof}

\begin{definition} {\bf (Martingale Processes.)}
A process $F_t$ is a martingale if for all times $t$ we have that
\begin{equation}
{\partial \over \partial t} \tilde F({\rm Ext_t}(y_\cdot, y_{t-0}),
t) + \sum_{\bar y \in \DDD(U, t, y_t)} \LLL(y_t, \bar y, t) \big (
F({\rm Ext_t}(y_\cdot, \bar y), t) - F({\rm Ext_t}(y_\cdot,
y_{t-0}), t) \big ) = 0 \label{eq_martingale}
\end{equation}
The process $F_t$ is called an {\it equivalent martingale} if there
is a second piecewise smooth Markov propagator $U'$ for which the
non-anticipatory path functional $ F(y_\cdot, t)$ is a martingale
process.
\end{definition}

\begin{theorem}{\bf (Equivalent Martingales.)}
\begin{itemize}
\item[(i)] If $F_t$ is an increasing adapted process then it
is not an equivalent martingale. Otherwise stated, if $F_t$ is an
equivalent martingale then it is not  increasing.
\item[(ii)] If $F_t$ is a non-decreasing adapted process and it
is also an equivalent martingale, then it is a constant process with
$ F(y_\cdot, t) = {\rm const}$ for all $t\in[T, T']$.
\end{itemize}
\end{theorem}

\begin{proof} This theorem descends from the fact that
monotonicity properties are preserved by measure changes.
\end{proof}

Martingales are particularly useful as they can be constructed by
taking expectations.

Let $\Phi(y_\cdot)$ be a continuous path-functional. From the
modeling viewpoint, such a functional can represent future cash flow
streams. For instance, one choice could be
\begin{equation}
\Phi(y_\cdot) = \Phi_0(y_{t'})
\end{equation}
where $\Phi_0$ is a continuous univariate function and $t'\in[T, T']$ is fixed.
In a more general example, one may consider a path functional of the form
\begin{equation}
\Phi(y_\cdot) \equiv \int_{T}^{T'} d s_1 ... \int_{s_{n-1}}^{T'} d
s_n a(t; s_1, .... s_n) F_1(y_{s_1}, s_1) ...  F_n(y_{s_n}, s_n)
\label{eq_defphi2}
\end{equation}
with $a(t; s_1, .... s_n)=0$ for $t\in[T, t']$. The path conditioned
expectation of $\Phi(y_\cdot)$ is the non-anticipatory path
functional $F_t = F(y_\cdot, t)$ such that
\begin{equation}
F(y_\cdot, t) = \int_{A(y_\cdot, t)} \Phi(z_\cdot) \mu[d z_\cdot]
\end{equation}
where the integral is restricted to the set $A(y_\cdot, t)
=\{z_\cdot \lvert z_s = y_s \forall s\leq t\}$. The intersection of
the set $A(y_\cdot, t)$ with each of the spaces $\h^q_{mn}$ is a
compact, finite dimensional submanifold with boundaries. To denote
path conditioning, we also use the following notation:
\begin{equation}
F_t = F(y_\cdot, t) = E[\Phi(y_\cdot) \lvert y_s, s\leq t ] =
E_t[\Phi(y_\cdot) ]. \label{eq_fmartingale}
\end{equation}
\begin{proposition}
The process $F_t$ in (\ref{eq_fmartingale}) is a martingale for
$t<t_1$.
\end{proposition}
\begin{proof}
Equation (\ref{eq_martingale}) descends from the backward equation
in (\ref{backward_eq}).
\end{proof}

\section{The Fundamental Theorem of Finance}

In this section, we derive the Fundamental Theorem of Finance
in a context general enough to encompass most cases of practical relevance.

\begin{definition} {\bf (Financial Model.)}
Let $A_m, m\ge m_0,$ be a simplicial sequence. A financial model is
given by a family $A_m^{k}(y_\cdot, t), k = 1, ... n$ of adapted
processes modeling asset prices and a non-decreasing adapted process
$B_m(y_\cdot, t)$ modeling the money-market account. For notational
convenience, we set $A_m^{0}(y_\cdot, t) = B_m(y_\cdot, t)$. Let us
introduce also the discounted asset price process defined as
follows:
\begin{equation}
\tilde A_m^{k}(y_\cdot, t) = {A_m^{k}(y_\cdot, t) \over B_m(y_\cdot,
t)}.
\end{equation}
\end{definition}

\begin{definition} {\bf (Trading Strategies.)} Given
a financial model with $n$ assets, a trading
strategy is given by a family of adapted processes
$\zeta_m^{k}(y_\cdot, t), k=0, .... n$. The {\it value
process} of a strategy is the adapted process
\begin{equation}
\Pi_m(y_\cdot, t) = \sum_{k=0}^n \zeta_m^{k}(y_\cdot, t)
A_m^{k}(y_\cdot, t)
\end{equation}
The {\it discounted value process} instead is given by
\begin{equation}
\tilde \Pi_m(y_\cdot, t) = \sum_{k=0}^n \zeta_m^{k}(y_\cdot, t) \tilde
A_m^{k}(y_\cdot, t)
\end{equation}
\end{definition}

\begin{definition}{\bf (Self-Financing Condition.)}
An adapted trading strategy is called self-financing if the following
two conditions hold:
\begin{equation}
(SF1) \hskip1cm \sum_{k=0}^n   {\partial \zeta^k_m \over
\partial t} ({\rm Ext_t}(y_\cdot, y_{t-0}), t)
\tilde A_m^{k}({\rm Ext_t}(y_\cdot, y_{t-0}), t) = 0
\end{equation}
and if, for all $\bar y\in \DDD_m(U, t, y_t)$, we have
that
\begin{equation}
(SF2) \hskip1cm \sum_{k=0}^n \big(\zeta^{k}_m({\rm Ext_t}(y_\cdot, \bar y), t)
- \zeta^{k}_m({\rm Ext_t}(y_\cdot, y_{t-0}), t)\big)
\tilde A_m^{k}({\rm Ext_t}(y_\cdot, \bar y), t) = 0 .
\end{equation}
\end{definition}

\begin{proposition}
If $\zeta^{k}_m(y_\cdot,  t)$ is a self-financing trading strategy, then
the corresponding discounted value process satisfies
\begin{equation}
{\partial \tilde \Pi_m({\rm Ext_t}(y_\cdot, y_{t-0}), t) \over \partial t} =
\sum_{k=0}^n  \zeta^k_m({\rm Ext_t}(y_\cdot, y_{t-0}), t)
 {\partial \tilde A_m^{k}({\rm Ext_t}(y_\cdot, y_{t-0}), t)\over \partial t}
\end{equation}
and for all $\bar y\in \DDD_m(U, t, y_t)$, we have that
\begin{align}
&\tilde \Pi_m({\rm Ext_t}(y_\cdot, \bar y), t) -
\tilde \Pi_m({\rm Ext_t}(y_\cdot, y_{t-0}), t) \\
&\hskip3cm=
\sum_{k=0}^n \zeta^{k}_m({\rm Ext_t}(y_\cdot, \bar y), t)
\big(\tilde A_m^{k}({\rm Ext_t}(y_\cdot, \bar y), t)
- \tilde A_m^{k}({\rm Ext_t}(y_\cdot, y_{t-0}), t)\big).
\end{align}
\end{proposition}

\begin{definition}{\bf (Arbitrage Strategies.)}
A self-financing strategy is called {\it arbitrage} at time $t$ if the
corresponding discounted value process $\tilde \Pi_m(y_\cdot, t)$
is  increasing at time $t$.
\end{definition}

\begin{theorem}\label{theofunda} {\bf (Fundamental Theorem of Finance.)}
If there is an equivalent measure with respect to which all discounted base asset
price processes are martingales, then
\begin{itemize}
\item[(i)] The discounted value process of any self-financing trading
strategy under the same equivalent measure is a martingale.
\item[(ii)] There is no arbitrage.
\end{itemize}
Conversely, if there is no arbitrage than there exists a measure change
under which all discounted asset price processes become martingales.
\end{theorem}

\begin{proof} The first part of the theorem is a simple consequence of the definitions
put forward. The converse instead requires a proof.

Let ${\mathcal V}_m(U, t, y_t)$ be the vector space spanned by
functions of the form $v(\bar y)$ where $\bar y\in\DDD_m(U, t,
y_t)$. Let $K$ be the cone in ${\mathcal V}_m(U, t, y_t)$ made up by
all the vectors with non-negative components. Also, let us introduce
the following vector $\xi^k_m\in {\mathcal V}_m(U, t, y_t)$:
\begin{equation}
\xi^k_m(\bar y) =
\delta_{\bar y, y_{t-0}} {\partial \tilde A_m^{k}({\rm Ext_t}(y_\cdot, y_{t-0}), t) \over \partial t}
+
\big(1-\delta_{\bar y, y_{t-0}}\big)
\big(\tilde A_m^k({\rm Ext_t}(y_{\cdot}, \bar y), t) - A_m^k({\rm Ext_t}(y_\cdot, y_{t-0}), t) \big)  .
\end{equation}
Notice that, in case $k=0$ and the asset is the money market account, then
$\xi^0_m(\bar y) = 0$ for all $\bar y\in\DDD_m(U, t, y_t)$.

Suppose there is no arbitrage and fix a time $t$. Assuming
there does not exist a trading strategy with a strictly increasing
value process, for all vectors $(v^k)_{k=1,...n} \in {\mathcal
V}_m(U, t, y_t)$ there are
two elements $y_+, y_- \in \DDD_m(U, t, y_t)$ such that
\begin{equation}
\sum_{k=1}^n v^k \xi^k_m(\bar y_+) > 0
\;\;\;\;{\rm while}\;\;\;\;
\sum_{k=1}^n v^k \xi^k_m(\bar y_-) < 0.
\label{eq_noarb}
\end{equation}

If in equation (\ref{eq_noarb}) we set $v^k = \delta_{k1}$, we
conclude that the vector $\big(\xi^1_m(\bar y)\big)_{\bar y
\in\DDD_m(U, t, y_t)}$ has both positive and negative components.
Hence, the the hyperplane $\Pi^1 \subset {\mathcal V}_m(U, t, y_t)$
orthogonal to the vector $\xi^1_m(\bar y) \in  {\mathcal V}_m(U, t,
y_t)$ intersects $K$.

Let $P_1$ be the orthogonal projection operator onto the hyperplane
$\Pi^1$ and let us consider the vector
\begin{equation}
(P_1 \xi^2_m)(\bar y) = \xi^2_m(\bar y) - {\sum_{\bar z\in\DDD_m(U, t, y_t)}
\xi^2_m(\bar z) \xi^1_m(\bar z) \over \sum_{\bar z\in\DDD_m(U, t, y_t)} \xi^1_m(\bar z)
\xi^1_m(\bar z)} \xi^1_m(\bar y).
\end{equation}
for $i=2,...n$. Due to absence of arbitrage, there are
two elements $y_+, y_- \in \DDD_m(U, t, y_t)$ such that
\begin{equation}
(P_1 \xi^2_m)(\bar y_+) > 0
\;\;\;\;{\rm while}\;\;\;\;
(P_1 \xi^2_m)(\bar y_-) < 0.
\end{equation}
Hence, the vector $P_1 \xi^2_m$ is transversal to the octant $K$ of positive vectors.
As a consequence, the hyperplane $\Pi_2
\subset \Pi_1 {\mathcal V}$ is orthogonal to both vectors $\xi^1_m$ and
$\xi^2_m$, also intersects $K$.

The argument above can be iterated $n$ times, leading to the
conclusion that there exists a strictly positive function
$g^{yt}_m(\bar y)>0, \bar y\in\DDD_m(U, t, y_t)$, such that
\begin{equation}
\sum_{\bar y \in \DDD_m(U, t, y_t)}  \xi_m^k(\bar y) g^{yt}_m(\bar y) = 0
\end{equation}
for all $k=0, ... n$. In particular, this implies that there exists
a measure change function $G^{yt}_m(\bar y)$ such that
\begin{equation}
\sum_{y \in h\ZZZ^d }  \big(\tilde A_m^k(
{\rm Ext_t}(y_{\cdot}, \bar y), t) - A_m^k({\rm Ext_t}(y_\cdot, y_{t-0}), t) \big)
\LLL(y, \bar y; t) G^{yt}_m(\bar y) = 0.
\end{equation}

\end{proof}

\section{Weak Convergence of Markov Generators}

Consider a one dimensional Markov process defined on the simplicial
sequence $A_m, m\ge m_0$ and the time interval $[T, T']$. Many
different specifications of Markov generators on $A_m$ may
correspond to the same limit. In this section we identify a
canonical sequence of generators under a few regularity hypotheses
which imply the existence of a weak limit in distribution sense for
the generator. This is a necessary first step to single out the
general form of an admissible Markov generator. In the following
sections, we then investigate convergence under the much finer
criteria of pointwise convergence for probability kernels and their
derivatives.

First consider the case when the limiting domain $A_\infty$ is bounded. Without
restricting generality, let us suppose that $A_\infty = [-L, L]$.

Let ${\mathcal L}_m$ be a sequence of Markov generators.
The first assumption we make is that the first two moments are finite, or more specifically

\vskip0.5cm
\noindent{\bf Hypothesis MG1. }{ \it The sequences
\begin{equation}
\mu_m(x, t) \equiv \sum_{y\in A_m} {\mathcal L}_m(x, y; t) (y-x),
\hskip1cm
\sigma_m(x, t) \equiv \sqrt{\sum_{y\in A_m} {\mathcal L}_m(x, y; t) (y-x)^2}
\end{equation}
are uniformly bounded in absolute value as $m\to\infty$ and converge to
limits for all $x\in A_m$ and all $t\in [T, T']$, i.e. the following limits exist:
\begin{equation}
\mu(x, t) \equiv \lim_{m\to\infty} \mu_{m}(x, t), \hskip1cm
\sigma(x, t) \equiv \lim_{m\to\infty} \sigma_{m}(x, t).
\end{equation}
}
\vskip0.5cm

Notice that, due to the dominated convergence theorem, the functions $\mu_m(x, t)$
and $\sigma_m(x, t)$ regarded as piecewise constant functions on $\A_\infty$
converge weakly to the corresponding limits.

\vskip0.5cm
\noindent{\bf Hypothesis MG2. }{ \it The following limits exist and are finite for all $m\ge m_0$ and
all pairs $x, y\in A_m$ such that $\lvert x - y \lvert \ge 2 h_m$:
\begin{equation}
\lambda_m(x, y; t) = \lim_{
m'\to\infty} \sum_{
z\in A_{m'} : y-h_m < z \leq y+h_m} {\mathcal L}_{m'}(x, z; t).
\end{equation}
}
\vskip0.5cm

The family of functions $\lambda_m(x, y)$ can be represented in terms of the
so called Levy measures specified as follows:
\begin{theorem} {\bf (Levy measures.)}
For all $x\in A_m$ there exists a measure $\nu_{xt}(d\xi)$ in $[-L, L]$ (called Levy measure)
with the following two properties:
\begin{itemize}
\item[(i)]
\begin{equation}
\int_{-L}^L \nu_{xt}(d\xi) \xi^2 < \infty
\end{equation}
\item[(ii)] For all $y \in A_m$ with $\lvert x - y \lvert \ge 2 h_m$ the
following representation is valid:
\begin{equation}
\lambda_m(x, y) = \lim_{\varepsilon\downarrow 0} \int_{y - h_m/2 +\varepsilon}^{y+h_m/2}
\nu_{xt}(d\xi).
\end{equation}
\end{itemize}
\end{theorem}

\begin{definition}{\bf (Finite activity jumps.)}
A Markov process is said to have finite activity jumps if
for all $m\ge m_0$ and all $x\in A_m$ we have that
\begin{equation}
\int_{-L}^L \nu_{xt}(d\xi) < \infty
\label{cond_finitem1}
\end{equation}
\end{definition}

Notice that given $MG1$ and $MG2$ we have
\begin{equation}
\limsup_{m\to\infty} \sum_{y} \lambda_m(x, y; t)(y-x)^2 \leq
\sum_y {\mathcal L}_m(x, y; t) (y-x)^2 < \infty.
\end{equation}
More generally, if $\phi\in{\mathcal D}(\RRR)$ is a test function
such that $\phi(0) = 0$ and $\phi'(0) = 0$, then we have that
\begin{equation}
\limsup_{m\to\infty} \sum_{y} \lambda_m(x, y; t)\phi(y-x)<\infty.
\end{equation}
Although the above sequence is bounded, the limit may not exist in general.
We thus need to stipulate this as a separate assumption:

\vskip0.5cm \noindent{\bf Hypothesis MG3. }{\it For all test
functions $\phi\in{\mathcal D}(\RRR)$ such that $\phi(0) = \phi'(0)
= 0$ and all $x\in A_m$, the sequence $\sum_{y} \lambda_{m'}(x, y;
t)\phi(y-x)$ admits a limit as $m'\to\infty$. } \vskip0.5cm

Let us introduce the following notations:
\begin{equation}
\mu^0_m(x; t) = \sum_{y} \lambda_m(x, y; t)(y-x),\;\;\;\;\;
\sigma^0_m(x; t) = \sqrt{\sum_{y} \lambda_m(x, y; t)(y-x)^2}.
\label{eq_mu0sigma0}
\end{equation}

Notice that the sequence $\sigma^0_m(x, t)$ is non-negative and
is uniformly bounded as a function of $m$. More precisely
\begin{equation}
0 \leq \sigma^0_m(x, t) \leq \sigma_m(x, t).
\end{equation}
On the other hand, we cannot conclude in general that $\mu^0_m(x,
t)$ is necessarily uniformly bounded. In fact, the difference
$\mu^m(x, t) - \mu^0_m(x, t)$ could diverge as $m\to\infty$ while
being compensated by terms concentrated at $y = x \pm 1$ which in
turn also diverge in the limit as $m\to\infty$ while keeping the
total drift $\mu^m(x, t)$ uniformly bounded. An exception to this
general situation is found in case the following condition is
satisfied:

\vskip0.5cm \noindent{\bf Hypothesis MG4. }{\it  For all $m\ge m_0$
and all $x\in A_m$ we have that
\begin{equation}
\int_{-L}^L \lvert x \lvert \nu_{xt}(d\xi) < \infty
\label{cond_finitem1}
\end{equation}
}
\vskip0.5cm

Two particularly important situations in which this condition $MG4$
holds are given in the following:
\begin{theorem} {\bf (MG4.)} Assuming $MG1, MG2, MG3$, if the Markov process
is either monotonic or has finite activity jumps, then $MG4$ holds.
\end{theorem}

Under the three assumptions $MG1, MG2, MG3$,
the sequence ${\mathcal L}_m$ of Markov generators
can be mapped into an equivalent canonical sequence of Markov generators ${\mathcal L}^C$.
More precisely, we set
\begin{equation}
{\mathcal L}_m^C(x, y; t) = \lambda(x, y; t)
\end{equation}
in case $\lvert x - y \lvert \ge 2 h_m$. Furthermore,
we set
\begin{equation}
{\mathcal L}_m^C(x, x\pm h_m; t) = \tilde \mu_m(x;t) \nabla_{h_m}(x, x\pm 1)
+{\tilde \sigma_m(x; t)^2 \over 2} \Delta_{h_m}(x, x\pm1)
\end{equation}
where
\begin{equation} \tilde \mu_m(x;t) = \mu_m(x) - \mu^0_m(x), \;\;\;\;\;
\tilde \sigma_m(x; t)^2 = \sigma_m(x; t)^2 - \sigma_m^0(x; t)^2.
\end{equation}
and the operators $\nabla_{h}$ and $\Delta_{h}$ are defined as in
equation (\ref{eq_nabladelta}). Finally,
\begin{equation}
{\mathcal L}_m^C(x, x; t) = \sum_{y\in A_m, y\neq x} {\mathcal L}_m^C(x, y; t)
\end{equation}

\begin{theorem} {\bf (Canonical Representations of Markov Generators.)}
For all smooth functions  of compact support $\phi \in {\mathcal
D}(A_{\infty})$ and all $x\in A_m$ we have that
\begin{equation}
\lim_{
m'\to \infty } \sum_{y\in A_{m'}} {\mathcal L}^C_{m'}(x, y) \phi(y)
=
\lim_{
m'\to \infty } \sum_{y\in A_{m'}} {\mathcal L}_{m'}(x, y) \phi(y).
\end{equation}
\end{theorem}

A compact representation for a canonical generator summarizing these
definition is obtained by considering the symbol as specified in
Definition (\ref{def_symbol}).

\begin{theorem}{\label{LKrep}}
{\bf (Levy-Khintchine representation.)} Under the hypothesis $MG1,
MG2, MG3$ above and in case $A_\infty = [-L, L]$ is bounded, the
symbol of a generator in canonical form can be expressed as follows:
\begin{equation}
\hat {\mathcal L}^C_m(x, p; t) =  i \mu(x; t) {\sin ph\over h} +
\tilde \sigma(x; t)^2 {\cos ph -1\over h^2} + \sum_{y\in A_m}
\left(e^{i p (y-x)} - 1 - i {\sin p h \over h} (y-x) \right)
\lambda_m(x, y; t).
\end{equation}
The following limit converges in the weak topology:
\begin{equation}
\lim_{m\to\infty} \hat {\mathcal L}^C_m(x, p; t) =
i \mu(x; t) p - {\tilde\sigma^2(x; t)\over 2} p^2 + \int_{-L}^L
 (e^{i p\xi} - 1 - ip\xi ) \nu_{xt}(d\xi),
\end{equation}
where $\nu_{xt}(d\xi)$ is the Levy measure. Moreover, if also
$MG4$ is satisfied, then
the symbol of a generator in canonical form can be expressed as follows:
\begin{equation}
\hat {\mathcal L}^C_m(x, p; t) =  i \tilde \mu(x; t) {\sin ph\over h} +
\tilde \sigma(x; t)^2 {\cos ph -1\over h^2} + \sum_{y\in A_m}
\left(e^{i p (y-x)} - 1\right) \lambda_m(x, y; t).
\end{equation}
The following limit converges in the weak topology:
\begin{equation}
\lim_{m\to\infty} \hat {\mathcal L}^C_m(x, p; t) =
i \tilde \mu(x; t) p - {\tilde\sigma^2(x; t)\over 2} p^2 + \int_{-L}^L
 (e^{i p\xi} - 1) \nu_{xt}(d\xi).
\end{equation}
\end{theorem}

\section{Fast Exponentiation and Spectral Methods}

Numerical analysis of pricing models in the operator formalism
depend on the ability to compute the propagator for a given
generator $\LLL(t)$. Time homogenous Markov generators $\LLL_m(y_1,
y_2)$ represent a privileged special case of particular importance.
In this case, the associated propagator solving the differential
equation in (\ref{eq_gen}) is given by the matrix exponential
\begin{equation}
U_m(y_1, s; y_2, s') = \exp((s'-s) \LLL_m) (y_1, y_2).
\end{equation}
Matrix exponentiation can be defined in several equivalent ways,
such as by Taylor expansion
\begin{equation}
\exp( t \LLL_m) = \sum_{j=0}^\infty {t^j\over j!} \LLL_m^j
\label{expdef1}
\end{equation}
or by means of Neper's formula
\begin{equation}
\exp( t \LLL_m) = \lim_{N\to\infty} \left(1+{t\over N}
\LLL_m\right)^{N} \label{expdef2}
\end{equation}

We find that the most efficient and robust method for exponentiating
Markov generators is the so called {\it fast exponentiation}
algorithm.
 Let us fix a
Markov generator $\L(x, y)$ and a time horizon $t$. Let $\delta t>0$ be the largest time
interval for which both of the following properties are satisfied:
\begin{align*}
&({\rm FE1}) \;\;\;\;\; \min_{y\in\Lambda}(1 + \delta t \L(y,y)) \ge {1/2} \\
&({\rm FE1}) \;\;\;\;\; \log_2 {t\over\delta t} = n \in\NNN.
\end{align*}
To compute $e^{t\LLL}(x, y)$, we first define the elementary
propagator
\begin{equation}
u_{\delta t} (x, y) = \delta_{xy} + \delta t \L(x,y)
\end{equation}
and then evaluate in sequence $u_{2\delta t} = u_{\delta t}\cdot
u_{\delta t}$, $u_{4\delta t} = u_{2\delta t}\cdot u_{2\delta t}$,
... $u_{2^n\delta t} = u_{2^{n-1}\delta t} \cdot u_{2^{n-1}\delta
t}$.

As we show in the next few sections, this algorithm approximates
probability kernels with errors with respect to the continuous time
kernel density which, in fairly general cases, are of order
$O(h_m^2\lvert \log h_m\lvert^3)$. This is the same order by which
the continuous time kernel density differs from its continuous limit
also according to estimates below. In the case of Brownian motion, a
sharp convergence estimates of ordered $O(h_m^2)$ can also be
proven.

Matrix multiplication is accomplished numerically by invoking the
routine ${\tt dgemm}$ in Level-3 BLAS. Very efficient, processor
specific version of ${\tt dgemm}$ are now available along with
implementations on massively parallel GPU chipsets. It turns out
that the standard measure of algorithmic complexity as the number of
floating point operations required to accomplish a certain task, is
not simply proportional and scales non-linearly with respect to
execution time. Using blocking and cache optimizations and
distributing the load across many cores, execution time for medium
to large matrices appears to scales much better than the naive $n^3$
scaling one would obtain by triple looping \cite{GotoGeijn}.

A more general method that allows to compute not only exponentials
but also other functions of a Markov generator is full
diagonalization. Unfortunately, unless the Markov generator is
symmetrizable, diagonalization algorithms can possibly run into
serious instabilities within double precision arithmetics because of
the phenomenon of pseudo-spectrum \cite{TE2006}. This makes it
impossible to diagonalize exactly within the limits of
double-precision arithmetics and forces one to resort to expedients
such as for instance spectral truncations. Since the fast
exponentiation method has much better scaling properties than full
diagonalization and is entirely stable, we recommend that it be used
in all situations where a matrix exponentiation is required.
However, it often arises the necessity to define and compute also
other functions of a given time independent Markov generator and for
this purpose diagonalization can be usefully employed, especially
when the target matrix is symmetrizable.

Not all Markov generators are diagonalizable but most are and it is
safe to take diagonalizability for granted in numerical
applications. To make this statement more precise, one sets out the
following definitions:

\begin{definition} {\bf (Generic Properties.)}
A dense $G_\delta$ of a topological space is
a countable intersection of dense open subsets. If a property is
valid on a dense $G_\delta$ one says that it is valid {\it
generically}, or that it is generic. Instead, if the same
topological set is endowed of a measure and a property is valid on a
full measure set of parameters, one says that it is valid almost
surely with respect to that particular measure.
\end{definition}

We have that \noindent \begin{proposition} Markov generators are
diagonalizable both generically and almost surely.
\end{proposition}

Let $u_n(x)$ and $\lambda_n$, $n = 1,..N$, be eigenfunctions and
eigenvalues of the Markov generator $\LLL$, i.e.
\begin{equation}
\LLL u_n = \lambda_n u_n.
\end{equation}
Let $U$ be the matrix whose columns are given by the vectors
$u_n(x)$ and let $\Lambda$ be the diagonal matrix with the
eigenvalues $\lambda_n$ on the diagonal. Hence
\begin{equation}
\LLL U = U \Lambda.
\end{equation}
We have that $\LLL = U \Lambda U^{-1}$ and $e^{t \LLL} = U e^{t
\Lambda} U^{-1}$. This equation expressed in components reads as
follows:
\begin{equation}
e^{t \LLL}(x, x') = \sum_{n=1}^N e^{\lambda_n t} u_n(x) v_n(y)
\label{eqeigexp}
\end{equation}
where $v_n(y)$ is the $n-$th row vector of the inverse matrix $V =
U^{-1}$.

One may extend the above definition to other functions of a Markov
generator. If $\psi(\lambda)$ is a function, one may define
$\psi(\LLL)$ as the operator whose matrix is given by
\begin{equation}
\psi(\LLL)(x, x') = \sum_{n=1}^N \psi(\lambda_n) u_n(x) v_n(y)
\end{equation}

An important example of functional calculus is found to express the
Markov generator of processes obtained  by stochastic time change,
whereby the time-change process has independent, uniformly
distributed increments. Processes in this class are called {\it
Bochner subordinators}. Because of time and space homogeneity,
Bochner subordinators can be constructed starting from a process on
simplicial sequence $h_m\ZZZ$ and are characterized by a Markov
generator of the form $\LLL_m(x,y) = \lll_{h_m}(y-x)$.

\begin{theorem} {\bf (Bochner Subordinators.)}
If the limit $\lim_{h\downarrow0} \lll_h = \lll$ exists in weak sense on
$\DDD'(\RRR)$ then the limit kernel has the following form
\begin{equation}
\int \lll(x) \phi(x) dx= \mu \phi'(0) + \int_0^\infty (\phi(x) -
\phi(0)) \nu(dx)
\end{equation}
where $\nu(dx)$ is a positive measure supported on $\RRR_+$ and is
such that
\begin{equation}
\int_0^\infty x \nu(dx) < \infty.
\end{equation}
\end{theorem}

The {\it characteristic function} of the process defined as the
Fourier transform of the Markov generator is thus
\begin{equation}
\epsilon(k) = \lim_{h\downarrow0} \sum_x e^{ikx} \lll_h(x) = i \mu k
+ \int_0^\infty (e^{ikx} - 1) \nu(dx).
\end{equation}
The Laplace transform instead is given by
\begin{equation}
\phi(\lambda) \equiv - \epsilon(i\lambda) = - \lim_{h\downarrow0}
\sum_x e^{-\lambda x} \lll_h(x) = \mu \lambda + \int_0^\infty (1 -
e^{-\lambda x}) \nu(dx).
\end{equation}
A function of this form is called {\it Bernstein function}.

Bernstein functions may be used to express Laplace transform of the
transition probability kernel of time-homogeneous monotonic
processes as follows:
\begin{equation}
\int_0^\infty e^{t \LLL} (0, x)  e^{ -\lambda x} dx =
e^{-t\phi(\lambda)}.
\end{equation}
In turn, we may also express the transition probability kernel in
terms of the characteristic function, i.e.
\begin{equation}
e^{t \LLL} (0, x) = \int_{-\infty}^\infty e^{i t \epsilon(k) + ikx}
{dk\over 2\pi}.
\end{equation}
Notice that since $\phi(-ik) = - \epsilon(k)$, the last formula may
also be interpreted as the Fourier-Mellin inversion of the previous
one.

\begin{example} {\bf (Poisson Process)} \rm The Poisson process corresponds to
\begin{equation}
\phi_P(\lambda; c) = c(1-e^{-\lambda}) = c \int_0^\infty
(1-e^{-\lambda t}) \delta(t-1) dt.
\end{equation}
\end{example}

\begin{example} {\bf (Stable Process)} \rm The stable subordinator with index
$\alpha\in(0,1)$ is given by
\begin{equation}
\phi_S(\lambda; \alpha) = \lambda^\alpha =
{\alpha\over\Gamma(1-\alpha)} \int_0^\infty (1-e^{-\lambda t})
t^{-1-\alpha} dt.
\end{equation}
\end{example}

\begin{example} {\bf (Gamma Process)} \rm The Gamma subordinator with variance rate
$\nu
>0$ is given by
\begin{equation}
\phi_{VG}(\lambda; \nu) = {1\over\nu} \log(1+\nu \lambda) =
{1\over\nu} \int_0^\infty (1-e^{-\lambda t}) t^{-1} e^{- t / \nu}
dt.
\end{equation}
\end{example}

To add jumps to a diffusion process one can use the method of
independent stochastic subordination. Namely let $F_t$ be the
diffusion process with drift function $\mu(F)$ and volatility
function $\sigma(F)$ and let $\L^d(y_1, y_2; \mu, \sigma)$ be the
corresponding Markov generator. Consider the generator
\begin{equation}
\L_{VG}(\mu, \sigma, \nu)  = -\phi_{VG}(-\L^d(\mu, \sigma); \nu).
\end{equation}
The propagator for $\L^j(\mu, \sigma, \nu)$ satisfies the following
equation:
\begin{equation}
\exp\big(t \L_{VG}(\mu, \sigma, nu)\big)(y_1, y_2) =
E_0\big[\exp\big(T_t^\nu \L^d(\mu, \sigma)\big)(y_1, y_2)\big] =
\sum_{n=1}^N e^{-\phi_{VG}(-\lambda_n; \nu)  t} u_n(x) v_n(y)
\label{eq_propjumps}
\end{equation}
where $T_t^\nu$ are the paths of the monotonic process in Example 3,
the $u_n(x)$ are the eigenfunctions of the diffusion generator
$\L^d(\mu, \sigma)$, $\lambda_n$ are the corresponding eigenvalues
and the functions $v_n(x)$ are defined as in (\ref{eqeigexp}).
Hence, the generator $\L_{VG}(\mu, \sigma, \nu)$ identifies the
time-changed process with paths $F_{T_t^\nu}$.

To model asymmetric jumps one can follow several strategies. A
simple one is to specify the two different variance rates $\nu_+$
and $\nu_-$ for the up and down jumps and compute separately two
Markov generators
\begin{eqnarray}
\L_{VG}(\mu, \sigma, \nu_\pm) = -\phi_{VG}(-\L^d(\mu, \sigma);
\nu_\pm).
\end{eqnarray}
The new generator for our process with asymmetric jumps is obtained
by combining the two generators above
\begin{equation}
\L_{VG}(y_1, y_2; \mu, \sigma, \nu_+, \nu_- ) =
\begin{cases}
\L_{VG}(y_1, y_2; \mu, \sigma, \nu_- )\;\;\;\;\;{\text {if}}\;\; F(y_1) < F(y_2)\\
\L_{VG}(y_1, y_2; \mu, \sigma, \nu_+ )\;\;\;\;\;{\text {if}}\;\; F(y_1) > F(y_2)\\
- \sum_{y_2 \neq y_1} \L_{VG} (y_1, y_2; \mu, \sigma, \nu_+, \nu_-
).
\end{cases}
\label{eqvgdef}
\end{equation}

The construction can also be localized. If $\phi_x(\lambda)$ is a
family of Bernstein functions indexed by the state variable
$x\in\Lambda_m$, then one can consider the Markov generator of
matrix
\begin{eqnarray}
\tilde \L(x, y) = -\phi_{x}(-\L^d(\mu, \sigma))(x, y).
\end{eqnarray}
This construction allows one to model state dependent jumps.

\section{Construction of Brownian Motion}

This section is based on work in collaboration with Alex Mijatovic,
see \cite{AlbaneseMijatovicBM}.

Let $A = (A_m), m = m_0, m_0+1, ...$ be a simplicial sequence
converging to an interval $\lim_{m\to\infty} A_m = [-L,L] \subset
\RRR$, where $0<L<\infty$. Let $\mu$ and $\sigma$ be two constants.

The generator of a Brownian motion on $A_m$ has the form
\begin{equation}
{\LLL_m} = \mu \nabla_{h_m} + {1\over 2} \sigma^2 \Delta_{h_m}
\label{def_brownian}
\end{equation}
for all interior points $x\in Int(A_m)$. We assume that $m_0$ is
large enough so that
\begin{equation}
{\sigma^2\over 2 h_m^2} > {\lvert \mu \lvert \over 2h_m}
\label{def_brownian_bounds}
\end{equation}
for all $m\ge m_0$.

If $x$ is a boundary point, i.e. $x\in \partial(A_m)$, then several
definitions of the operator ${\LLL_m}$ are possible depending on the
choice of boundary conditions. {\it Absorbing} boundary conditions
correspond to the choice
\begin{equation}
{\LLL_m}(x, y) = 0 \;\;\;\;\forall y\in A_m.
\end{equation}
{\it Reflecting} boundary conditions correspond to
\begin{align}
{\LLL_m}(x, x) &= \mu \nabla_{h_m}(x,x) + {1\over 2} \sigma^2
\Delta_{h_m}(x,x)\\
{\LLL_m}(x, y) &= -{\LLL_m}(x, x)
\end{align}
where $y\in Int(A_m)$ is the closest point to $x$ in the interior of
$A_m$, while ${\LLL_m}(x, y)=0$ for all other points $y$. {\it
Periodic} boundary conditions are implemented by setting
\begin{align}
{\LLL_m}(x, y) &= \mu \nabla_{h_m}(x,y) + {1\over 2} \sigma^2
\Delta_{h_m}(x,y)
\end{align}
if $y = x$ or $y\in Int(A_m)$ is the closest point to $x$ in the
interior of $A_m$, and also
\begin{align}
{\LLL_m}(x, x') &= \mu \nabla_{h_m}(x,y) + {1\over 2} \sigma^2
\Delta_{h_m}(x,y)
\end{align}
where $x'\in \partial(A_m)$ is the boundary point at the opposite
extreme of the simplex $A_m$. Finally, {\it mixed} boundary
conditions can also be defined by taking a convex linear combination
of the generators satisfying one of the three boundary conditions
above.

\begin{theorem} {\bf (Convergence Estimates for Brownian Motion.)} Consider a Brownian motion
defined as the process with generator in (\ref{def_brownian}) where
$\mu$ and $\sigma$ are constants satisfying the bounds in
(\ref{def_brownian_bounds}) for all $m\ge m_0$ and we assume
periodic boundary conditions. Let us consider the kernels
\begin{equation}
U_m(x, 0; y, T) \equiv e^{T {\mathcal L}_m}(x, y)
\end{equation}
and
\begin{equation}
U_m^{\delta t}(x, 0; y, T) \equiv \bigg( 1 + \delta t  {\mathcal
L}_m \bigg)^{T\over \delta t}(x, y).
\end{equation}
If $\delta t$ is such that $0 < \delta t  \leq {1\over 2} {h^2 \over
2 \sigma^2 - h \lvert\mu\lvert }$ and $m_0$ is large enough, then
there is a constant $c>0$ such that for all $m'\ge m \ge m_0$ we
have the following inequalities:
\begin{itemize}
\item[(i)]
\begin{equation}
\lvert h_m^{-1} U_m(x, 0; y, T) - h_{m'}^{-1} U_{m'}(x, 0; y, T)
\lvert \leq c h_m^2 \label{bm_uest}
\end{equation}
\item[(ii)]
\begin{equation}
\lvert h_m^{-1} \nabla_{h_m} U_m(x, 0; y, T) - h_{m'}^{-1}
\nabla_{h_{m'}} U_{m'}(x, 0; y, T) \lvert \leq c h_m^2
\label{bm_dest}
\end{equation}
\item[(iii)]
\begin{equation}
\lvert h_m^{-1} \Delta_{h_m} U_m(x, 0; y, T) - h_{m'}^{-1}
\Delta_{h_{m'}} U_{m'}(x, 0; y, T) \lvert \leq c h_m^2
\label{bm_gest}
\end{equation}
\item[(iv)]
\begin{equation}
\lvert h_m^{-1} U_m(x, 0; y, T) - h_{m}^{-1} U_{m}^{\delta t}(x, 0;
y, T) \lvert \leq c h_m^2 \label{bm_udtest}
\end{equation}
\end{itemize}
\end{theorem}

\begin{proof}
It suffices to show this inequality for $m' = m+1$. Let us assume
for simplicity that $L = 2^{m_0} h_{m_0}$ and $A_m = \{ x = h_m i, i
= 0, .... 2^m\}$ for all $m\ge m_0$. The argument extends to more
general lattice geometry but the consideration of these more general
cases would obscure the simplicity of the proof with needless detail
and will thus be omitted.

Let $B_m$ be the Brillouin zone as defined in equation
(\ref{brillouin}). Let ${\mathcal F}_m: \ell^2(A_m) \to \ell^2(B_m)$
be the Fourier transform operator defined so that:
\begin{equation}
\hat f(p) \equiv {\mathcal F}_m(f)(p) = h_m \sum_{x\in A_m} f(x) e^{
- i p x}
\end{equation}
for all $p\in B_m$. The inverse Fourier transform is given by
\begin{equation}
{\mathcal F}_m^{-1} (\hat f)(x) =  {1 \over 2 L} \sum_{p\in B_m}
\hat f(p) e^{i p x}.
\end{equation}

The Fourier transformed generator is diagonal and is given by the
operator of multiplication by
\begin{equation}
\hat {\ell}_m (p) = {\mathcal F}_m {\LLL_m} {\mathcal F}_m^{-1}(p,
p) = -i \mu {\sin h_m p \over h_m} + {\sigma^2 } {\cos h_m p - 1
\over h_m^2}.
\end{equation}
We have
\begin{equation}
h_m^{-1} U_m(x, 0; y, T) = {1 \over 2 L} \sum_{p\in B_m} e^{T\hat
{\ell_m} (p)} e^{i p (y - x) }.
\end{equation}
Using this Fourier series representation, we find
\begin{align}
& \big\lvert h_m^{-1} U_m(x, 0; y, T) - h_{m+1}^{-1} U_{m+1}(x, 0;
y, T) \big\lvert \notag \\
&\leq {1 \over 2 L} \bigg\lvert \sum_{p\in B_m} \bigg( e^{T {\hat
\ell_m} (p)} - e^{T {\hat \ell_{m+1}} (p)} \bigg) e^{i p (y - x)
}\bigg\lvert + {1 \over 2 L} \bigg\lvert \sum_{p\in B_{m+1}\setminus
B_m } e^{T {\hat \ell_{m+1}} (p)} e^{i p (y - x) }\bigg\lvert.
\notag \\
\end{align}

Let
\begin{equation}
K_m = \sqrt{ \lvert \log h_{m+1} \lvert\over \sigma^2 T}.
\label{bm_km}
\end{equation}
If $h_m$ is small enough, i.e. if $m_0$ is sufficiently large, we
have that
\begin{align}
{1 \over 2 L} \bigg\lvert \sum_{p\in B_{m}, \lvert p \lvert \ge K_m
} e^{T {\hat \ell^{m}} (p)} e^{i p (y - x) }\bigg\lvert \leq {1
\over 2 L} \sum_{p\in B_{m}, \lvert p \lvert \ge K_m } e^{T \Re(
\hat\ell^{m} (p))} \leq c \exp\bigg( T \sigma^2 {\cos h_m K_m -
1\over h_m^2} \bigg) \leq c h_m^2. \notag \\
\end{align}
where $\Re(a)$ denotes the real part of $a\in {\mathbb C}$ and $c$
denotes a generic constant. Similarly
\begin{align}
{1 \over 2 L} \bigg\lvert \sum_{p\in B_{m+1}, \lvert p \lvert \ge
K_{m} } e^{T {\hat \ell_{m+1}} (p)} e^{i p (y - x) }\bigg\lvert \leq
{1 \over 2 L} \sum_{p\in B_{m}, \lvert p \lvert \ge K_m } e^{T \Re(
\hat\ell_{m+1} (p))} \notag \\
\leq c \exp\bigg( T \sigma^2 {\cos h_{m+1} K - 1\over h_{m+1}^2}
\bigg) \leq c h_{m+1}^2 \notag \\
\end{align}
Since
\begin{align}
{1 \over 2}h^2 p^3 - {1\over8} h^4 p^5 \leq {\sin hp \over h} -
{\sin 2hp \over 2h} \leq {1 \over 2} h^2 p^3
\end{align}
and
\begin{align}
-{1 \over 8}h^2 p^4 \leq {\cos hp -1 \over h^2} - {\cos 2hp -1 \over
(2h)^2} \leq - {1\over8} {h^2 p^4} + {1\over 48} h^4 p^6.
\end{align}
we find that if $\lvert p \lvert \leq {\sqrt{2} \over h }$ then
\begin{align}
\lvert \hat {\ell}_m(p) - \hat {\ell}_{m+1}(p) \lvert \leq {\mu\over
4} h^2 \lvert p \lvert^3 + {\sigma^2 \over16} h^2 p^4.
\end{align}
Moreover, since
\begin{align}
- {1 \over 2} p^2 \leq {\cos hp -1 \over h} \leq  - {1 \over 2} p^2
+ {1\over 24} h^2 p^4
\end{align}
we conclude that in case $\lvert p \lvert \leq h^{-1} \sqrt {2 \over
3}$, the following inequality holds:
\begin{align}
{\cos hp -1 \over h} \leq - {1\over4} {p^2}
\end{align}
Hence, if  $m_0$ is large enough, we find
\begin{align}
{1 \over 2 L} \bigg\lvert \sum_{p\in B_m, \lvert p \lvert \leq K}
\bigg( e^{T {\hat \ell_m} (p)} - & e^{T {\hat \ell_{m+1}} (p)}
\bigg) e^{i p (y - x) }\bigg\lvert \notag \\
& \leq {1 \over 2 L}  \sum_{p\in B_m, \lvert p \lvert \leq K} e^{-
{1\over4} {p^2}}
\bigg( e^{{\mu T\over 4} h_m^2 \lvert p \lvert^3 + {\sigma^2 T\over16} h_m^2 p^4 } - 1\bigg) \notag \\
& \leq {1 \over 2 L}  \sum_{p\in B_m, \lvert p \lvert \leq K} e^{-
{1\over4} {p^2}} \bigg( {\mu T\over 4} h_m^2 \lvert p \lvert^3 +
{\sigma^2 T\over16} h_m^2 p^4 \bigg) \leq c h_m^2
\end{align}
for some constant $c>0$ independent of $m$. This concludes the proof
of the bound in (\ref{bm_uest}).

To estimate the sensitivity in (\ref{bm_dest}) notice that
\begin{equation}
h_m^{-1} \nabla U_m(x, 0; y, T) = {1 \over L} \sum_{p\in B_m}
e^{T\hat {\ell_m} (p)} {\sin p h_m \over h_m} e^{i p (y - x) }.
\end{equation}
and
\begin{align}
& \big\lvert h_m^{-1} U_m(x, 0; y, T) - h_{m+1}^{-1} U_{m+1}(x, 0;
y, T) \big\lvert \notag \\
&\leq {1 \over 2 L} \bigg\lvert \sum_{p\in B_m} \bigg( e^{T {\hat
\ell_m} (p)}{\sin p h_m \over h_m} - e^{T {\hat \ell_{m+1}}
(p)}{\sin p h_{m+1} \over h_{m+1}} \bigg) e^{i p (y - x)
}\bigg\lvert \notag \\
& + {1 \over 2 L} \bigg\lvert \sum_{p\in B_{m+1}\setminus B_m } e^{T
{\hat \ell_{m+1}} (p)} e^{i p (y - x) }\bigg\lvert.
\notag \\
\end{align}
Let
\begin{equation}
K_m = 2 \sqrt{ \lvert \log h_{m+1} \lvert\over \sigma^2 T}.
\end{equation}
If $h_m$ is small enough, we have that
\begin{align}
&{1 \over 2 L} \bigg\lvert \sum_{p\in B_{m}, \lvert p \lvert \ge K_m
} e^{T {\hat \ell_{m}} (p)} {\sin p h_m \over h_m} e^{i p (y - x)
}\bigg\lvert \leq {1 \over 2 L} \sum_{p\in B_{m}, \lvert p \lvert
\ge K_m }
e^{T \Re (\hat\ell^{m} (p))} {\sin p h_m \over h_m} \notag \\
& \hskip 4cm\leq c \bigg\lvert {\sin K h_m \over h_m} \bigg\lvert
\exp\bigg( T \sigma^2 {\cos K h_m - 1\over h_m^2} \bigg) \leq c
h_m^2. \notag \\
\end{align}
where $c$ denotes a generic constant. Similarly
\begin{align}
{1 \over 2 L} \bigg\lvert \sum_{p\in B_{m+1}, \lvert p \lvert \ge
K_m } e^{T {\hat \ell_{m+1}} (p)} e^{i p (y - x) }\bigg\lvert \leq c
h^2.
\end{align}

If  $m$ is large enough, we also find
\begin{align}
{1 \over 2 L} \bigg\lvert & \sum_{p\in B_m, \lvert p \lvert \leq
K_m} \bigg( {\sin p h_m \over h_m} e^{T {\hat \ell_m} (p)} - {\sin p
h_{m+1} \over h_{m+1}} e^{T {\hat \ell_{m+1}} (p)} \bigg) e^{i p (y
- x) }\bigg\lvert \notag \\
&\leq {1 \over 2 L}  \sum_{p\in B_m, \lvert p \lvert \leq K_m}
\bigg\lvert {\sin p h_{m} \over h_{m}} \bigg\lvert e^{- {1\over4}
{p^2}} \bigg( e^{{\mu T\over 4} h_m^2 \lvert p \lvert^3 + {\sigma^2
T\over16} h_m^2 p^4 } - 1\bigg) \notag
\\ & + e^{- {1\over4} {p^2}} \bigg\lvert
{\sin p h_{m+1} \over h_{m+1}} - {\sin p h_{m} \over h_{m}}
\bigg\lvert \leq c h_m^2 \notag \\
\end{align}
for some constant $c>0$ independent of $m$. This concludes the proof
of the bound in (\ref{bm_dest}). The bound in  (\ref{bm_gest}) can
be derived in a similar way.

Finally, consider the following Fourier representation for the
discretized kernel
\begin{equation}
h_m^{-1} U_m^{\delta t} (x, 0; y, T) = {1 \over L} \sum_{p\in B_m}
\bigg( 1 + \delta t \hat {\ell_m} (p)\bigg)^{T\over \delta t} e^{i p
(y - x) }.
\end{equation}
Consider the formula
\begin{equation}
\bigg( 1 + \delta t \hat {\ell_m} (p)\bigg)^{T\over \delta t} =
\exp\bigg( T \log \big(1 + \hat {\ell_m} (p)\big) \bigg).
\end{equation}
and let's represent the difference between kernels in
(\ref{bm_udtest}) as follows:
\begin{align}
\lvert h_m^{-1} &U_m(x, 0; y, T) - h_{m}^{-1} U_{m}^{\delta t}(x, 0;
y, T)
\lvert \notag \\
& \leq {1 \over 2 L} \bigg\lvert \sum_{p\in B_m} \bigg( \exp\big(T
{\hat \ell_m} (p)\big) - \exp\bigg( {T\over \delta t} \log \big(1 +
\delta t {\hat\ell}_m (p)\big) \bigg) e^{i p (y - x)
}\bigg\lvert \notag \\
& \leq {1 \over 2 L} \sum_{p\in B_m, p\leq K_m} e^{-{1\over4} p^2}
\bigg\lvert \exp\bigg( {T\over \delta t} \log \big(1 + \delta t
{\hat\ell}_m (p)\big) - T {\hat \ell_m} (p) \bigg) -1 \bigg\lvert
\notag \\
&\hskip1cm {1 \over 2 L} \sum_{p\in B_m, p \ge K_m} \bigg\lvert
\exp\big(T {\hat \ell_m} (p)\big)\bigg\lvert  +{1 \over 2 L}
\sum_{p\in B_m, p \ge K_m} \bigg\lvert \exp\bigg({T\over \delta t}
\log \big(1 + \delta t {\hat\ell}_m (p)\big) \bigg) \bigg\lvert
\notag \\
\end{align}
where $K_m$ is chosen as in (\ref{bm_km}). The very same bounds
above lead to the conclusion that this difference is $\leq c h_m^2$.

\end{proof}

\section{Kernel Convergence Estimates for Diffusion Processes}

Diffusions are a particularly important class of Markov processes
which generalize Brownian motions to allow for space dependent
drifts and volatility. In this section we find kernel convergence
estimates for the one-dimensional case following
\cite{AlbaneseKernelEsimtatesA}.

Let $A = (A_m), m = m_0, m_0+1, ...$ be a simplicial sequence
converging to an interval $\lim_{m\to\infty} A_m = [-L,L] \subset
\RRR$, where $0<L<\infty$. Let $\mu(x)$ and $\sigma(x)$ be smooth
functions defined in a neighborhood of $[-L,L]$. The generator of a
diffusion on $A_m$ has the form
\begin{equation}
{\LLL_m} = \mu(x) \nabla_{h_m} + {1\over 2} \sigma(x)^2
\Delta_{h_m}. \label{def_diffusion}
\end{equation}
We assume that $m_0$ is large enough so that
\begin{equation}
{\sigma^2(x)\over 2 h_m^2} > {\lvert \mu(x) \lvert \over 2h_m}
\label{def_diffusion_bounds}
\end{equation}
for all $m\ge m_0$ and all $x\in A_m$. The definition of the
generators at boundary points can be extended by imposing one of the
boundary conditions in the previous section, i.e. reflecting,
absorbing, periodic or mixed.

\begin{theorem}\label{theo_diffker} {\bf (Convergence Estimates for Diffusions.)}
Consider a diffusion process
defined as in (\ref{def_diffusion}) where $\mu(y)$ and $\sigma(y)$
are smooth functions satisfying the bounds in
(\ref{def_diffusion_bounds}) for all $m\ge m_0$. Assume that
boundary conditions are either periodic or absorbing. Then there is
a constant $c>0$ such that
\begin{equation}
\lvert h_m^{-1} U_m(x, 0; y, T) - h_{m'}^{-1} U_{m'}(x, 0; y, T)
 \lvert \leq c h_m^2
\end{equation}
for all $m'\ge m$  and all $y\in A_m$.
\end{theorem}

It suffices to establish the above inequality in the case $m' =
m+1$. In fact, given this particular case, the general statement can
be derived with an iterative argument. Let $h = h_{m+1}$ so that
$h_m = 2 h$.

Recall that the path integral representation for a diffusion process
has the form in equation (\ref{eq_udiff1}). In the special case of a
time-homogeneous process, this expansion reads as follows:
\begin{align}
U_m(x, 0; y, T) =& \sum_{q=1}^\infty 2 ^{-q} \sum_{
\begin{matrix}
\gamma\in \Gamma_m : \gamma_0 = x, \gamma_q = y \\
\lvert\gamma_j - \gamma_{j-1}\lvert = 1 \;\;\forall j\ge1
\end{matrix}}
 W_m(\gamma, q, T)
 \label{eq_udiff1}
 \end{align}
where
\begin{align}
W_m(&\gamma, q, T) = \int_0^{T} ds_1 \int_{s_1}^{T} ds_2 ...
\int_{s_{q-1}}^{T} d s_{q} e^{ (T - s_q) {\mathcal L}_m(y, y)}
\prod_{j=0}^{q-1}  \bigg( e^{ (s_{j+1} - s_{j}) {\mathcal
L}_m(\gamma_j, \gamma_{j})} {\mathcal L}_m(\gamma_j, \gamma_{j+1})
\bigg) \label{eq_wdiff}
\end{align}
where $s_0 = 0$.

Let us introduce the following two constants characterizing the
volatility function:
\begin{equation}
\Sigma_0 = \inf_{x \in A_m} \sigma(x), \;\;\;\;\; \Sigma_1 = \sup_{x
\in A_m } \sqrt{\sigma(x)^2 + h_m \lvert \mu(x) \lvert }.
\end{equation}
and let
\begin{equation}
M = \sup_{x \in A_m} \lvert \mu(x)\lvert .
\end{equation}
Since our interval is bounded, we have that $\Sigma_0>0$ and
$\Sigma_1, M <\infty$.

\begin{figure}
\begin{center}
    \includegraphics[width = 12cm]{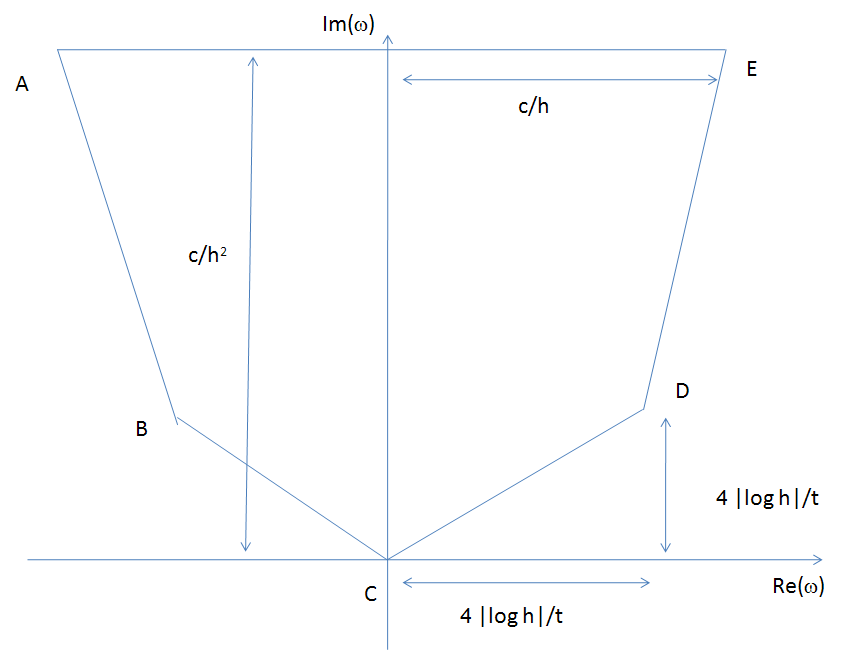}
    \caption{Contour of integration for the integral in (\ref{eq_greenfunc}).
    ${\mathcal C}_+$ is the countour joining the point $D$ to the
    points $E, A, B$. ${\mathcal C}_-$ is
    the countour joining the point $B$ to $C$ to $D$.}
    \label{fig_contour2}
\end{center}
\end{figure}

Let us introduce the following Green's function:
\begin{equation}
G_{m}(x, y; \omega) = \int_0^\infty U_{m}(x, 0; y, t) e^{-i\omega t}
dt = {1\over \mathcal L + i\omega}(x, y).
\end{equation}
The propagator can be expressed as the following contour integral
\begin{equation}
U_{m}(x, 0; y, T) = \int_{{\mathcal C}_-} {d\omega \over 2 \pi}
G_{m}(x, y; \omega) e^{i\omega T} + \int_{{\mathcal C}_+} {d\omega
\over 2 \pi} G_{m}(x, y; \omega) e^{i\omega T}. \label{eq_greenfunc}
\end{equation}
Here, ${\mathcal C}_+$ is the contour joining the point $D$ to the
points $E, A, B$ in Fig. \ref{fig_contour2}, while ${\mathcal C}_-$
is the contour joining the point $B$ to $C$ to $D$. By design, each
point $\omega$ on the upper path ${\mathcal C}_+$ is separated from
the spectrum of $\mathcal L$.

\begin{lemma} \label{lemma_cplus} For $m$ sufficiently large,
there is a constant $c>0$ such that
\begin{equation}
\bigg\lvert\int_{{\mathcal C}_+} {d\omega \over 2 \pi} G_{m}(x, y;
\omega) e^{i\omega T}\bigg\lvert \leq c h^3 .
\end{equation}
\end{lemma}
\begin{proof}
The proof is based on the geometric series expansion
\begin{equation}
G_{m}(\omega) = {1\over {\mathcal L}^m + i\omega} =
\sum_{j=0}^\infty {1\over {1\over2}\sigma^2\Delta^m + i\omega}
\bigg[ \mu\nabla^m {1\over {1\over2}\sigma^2\Delta^m + i\omega}
\bigg]^j \label{geoseries}
\end{equation}
whose convergence for $\omega\in{\mathcal C}_+$ can be established
by means of a Kato-Rellich relative bound, see \cite{Ka1976}. More
precisely, for any $\alpha>0$, one can find a $\beta>0$ such that
the operators $\nabla^m$ and $\Delta^m$ satisfy the following
relative bound estimate:
\begin{equation}
\lvert\lvert \nabla^m f\lvert\lvert_2 \leq \alpha \lvert\lvert
\Delta^m f \lvert\lvert_2 + \beta \lvert\lvert f \lvert\lvert_2.
\end{equation}
for all periodic functions $f$ and all $m\ge m_0$. This bound can be
derived by observing that $\nabla^m$ and $\Delta^m$ can be
diagonalized simultaneously by a Fourier transform, as done in the
previous section, and by observing that for any $\alpha>0$, one can
find a $\beta>0$ such that
\begin{equation}
\bigg\lvert {\sin h_m p \over h_m}  \bigg\lvert \leq \alpha
\bigg\lvert {\cos h_m p - 1 \over h_m^2} \bigg\lvert + \beta
\end{equation}
for all $m\ge m_0$ and all $p\in B_m$.

Under the same conditions, we also have that
\begin{equation} \big\lvert\big\lvert
\mu \nabla^m f\big\lvert\big\lvert_2 \leq {2 M \alpha \over
\Sigma_0^2} \bigg\lvert\bigg\lvert {1\over2} \sigma^2 \Delta^m f
\bigg\lvert\bigg\lvert_2 + \beta \lvert\lvert f \lvert\lvert_2.
\end{equation}
Hence
\begin{equation} \bigg\lvert\bigg\lvert
\mu \nabla^m {1\over {1\over2}\sigma^2\Delta^m + i\omega}
f\bigg\lvert\bigg\lvert_2 \leq {2 M \alpha \over \Sigma_0^2}
\bigg\lvert\bigg\lvert {1\over2} \sigma^2 \Delta^m {1\over
{1\over2}\sigma^2\Delta^m + i\omega} f \bigg\lvert\bigg\lvert_2 +
\beta \bigg\lvert\bigg\lvert {1\over {1\over2}\sigma^2\Delta^m +
i\omega} f \bigg\lvert\bigg\lvert_2 < 1
\end{equation}
where the last inequality holds if $\omega\in{\mathcal C}_+$, if
$\alpha$ is chosen sufficiently small and if $m$ is large enough. In
this case, the geometric series expansion converges in
(\ref{geoseries}) converges in $L^2$ operator norm. The uniform norm
of the kernel $\lvert G_{m}(x, y; \omega) \lvert$ is pointwise
bounded from above by $h_m^{-1}$.

Since the points $B$ and $D$ have imaginary part equal at height $
4{\lvert \log h_m \lvert \over T}$, the integral over the contour
${\mathcal C}_+$ converges also and is bounded from above by $c
h_m^3$ in uniform norm.

\end{proof}

\begin{lemma} If $q\ge {e^2
\Sigma_1^2 T\over 2 h_m^2}$ we have that
\begin{equation}
W_m(\gamma, q; 0, T) \leq \sqrt{q\over 2\pi} \exp\left(-{\Sigma_0^2
T\over 2} - q\right). \label{eq_wbound}
\end{equation}
\end{lemma}
\begin{proof}
Let us define the function
\begin{equation}
\phi(t) = {\Sigma_1^2\over 2 h_m^2} \; e^{-{\Sigma_0^2 t \over 2
h_m^2}} \; 1(t\ge 0)
\end{equation}
where $1(t\ge0)$ is the characteristic function of $\RRR_+$. We have
that
\begin{equation}
W_m(\gamma, q; 0, T) \leq \phi^{\star q}(T)
\end{equation}
where $\phi^{\star q}$ is the $q-$th convolution power, i.e. the
$q-$fold convolution product of the function $\phi$ by itself. The
Fourier transform of $\phi(t)$ is given by
\begin{equation}
\hat\phi(\omega) = {\Sigma_1^2\over 2 h_m^2} \int_0^\infty
e^{-i\omega t - {\Sigma_0^2 t\over 2 h_m^2}} dt = {\Sigma_1^2 \over
2 i \omega h_m^2 + \Sigma_0^2}.
\end{equation}
The convolution power is given by the following inverse Fourier
transform:
\begin{equation}
\phi^{\star q}(T) = \int_0^\infty \hat \phi(\omega)^q e^{i \omega T}
= \left( {\Sigma_1\over\Sigma_0} \right)^{2 q}
\int_{-\infty}^\infty\left( 1 + {2 i \omega h_m^2 \over \Sigma_0^2}
\right)^{-q} e^{i \omega T} {d\omega\over 2\pi}.
\end{equation}
Introducing the new variable $z = 1 + {2i\omega h_m^2\over
\Sigma_0^2}$, the integral can be recast as follows
\begin{equation}
\phi^{\star q}(T) = {\Sigma_0^{2-2q} \Sigma_1^{2q}\over 4\pi i
h_m^2} \lim_{R\to\infty} \int_{{\mathcal C}_R} z^{-q} \exp\left(
{\Sigma_0^2 T\over 2 h_m^2} (z-1) \right) dz \label{eq_intc}
\end{equation}
where ${\mathcal C}_R$ is the contour in Fig. \ref{fig_contour1}.
Using the residue theorem and noticing that the only pole of the
integrand is at $z = 0$, we find
\begin{equation}
\phi^{\star q}(T) = {1\over (q-1)!} \left({\Sigma_1^2 T \over 2
h_m^2}\right)^q \exp\left( - \Sigma_0^2 T \over 2 h_m^2 \right).
\end{equation}
Making use of Stirling's formula $q! \approx \sqrt{2\pi}
q^{q+{1\over2}} e^{-q}$, we find
\begin{equation}
\phi^{\star q}(T) \approx \sqrt{q\over 2\pi} \exp\left( -
{\Sigma_0^2 T \over 2 h_m^2} + q\log{\Sigma_1^2 T \over 2 h_m^2} + q
(1-\log q) \right).
\end{equation}
If $\log q \ge \log {\Sigma_1^2 T \over 2 h_m^2} +2$, then we arrive
at the bound in (\ref{eq_wbound}).

\begin{figure}
\begin{center}
    \includegraphics[width = 12cm]{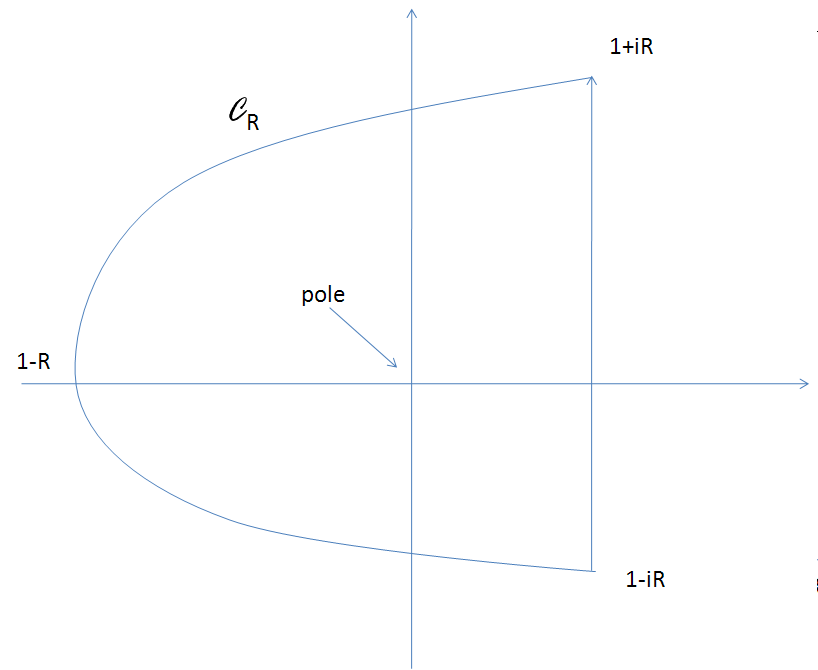}
    \caption{Contour of integration ${\mathcal C}_R$ for the integral in (\ref{eq_intc}).}
    \label{fig_contour1}
\end{center}
\end{figure}

\end{proof}

\begin{definition}{\bf (Decorating Paths.)}
 Let $m\ge m_0$ and let $\gamma = \{y_0, y_1, y_2,
.... \}$ be a symbolic sequence in $\Gamma_m$. A {\it decorating
path around $\gamma$} is defined as a symbolic sequence $\gamma' =
\{y_0, y_1', y_2', .... \}$ with $y_i' \in h_{m+1} \ZZZ$ containing
the sequence $\gamma$ as a subset and such that if $y_j' = y_i$ and
$y_k' = y_{i+1}$, then all elements $y_n'$ with $j < n < k$ are such
that $\lvert y_n' - y_j' \lvert \leq h_{m+1}$. Let ${\mathcal
D}_{m+1}(\gamma)$ be the set of all decorating sequences around
$\gamma$. The decorated weights are defined as follows:
\begin{equation}
{\tilde W}_{m}(\gamma, q; 0, T) = \sum_{q' = q}^\infty \sum_{
\begin{matrix}
\gamma' \in {\mathcal D}_{m+1}(\gamma)\\
\gamma'_{q'} = \gamma_q
\end{matrix}
} W_{m+1}(\gamma', q'; 0, T).
\end{equation}
Finally, let us introduce also the following Fourier transform:
\begin{equation}
\hat W_{m}(\gamma, q; \omega) = \int_0^\infty W_{m}(\gamma, q; 0, t)
e^{i\omega t} dt, \;\;\;\; \hat {\tilde W}_{m}(\gamma, q; \omega) =
\int_0^\infty \tilde W_{m}(\gamma, q; 0, t) e^{i\omega t} dt.
\end{equation}
\end{definition}

\begin{definition}{\bf (Notations.)}
In the following, we set $h = h_{m+1}$ so that $h_m = 2 h$. We also
use the Landau notation $O(h^n)$ to indicate a function $f(h)$ such
that $ h^{-n} f(h) $ is bounded in a neighborhood of $(0)$.
\end{definition}

\begin{lemma}\label{lemma_cplus}  Let $x, y\in A_m$ and let ${\mathcal C_-}$ be an
integration contour as in Fig. \ref{fig_contour2}. Then
\begin{equation}
\bigg\lvert \bigg( \int_{{\mathcal C}_-}  2 G_{m+1}(x, y; \omega) -
G_m(x, y; \omega) \bigg) e^{i\omega T} {d\omega\over 2\pi}
\bigg\lvert = O(h^3).
\end{equation}
\end{lemma}

\begin{proof}

We have that
\begin{align}
2 G_{m+1}(x, y; \omega) - G_m(x, y; \omega) = & \sum_{q=1}^\infty 2
^{-q} \sum_{
\begin{matrix}
\gamma\in \Gamma_m : \gamma_0 = x, \gamma_{q} = y\\
\lvert\gamma_j - \gamma_{j-1}\lvert = 1 \forall j\ge1
\end{matrix}}
\left( 2 \hat {\tilde W}_m (\gamma, q; \omega) - \hat W_m (\gamma,
q; \omega)\right). \label{eq_udiff}
\end{align}
The number of paths over which the summation is extended is
\begin{equation}
N(\gamma, q; x, y) \equiv \sharp\{\gamma\in \Gamma_m : \gamma_0 = x,
\gamma_{q} = y, \lvert\gamma_j - \gamma_{j-1}\lvert = 1 \forall
j\ge1 \} = \left(\begin{matrix} q \\ {q\over2} + k
\end{matrix}\right)
\end{equation}
where $k = {\lvert y - x \lvert \over h_m}.$ Applying Stirling's
formula we find
\begin{equation}
N_\gamma \lesssim 2^q \sqrt{2\over \pi q}.
\end{equation}
Hence
\begin{align}
&\bigg\lvert \int_{{\mathcal C}_-} \bigg( 2 G_{m+1}(x, y; \omega) -
G_m(x, y; \omega) \bigg) e^{i\omega T}  {d\omega\over 2\pi}
\bigg \lvert \notag \\
&\leq c \sum_{q=1}^\infty \sqrt{1\over q} \max_{
\begin{matrix}
\gamma\in \Gamma_m : \gamma_0 = x, \gamma_{q} = y\\
\lvert\gamma_j - \gamma_{j-1}\lvert = 1 \forall j\ge1
\end{matrix}}
\bigg \lvert \int_{{\mathcal C}_-} \bigg(2 \hat {\tilde W}_m
(\gamma, q; \omega) - \hat W_m (\gamma, q; \omega)\bigg) e^{i\omega
T}  {d\omega\over 2\pi} \bigg\lvert. \notag
\\\label{eq_udiff}
\end{align}
for some constant $c\approx\sqrt{2\over\pi} >0$. It suffices to
extend the summation over $q$ only up to
\begin{equation}
q_{\rm max} \equiv {e^2 \Sigma_1^2 T\over 2 h^2}.
\end{equation}
To resum beyond this threshold, one can use the previous lemma. More
precisely, we have that
\begin{align}
&\bigg\lvert \int_{{\mathcal C}_-} \bigg( 2 G_{m+1}(x, y; \omega) -
G_m(x,
y; \omega) \bigg) e^{i\omega T}  {d\omega\over 2\pi} \bigg\lvert \notag \\
&\leq c \sqrt{q_{\rm max}} \max_{
\begin{matrix}
q, \gamma\in \Gamma_m : \gamma_0 = x, \gamma_{q} = y\\
\lvert\gamma_j - \gamma_{j-1}\lvert = 1 \forall j\ge1
\end{matrix}}
\bigg \lvert \int_{{\mathcal C}_-} \bigg( 2 \hat {\tilde W}_m
(\gamma, q; \omega) - \hat W_m (\gamma, q; \omega) \bigg) e^{i\omega
T}  {d\omega\over 2\pi}  \bigg\lvert. \notag \\\label{eq_udiff2}
\end{align}

Let $v(x) = \sigma(x)^2$. To evaluate the resummed weight function,
let us form the matrix
\begin{equation} \bar {\mathcal L}(x;h) =
\left(
\begin{matrix}
-{v\left(x+h\right)\over h^2} && {v\left(x+ h\right)\over 2 h^2} -
{\mu(x+h)\over 2 h} && 0 \\
{v\left(x\right)\over 2 h^2}  +\mu(x)/(2h) && -{ v\left(x \right)\over h^2}  && {v\left(x\right) \over 2 h^2} - {\mu(x) \over 2 h} \\
0 && {v\left(x - h\right)\over 2 h^2} + {\mu(x-h)\over 2 h} &&
-{v\left(x - h \right) \over h^2}
\end{matrix} \right)
\end{equation}
and decompose it as follows:
\begin{equation}
\bar {\mathcal L}(x;h) = {1\over h^2} \bar {\mathcal L}_0(x) +
{1\over h} \bar {\mathcal L}_1(x) + \bar {\mathcal L}_2(x) + h \bar
{\mathcal L}_3(x) + O(h^2).
\end{equation}
where
\begin{equation}
\bar {\mathcal L}_0(x) = \left(
\begin{matrix}
-v(x) && {1\over2} {v(x)} && 0 \\
{1\over2} {v(x)} && -{v(x)} && {1\over2} {v(x)} \\
0 && {1\over2} {v(x)} && -{v(x)}
\end{matrix}
\right),
\end{equation}
\begin{equation}
\bar {\mathcal L}_1(x) = \left(
\begin{matrix}
-v'(x) && {1 \over2} v'(x) - {1\over2} \mu(x) && 0 \\
{1\over2} \mu(x) && 0 && -{1\over2} \mu(x) \\
0 && - {1\over2} v'(x) + {1\over2}  \mu(x) &&  v'(x)
\end{matrix}\right),
\end{equation}
\begin{equation}
\bar {\mathcal L}_2(x) = \left(
\begin{matrix}
-{1\over2} v''(x)  && {1\over4} v''(x) - {1\over2} \mu'(x)&& 0 \\
0 && 0 && 0 \\
0 && {1\over4} v''(x) - {1\over2} \mu'(x) && -{1\over2} v''(x)
\end{matrix}\right).
\end{equation}
and
\begin{equation}
\bar {\mathcal L}_3(x) = \left(
\begin{matrix}
-{1\over6} v'''(x)  && {1\over12} v'''(x) - {1\over4} \mu''(x)&& 0 \\
0 && 0 && 0 \\
0 && - {1\over12} v'''(x) + {1\over4} \mu''(x) && {1\over6} v'''(x)
\end{matrix}\right).
\end{equation}

Let us introduce the sign variable $\tau = \pm1$, the functions
\begin{align}
\phi_{0}(t, x, \tau) &\equiv 2 {\mathcal L}_m(x, x+2\tau h) e^{ t
{\mathcal
L}_m(x, x)} 1(t\ge0) \\
\phi_{1}(t, x, \tau) &\equiv 2 {\mathcal L}_{m+1}(x+\tau h, x+2\tau
h) e^{t \bar {\mathcal L}(x; h)}(x, x+\tau h) 1(t\ge0)
\end{align}
and their Fourier transforms
\begin{align}
\hat\phi_{0}(\omega, x, \tau) &= \left({v(x)\over 4 h^2} + \tau {
\mu(x)\over 2 h}\right) \left( {v(x)\over 4 h^2} + i \omega
\right)^{-1} \notag \\ \hat\phi_{1}(\omega, x, \tau) &=
\left({v(x)\over h^2} + \tau  {\mu(x) + v'(x)\over h} +
{v''(x)+\mu'(x)\over2} +\left( {v'''(x)\over6} + {\mu''(x)\over2}
\right)\tau h +
O(h^2)\right) \notag \\
&\hskip8cm < x \lvert \left( -\bar {\mathcal L}(x; h) + i\omega
\right)^{-1}\lvert x+\tau h >.
\end{align}
where
\begin{equation}
\lvert x> = \left(\begin{matrix} 0 \\ 1 \\
0 \end{matrix}\right), \;\;\;\;{\rm and}\;\;\;\; \lvert x+\tau h> =
\left( \begin{matrix} \delta_{\tau, 1} \\ 0 \\
\delta_{\tau, -1} \end{matrix}\right).
\end{equation}
We also require the functions
\begin{equation}
\psi_{0}(t, x) \equiv e^{ t {\mathcal L}_m(x, x) } 1(t\ge0),
\;\;\;\;\;\; \psi_{1}(t, x) \equiv e^{t \bar {\mathcal L}(y; h)}(x,
x) 1(t\ge0)
\end{equation}
and the corresponding Fourier transforms
\begin{align}
\hat\psi_{0}(\omega, x) = \left( {v(x)\over 4 h^2}  + i \omega
\right)^{-1}, \;\;\;\;\; \hat\psi_{1}(\omega, x) =  <x \lvert \left(
-\bar {\mathcal L}(x; h) + i\omega \right)^{-1}\lvert x>.
\end{align}

If $\gamma$ is a symbolic sequence, then
\begin{align}
\hat W_m(\gamma, q; \omega) &= \hat\psi_{0}(\omega, \gamma_q)
\prod_{j=0}^{q-1} \hat\phi_0(\omega;
\gamma_j, {\rm sgn}(\gamma_{j+1}-\gamma_j))\\
 \hat {\tilde W}_m(\gamma, q; \omega)
&= \hat\psi_{1}(\omega, \gamma_q) \prod_{j=0}^{q-1}
\hat\phi_1(\omega; \gamma_j, {\rm sgn}(\gamma_{j+1}-\gamma_j)).
\end{align}

Let us estimate the difference between the functions
$\hat\phi_{1}(\omega, x, \tau)$ and $\hat\phi_{2}(\omega, x, \tau)$
assuming that $\omega$ is in the contour ${\mathcal C}_-$ in Fig.
\ref{fig_contour1}. Retaining only terms up to order up to $O(h^3)$,
we find
\begin{eqnarray}
\hat\phi_{0}(\omega, x, \tau) = 1 + {2\mu(x)\tau h \over v(x)} - {4
i \omega h^2 \over v(x)} - 8\mu(x) {i\omega\tau h^3 \over v(x)^2}
-{16\omega^2 h^4 \over v(x)^2}
+ O(h^5). \notag \\
\end{eqnarray}
A lengthy but straightforward calculation which is best carried out
using a symbolic manipulation program, gives
\begin{align}
&\hat\phi_{1}(\omega, x, \tau) = 1+{2\mu(x) \tau h \over v(x)}-{ 4 i
\omega h^2 \over v(x)}  - \big[8\mu(x)  - v'(x) \big]
{i\omega\tau h^3 \over v(x)^2} \notag \\
&\hskip6cm+ r(x)\cdot h^3 \tau
 + i\omega h^4 p(x) - {14\omega^2 h^4 \over v(x)^2} + O(h^5) \notag \\
\end{align}
where
\begin{align}
& r(x) = {1\over 2 v(x)^3} \big[ \mu''(x) v(x) -  4\mu(x)^3 + 2
v'(x) \mu(x)^2 - 2v'(x) v(x) \mu'(x)  \notag \\
&\hskip6cm  - \big(2\mu(x) \mu'(x) + v''(x) v(x) - 2v'(x)^2\big) \mu(x) \big] . \notag \\
&p(x) = {1\over v(x)^3} \big[4\mu(x)^2-2v'(x)\mu(x)+4v(x)\mu'(x)+
v''(x)
v(x)- 2 v'(x)^2  \big].  \notag \\
\end{align}

We have that
\begin{align}
& \sum_{j=0}^{q-1} \bigg(\log \hat\phi_0(\omega; \gamma_j, {\rm
sgn}(\gamma_{j+1}-\gamma_j)) -
 \log \hat\phi_1(\omega; \gamma_j, {\rm
sgn}(\gamma_{j+1}-\gamma_j))\bigg)
\notag \\
&= \sum_{j=0}^{q-1} \bigg(  {i \omega v'(\gamma_j)\over
v(\gamma_j)^2} + r(\gamma_j) \bigg) h^3 {\rm
sgn}(\gamma_{j+1}-\gamma_j) + \big(\lvert \omega\lvert \lvert\lvert
p\lvert\lvert_{\infty} + 2 \lvert \omega\lvert^2
\lvert\lvert v^{-2} \lvert\lvert_{\infty} \big) O(h^4 q) \notag \\
&= i\omega h^2 \log\bigg({v(\gamma_{q})\over v(\gamma_0)}\bigg) +
h^2 \big(R(\gamma_q) - R(\gamma_0)\big) + \big(\lvert \omega\lvert
\lvert\lvert p\lvert\lvert_{\infty} + 2 \lvert \omega\lvert^2
\lvert\lvert v^{-2} \lvert\lvert_{\infty} \big)
O( h^4 q) \notag \\
\end{align}
where $R(x)$ is a primitive of $r(x)$, i.e.
\begin{equation}
R(x) = \int^x r(z) dz.
\end{equation}

We conclude that there is a constant $c>0$ such that
\begin{equation}
\bigg\lvert \int_{{\mathcal C}_-} \bigg( \prod_{j=0}^{q-1}
\hat\phi_0(\omega; \gamma_j, {\rm sgn}(\gamma_{j+1}-\gamma_j)) -
\prod_{j=0}^{q-1} \hat\phi_1(\omega; \gamma_j, {\rm
sgn}(\gamma_{j+1}-\gamma_j)) \bigg) e^{i\omega T}  {d\omega\over
2\pi} \bigg\lvert \leq c h^2.
\end{equation}
for all $q\leq q_{\max}$. Here we use the decay of $e^{i\omega T}$
in the upper half of the complex $\omega$ plane to offset the
$\omega$ dependencies in the integrand. Similar calculations lead to
the following expansions:
\begin{equation}
\hat\psi_{0}(\omega, y) = {4 h^2 \over v(y)}  + O(\omega h^4),
\;\;\;\;\; \hat\psi_{1}(\omega, y) =  {2 h^2 \over v(y)} + O(\omega
h^4) = {1\over 2} \hat\psi_{0}(\omega, y) + O(\omega h^4).
\end{equation}

Since $q<c h^{-2}$ and $\omega \leq \lvert \log h \lvert$, we find
\begin{align}
\bigg\lvert \int_{{\mathcal C}_-} \bigg(  2 G_{m+1}(x, y; \omega) -
G_m(x, y; \omega) \bigg) e^{i\omega T}  {d\omega\over 2\pi}
\bigg\lvert \leq c {q_{\rm max}}^{1\over2} h^4 \leq c h^3.
\label{eq_udiff3}
\end{align}
This completes the proof of the Lemma and of the Theorem.
\end{proof}

\section{Convergence of Time Discretization Schemes}

In this section we analyze the convergence of the time discretized kernel
that is obtained by means of fast exponentiation.

Let $A = (A_m), m = m_0, m_0+1, ...$ be a simplicial sequence
converging to an interval $\lim_{m\to\infty} A_m = [0,L] \subset
\RRR$, where $0<L<\infty$. Let $\mu(x)$ and $\sigma(x)$ be
functions on $[0,L]$ satisfying all the conditions in the previous
section and consider the generator of a diffusion on $A_m$ of the form
\begin{equation}
{\mathcal L}_m = \mu(x) \nabla_{h_m} + {1\over 2} \sigma(x)^2 \Delta_{h_m}.
\label{def_diffdtgen}
\end{equation}

\begin{theorem}\label{theo_diffdtker} {\bf (Convergence Estimates for Fast Exponentiation.)}
Let $\delta t>0$ and consider the discretized kernel
\begin{equation}
U^{\delta t}_m(x, 0; y, T) = \left( 1 + \delta t {\mathcal L}_m \right)^{T\over\delta t} (x, 0; y, T).
\end{equation}
where ${\mathcal L}_m$ is the operator in (\ref{def_diffdtgen}) and
$\delta t_m $ is so small that
\begin{equation}
\min_{x\in A_m} 1 + \delta t_m {\mathcal L}_m(x, x) > 0
\end{equation}
Assume that boundary conditions are either periodic or absorbing
and that the ratio ${T\over \delta t} = N$ is an integer.
Then there is a constant $c>0$ such that
\begin{equation}
\lvert h_m^{-1} U_m(x, 0; y, T) - h_m^{-1} U_m^{\delta t}(x, 0; y, T)
 \lvert \leq c h_m^2
\end{equation}
for all $m\ge m_0$ and all $y\in A_m$.
\end{theorem}

\begin{proof}

A Dyson expansion can also be obtained for the time-discretized kernel and
has the form
\begin{align}
U_m^{\delta t}(y_1, 0; y_2, T) =& \sum_{q=1}^\infty \sum_{\gamma\in \Gamma_m :
\gamma_0 = x, \gamma_{q} = y}
 \sum_{k_1 = 1}^{N} \sum_{k_2 = k_1 + 1}^{N} ...
\sum_{k_{q} = k_{q-1} + 1}^{N}  \notag \\
& \bigg(1 + \delta t \LLL_m(\gamma_0, \gamma_0)\bigg)^{k_{1}-1}
(\delta t)^q \prod_{j=1}^{q} \LLL_m(\gamma_{j-1}, \gamma_j) \bigg(1
+ \delta t \LLL_m(\gamma_j, \gamma_j)\bigg)^{k_{j+1} - k_{j} - 1}
\label{eq_upathint}
\end{align}
where $t_{q+1} = T$ and $k_{q+1}=N$. In this case, the propagator can be expressed
through a Fourier integral as follows:
\begin{equation}
U_m^{\delta t}(y_1, 0; y_2, T) = \int_{-{\pi\over\delta
t}}^{\pi\over\delta t} G_m^{\delta t}(y_1, y_2; \omega) e^{i\omega
t} {d\omega\over 2\pi}
\end{equation}
where
\begin{equation}
G_m^{\delta t}(y_1, y_2; \omega) = \delta t \sum_{j = 0}^{T\over
\delta t} U_m^{\delta t}(y_1, 0; y_2, j \delta t) e^{ - i\omega j
\delta t}.
\end{equation}
The propagator can also be represented as the limit
\begin{equation}
U_m^{\delta t}(y_1, 0; y_2, T) = \lim_{H\to\infty} \int_{{\mathcal
C}_H} G_m^{\delta t}(y_1, y_2; \omega) e^{i\omega t} {d\omega\over
2\pi} \label{eq_greenfuncdt}
\end{equation}
where ${\mathcal C}_H$ is the contour in Fig. \ref{fig_contour3}.
This is due to the fact that the integral along the segments $BC$
and $DA$ are the negative of each other, while the integral over
$CD$ tends to zero exponentially fast as $\Im(\omega) \to \infty$,
where $\Im(\omega)$ is the imaginary part of $\omega$. Using
Cauchy's theorem, the contour in Fig. \ref{fig_contour3} can be
deformed into the contour in Fig. \ref{fig_contour2}. To estimate
the discrepancy between the time-discretized kernel and the
continuous time one, one can thus compare the Green's function along
such contour. Again, the only arc that requires detailed attention
is the arc $BCD$, as the integral over rest of the contour of
integration can be bounded from above as in the previous section.
\begin{figure}
\begin{center}
    \includegraphics[width = 12cm]{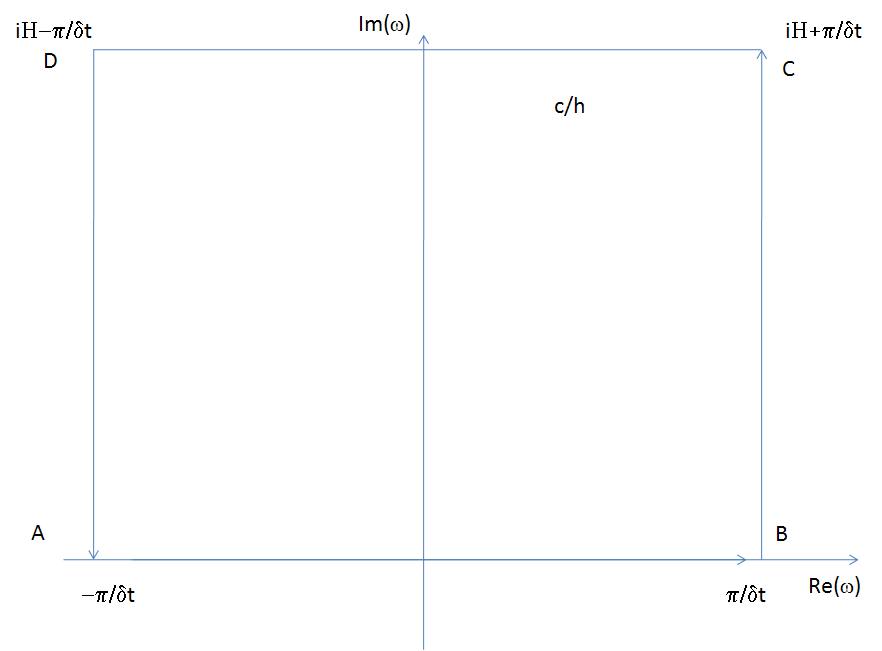}
    \caption{Contour of integration for the integral in (\ref{eq_greenfuncdt}).}
    \label{fig_contour3}
\end{center}
\end{figure}

Let $h = h_m$ and let us introduce the two functions
\begin{align}
\phi_{0}(t, x, \tau) &\equiv 2 {\mathcal L}_m(x, x+\tau h)
e^{ t {\mathcal L}_m(x, x)} 1(t\ge0), \\
\phi_{\delta t}(j, x, \tau) &\equiv 2 {\mathcal L}_m(x, x+\tau h)
\big( 1 + \delta t {\mathcal L}_m(x, x) \big)^{j-1}.
\end{align}
and the corresponding Fourier transforms
\begin{align}
\hat\phi_{0}(\omega, x, \tau) &= \int_0^\infty \phi_{0}(t, x, \tau)
e^{-i\omega t} {d\omega\over 2\pi} =  \left({v(x)\over h^2} + \tau {
\mu(x)\over h}\right) \left( {v(x)\over
h^2} + i \omega \right)^{-1} \\
\hat\phi_{\delta t}(\omega, x, \tau) &= \sum_{j=0}^{t\over \delta t}
\phi_{\delta t}(j, x, \tau) e^{-i\omega j\delta t} =
\left({v(x)\over h^2} + \tau {\mu(x)\over h}\right) \left(
e^{i\omega\delta t} - 1 + \delta t {v(x)\over h^2} \right)^{-1}.
\end{align}
We have that
\begin{align}
\hat\phi_{\delta t}(\omega, x, \tau) &=
\left({v(x)\over h^2} + \tau {\mu(x)\over h}\right)
\left(
i\omega  + {v(x)\over h^2} - {\omega^2 \over 2} \delta t + O(\delta t^2)
\right)^{-1} \notag \\
& = \hat\phi_{0}(\omega, x, \tau) + {\omega^2\over 2 v(x)} h^2 \delta t + O(h^2 \delta t^2).
 = \hat\phi_{0}(\omega, x, \tau) +  O(h^4),
 \end{align}
where the last step uses the fact that $\delta t = O(h^2)$.

Let us also introduce the functions
\begin{align}
\psi_{0}(t, x, \tau) &\equiv e^{ t {\mathcal L}_m(x, x)} 1(t\ge0),
\hskip1cm \psi_{\delta t}(j, x, \tau) \equiv \sum_{k = 1}^j \big( 1
+ \delta t {\mathcal L}_m(x, x) \big)^{j-1}.
\end{align}
and the corresponding Fourier transforms
\begin{align}
\hat\psi_{0}(\omega, x, \tau) &= \left( {v(x)\over h^2} + i \omega
\right)^{-1}, \hskip1cm \hat\psi_{\delta t}(\omega, x, \tau) =
\left( e^{i\omega\delta t} - 1 + \delta t {v(x)\over h^2}
\right)^{-1}.
\end{align}
Again we find that
\begin{align}
\hat\psi_{0}(\omega, x, \tau) = \hat\psi_{\delta t}(\omega, x, \tau)
 +  O(h^4).
\end{align}

If $\gamma$ is a symbolic sequence, then let us set
\begin{align}
\hat W_m(\gamma, q; \omega) &= \hat\psi_{0}(\omega, \gamma_q)
\prod_{j=0}^{q-1} \hat\phi_0(\omega;
\gamma_j, {\rm sgn}(\gamma_{j+1}-\gamma_j))\\
\hat {W}_m^{\delta t}(\gamma, q; \omega)
&=  \hat\psi_{\delta t}(\omega, \gamma_q) \prod_{j=0}^{q-1}
\hat\phi_{\delta t}(\omega; \gamma_j, {\rm sgn}(\gamma_{j+1}-\gamma_j)).
\end{align}
We have that
\begin{align}
G^{\delta t}_{m}(x, y; \omega) - G_m(x, y; \omega) = &
\sum_{q=1}^\infty 2 ^{-q} \sum_{
\begin{matrix}
\gamma\in \Gamma_m : \gamma_0 = x, \gamma_{q} = y\\
\lvert\gamma_j - \gamma_{j-1}\lvert = 1 \forall j\ge1
\end{matrix}}
\left( \hat {W}_m^{\delta t} (\gamma, q; \omega) - \hat W_m (\gamma,
q; \omega)\right). \label{eq_udiff}
\end{align}
The integration over the contour in Fig.  \ref{fig_contour2} can
again be split into an integration over the countour ${\mathcal
C}_-$ and an integration over ${\mathcal C}_+$. The integral over
${\mathcal C}_+$ can be bounded from above thanks to Lemma
\ref{lemma_cplus}. Furthermore,  we have that
\begin{align}
&\bigg\lvert \int_{{\mathcal C}_-} \bigg( G^{\delta t}_{m}(x, y;
\omega) - G_m(x, y; \omega) \bigg)
e^{i\omega T}  {d\omega\over 2\pi} \bigg\lvert \notag \\
&\leq c \sqrt{q_{\rm max}} \max_{
\begin{matrix}
q, \gamma\in \Gamma_m : \gamma_0 = x, \gamma_{q} = y\\
\lvert\gamma_j - \gamma_{j-1}\lvert = 1 \forall j\ge1
\end{matrix}}
\bigg \lvert \int_{{\mathcal C}_-} \bigg( \hat {W}_m^{\delta t}
(\gamma, q; \omega) - \hat W_m (\gamma, q; \omega) \bigg) e^{i\omega
T}  {d\omega\over 2\pi}  \bigg\lvert. \notag
\\
& \leq c h^3 \label{eq_udiff2}
\end{align}

\end{proof}

\section{Hypergeometric Brownian Motion}

This section is based on work in collaboration with Joe Campolieti,
Peter Carr and Alex Lipton, see \cite{ACCL2002}.

In this section we pose the problem of constructing driftless
diffusion models which reduce to a given diffusion by means of a
combination of a measure change and a coordinate transformation.

Consider a Markov process on the simplicial sequence $A_m \subset
\RRR^d$ with generator ${\mathcal L}_m(x, x'; t)$. Let $\rho$ be a
real valued parameter and suppose $f_{m}^1(x)$ and $f_{m}^2(x)$ are
two linearly independent solutions of the equation
\begin{eqnarray}
    \sum_{x'\in A_m} {\mathcal L}_m(x, x'; t) f_m)(x') = \rho f_m(x)
    \label{eq_lap2}
\end{eqnarray}
for all $x\in Int(A_m)$. Notice that on the boundary $\partial
(A_m)$ these equations may fail and, in actual applications, they
will indeed as a rule fail. Consider the function
\begin{equation}
g_m(x; t) = e^{-\rho t} \big(c_1 f_{m}^1(x) + c_2 f_{m}^2(x)\big)
\end{equation}
for some choice of constants $c_1, c_2$ such that this function is
strictly positive. This function satisfies the equation
\begin{equation}
{\partial g_m(x; t)\over\partial t} + (\LLL_m g_m)(x; t) = 0.
\label{eq_numeraire}
\end{equation}
Hence $g_m(x; t)$ defines a measure change and one can construct a
new Markov generator by setting
\begin{equation}
\LLL_m^g(x,x'; t) = {g_m(x'; t)\over g_m(x; t)} \LLL_m(x,x';t)  +
{1\over g_m(x; t)} {\partial g_m(x; t) \over \partial t} \delta_{x
x'}. \label{eq_numchange}
\end{equation}
Consider the linear fractional transformation
\begin{equation}
Y_m(x) = {c_3 f_{m}^1(x) + c_4 f_{m}^2(x) \over c_1 f_{m}^1(x) + c_2
f_{m}^2(x)}.
\end{equation}
for some choice of constants $c_3, c_4$.
\begin{theorem}{\bf(Linear Fractional Transformations.)}
The process $Y(x_t)$ satisfies the martingale condition
\begin{equation}
\lim_{s\downarrow 0} {1\over
s}E_t\bigg[{Y_m(x_{t+s})-Y_m(x_{t})\over s} \bigg] = 0.
\end{equation}
for all $x_t\in Int(A_m)$.
\end{theorem}
\begin{proof}
We have that
\begin{eqnarray}
&\hskip-4cm \lim_{s\downarrow 0} {1\over
s}E_t\bigg[{Y_m(x_{t+s})-Y_m(x_{t})\over s} \bigg] =
 \sum_{x'\in A_m} {\mathcal L}_m^g(x, x')
\big(Y_m(x')-Y_m(x)\big) \notag \\
& = {1\over g_m(x, t)} \bigg(\sum_{x'\in A_m} {\mathcal L}_m(x, x')
\big(c_3 f_{m}^1(x') + c_4 f_{m}^2(x')\big)\bigg) + {1\over g_m(x,
t)^2} {\partial g_m(x; t) \over \partial t}
\big(c_3 f_{m}^1(x) + c_4 f_{m}^2(x)\big) \notag \\
& \hskip-8.75cm = e^{\rho t} {\partial \over \partial t}
\bigg[{e^{-\rho t}( c_3 f_{m}^1(x) + c_4 f_{m}^2(x) ) \over g_m(x,
t)} \bigg] = 0
\notag \\
\end{eqnarray}
\end{proof}

This Theorem provides a general methodology for constructing a
process which is nearly a martingale out of a Markov process. In the
particular case of one-dimensional diffusion processes, the
construction gives rise to families with up to 7 adjustable
parameters of analytically solvable diffusions with drift equal to
zero within the interior of the domain of definition. What makes the
case of one-dimensional diffusions special is the fact that the
function $Y_m(x)$ is invertible in the limit as $m\to\infty$, i.e.
either monotone increasing or monotone decreasing in this limit.

More specifically, consider the diffusion with Markov generator:
\begin{equation}
{\mathcal L}_m = \delta(x) \nabla_{h_m} + {1\over 2} \nu(x)^2
\Delta_{h_m}. \label{ch1_app_eq2}
\end{equation}
where $x\in\RRR$. Two important cases are the CIR process (reducing
to the Bessel equation in the continuous limit) for which
\begin{equation}
\delta(x) = (\lambda_0+\lambda_1x), \;\;\;\;\; \nu(x) = \nu_0
\sqrt{x} \label{cirdef}
\end{equation}
with $x\in\RRR_+$ and the Jacobi process (reducing to a gaussian
hypergeometric polynamials of type $_2F_1$) for which
\begin{equation}
\delta(x) = (\lambda_0+\lambda_1x), \;\;\;\;\; \nu(x) = \nu_0 \sqrt{
x(1 - x)}
\end{equation}
and $x\in[0,1]$.

The construction above admits a continuous limit if the simplicial
functions $f_m^j(x), j=1, 2, 3, 4$ converge to twice differentiable
functions $f^j(x), j = 1, 2, 3, 4$ and satisfy the equation
\begin{eqnarray}
    {\mathcal L} f = \rho f
\end{eqnarray}
within the interior of the domain $A = \lim_{m\to\infty} A_m$. Here
\begin{equation}
{\mathcal L} = \delta(x) {\partial\over\partial x} + {1\over 2}
\nu(x)^2 {\partial^2 \over\partial x^2}. \label{eql_cont}
\end{equation}
\begin{theorem}{\bf(Invertibility.)} Let $f^j(x), j=1, 2, 3, 4$ be functions satisfying
equation \ref{eql_cont} in the interior of the corresponding domain
of definition and let
\begin{equation}
Y(x) = {c_3 f^1(x) + c_4 f^2(x) \over c_1 f^1(x) + c_2 f^2(x)}.
\end{equation}
for some choice of constants $c_1, c_2, c_3, c_4$ such that the
denominator in this equation has no zeros. Then we have that
\begin{equation}
Y(x) = \int^x {W(y)\over g(y)^2} dy + {\rm const}
\end{equation}
where
\begin{equation}
W(x) = \pm {\rm \sigma_0} \exp\bigg( - \int^x {2\delta(y)\over
\nu(y)^2} dy\bigg) \label{eqw}
\end{equation}
where either sign would be allowed and $\sigma_0$ is a positive
constant. In particular, the function $Y(x)$ is invertible.
\end{theorem}
\begin{proof}
Let us introduce the Wronskian
\begin{equation}
W(x) = {d h(x)\over d x } g(x) - {d g(x)\over d x } h(x)
\end{equation}
such that
\begin{equation}
g(x) = c_1 f^1(x) + c_2 f^2(x), \;\;\;\;\; h(x) = c_3 f^1(x) + c_4
f^2(x).
\end{equation}
A direct calculation shows that
\begin{equation}
{d\over d x} Y(x) = W(x) Y(x)
\end{equation}
and
\begin{equation}
{d\over d x} W(x) = - {2\delta(x)\over \nu(x)^2} W(x). \label{eqw2}
\end{equation} Hence, (\ref{eqw}) gives the solution to the
equation in (\ref{eqw2}).
\end{proof}

Let $X(y)$ be the inverse of the function $Y(x)$, we have that
\begin{equation}
{dy\over \sigma(y)} = \pm {dX(y)\over \nu(X(y))}
\end{equation}
where
\begin{equation}
    \sigma(y) = Y'(X(y)) \nu(X(y)) = {\sigma_0 \nu(X(y))
        \exp\left({-2\int^{X(y)}
            {\delta(s) ds\over \nu(s)^2}}\right)
            \over  g(X(y), \rho)^2}
\label{eq:volatility}
\end{equation}

\begin{theorem}{\bf(Kernel mapping.)}
In the limit $m\to\infty$, the propagator density for the process
$Y_t$ with zero drift and volatility as in (\ref{eq:volatility}) is
given by
\begin{equation}
U(y_0, 0; y, t) = {\nu(X(y))\over \sigma(y)}
        {g(X(y),\rho)
            \over
          g(X(y_0),\rho) }
        e^{-\rho t} u(X(y_0), 0; X(y), t)
\label{eq:propagator}
\end{equation}
where $u(x_0, 0; x, t)$ is the propagator density for the process of
generator (\ref{eql_cont}).
\end{theorem}

Let us consider in detail the case of the CIR model in
(\ref{cirdef}). If $\lambda_1=0$, then the pricing kernel for the
state variable is expressed in terms of modified Bessel functions as
follows:
\begin{equation}
    u(x,t; x_0,0) = \left({x\over x_0} \right)^{{1\over2}({2\lambda_0\over \nu_0^2}-1)}{e^{-2(x+x_0)/\nu_0^2t}\over \nu_0^2 t/2} I_{{2\lambda_0\over \nu_0^2} - 1}\left({4\sqrt{x x_0}\over \nu_0^2 t}\right)\,.
\label{eq:modprop}
\end{equation}
The generating function is
\begin{equation}
\hat v(x,\rho) = x^{{1\over 2}(1-{2\lambda_0\over \nu_0^2})} \bigg[
q_1I_{{2\lambda_0\over \nu_0^2} - 1}\left(\sqrt{8\rho
x\over\nu_0^2}\right) + q_2K_{{2\lambda_0\over \nu_0^2} -
1}\left(\sqrt{8\rho x\over \nu_0^2}\right) \bigg] \,,
\label{sec1_eq3}
\end{equation}
with arbitrary constants $q_1$,$q_2$. Here $I_\nu(z)$ is the
modified Bessel function of order $\nu$ and $K_\nu(z)$ is the
associated McDonalds function. In this case we obtain a dual family
with 6 adjustable parameters.

In case $\lambda_1<0$. the propagator density for the state variable
$x$ can still be expressed in terms of modified Bessel functions as
follows:
\begin{equation}
    u(x,t; x_0,0) = c_t\bigg({x e^{-\lambda_1 t} \over x_0}\bigg)^{{1\over 2}({2\lambda_0\over \nu_0^2}-1)}
\exp\left[-c_t(x_0e^{\lambda_1 t} + x)\right] I_{{2\lambda_0\over
\nu_0^2}-1}\bigg(2 c_t\sqrt{xx_0 e^{\lambda_1 t}}\bigg)\,,
\label{eq:modprop2}
\end{equation}
where $c_t \equiv 2\lambda_1/(\nu_0^2(e^{\lambda_1 t} -1))$. For a
derivation see \cite{Giorno88}. The general solution of equation
(\ref{eq_lap2}) reduces to Whittaker's equation and generating
functions have the general form
\begin{equation}
\hat v(x,\rho) = x^{-\lambda_0/ \nu_0^2} e^{-\lambda_1 x/\nu_0^2}
\bigg[ q_1 W_{k,m}\bigg(-{2\lambda_1\over \nu_0^2} x\bigg) + q_2
M_{k,m}\bigg(-{2\lambda_1\over \nu_0^2} x\bigg) \bigg]
\label{sec2_eq3}
\end{equation}
for arbitrary constants $q_1$,$q_2$. Here $W_{k,m}(\cdot)$ and
$M_{k,m}(\cdot)$ are Whittaker functions which can also be expressed
in terms of confluent hypergeometric functions or in terms of Kummer
functions.\cite{Abramowitz} This construction gives rise to a dual
family with 7 free parameters where
\begin{equation}
k = {\lambda_0 \over \nu_0^2} + {\rho \over \lambda_1}\,,\,\,m =
{\lambda_0 \over \nu_0^2} - {1\over 2}\,. \label{sec2_k_defn}
\end{equation}

The 7 parameter family which reduces to the CIR model has a local
volatility function defined on either an interval or on a half line
and behaves asymptotically as the CEV volatility on one hand and as
a quadratic model on the other. This hybrid shape allows for greater
flexibility.

Next, we show that classic exact solutions in the literature, namely
quadratic and CEV models, can all be rediscovered as particular
cases of our general formula for the Bessel family where we make use
of the above solutions to the underlying $x$ space process with
$\beta = {1\over2}$, $\lambda_1=0$ and $\lambda\equiv\lambda_0$.
Without loss of generality, we can fix $\nu_0=2$. Let's specialize
further to the case where
\begin{equation}
    Y(x) = \bar{y} - a{K_{{\lambda\over 2} - 1}(\sqrt{2\rho x})\over I_{{\lambda\over 2} - 1}(\sqrt{2\rho x})}
\label{eq:bes1}
\end{equation}
which leads to a transformed process $y_t = Y(x_t)$ with volatility
\begin{equation}
    \sigma(y) = {a \over \sqrt{X(y)}
        \big[I_{{\lambda\over 2} - 1}(\sqrt{2\rho X(y)})\big]^2},
\label{eq:bes2}
\end{equation}
where $x=X(y)$ is the inverse of the function in equation
(\ref{eq:bes1}).  In this family, $a$ and $\rho$ are positive,
$\bar{y}$ is arbitrary and $\lambda > 2$.  The function $Y(x)$ maps
the half line $x \in [0,\infty)$ into $y \in (-\infty, \bar{y}]$,
where $Y(x)$ is a strictly monotonically increasing function with
$dY(x)/dx = \sigma(Y(x))/\nu(x)$. This solution region can be
inverted so that $y\in [\bar{y},\infty)$.  This is accomplished by
either replacing $a$ by $-a$ in equation (\ref{eq:bes1}) or by
applying a linear change of variables that maps $y$ into $2\bar{y} -
y$.  In this special case, we make use of the generating function in
equation (\ref{sec1_eq3}), with the choice $q_2=0$, and formula
(\ref{eq:propagator}) reduces to
\begin{equation}
    U(y,t; y_0,0) = {e^{-\rho t - (X(y) + X(y_0))/2t}\over at}{X(y)
    \big[I_{{\lambda\over 2} - 1}(\sqrt{2\rho X(y)})\big]^3 \over I_{{\lambda\over 2} - 1}(\sqrt{2\rho X(y_0)})}
I_{{\lambda\over 2} - 1}\left({\sqrt{X(y) X(y_0)} \over t}\right).
\label{eq:bes3}
\end{equation}
We note that this density integrates exactly to unity in $y$ space
(i.e. no absorption).

\begin{example}{\bf (Quadratic volatility models)} \rm Pricing kernels for quadratic
volatility models are readily obtained as a subset of the above
general family with the special choice of parameter $\lambda = 3$.
After making the substitution $y\to 2\bar{y} - y$ and setting
$a=(\bar{y} - \bar{\bar{y}})/\pi$ the transformation function $Y(x)$
becomes
\begin{equation}
    Y(x) = \bar{y} + {(\bar{y} - \bar{\bar{y}})\over \pi}{K_{1\over 2}
    (\sigma_0\sqrt{x}/2)\over I_{1\over 2}(\sigma_0\sqrt{x}/2)} = \bar{y} + {(\bar{y} - \bar{\bar{y}})\over \exp(\sigma_0\sqrt{x}) -1}
\label{eq:quad1}
\end{equation}
where $\sigma_0>0$. Here, we assume that $\bar{y} > \bar{\bar{y}}$.
The inverse transformation $X(y)$ is given by
\begin{equation}
    X(y) = (1/\sigma_0^2)\log^2[1 + (\bar{y} - \bar{\bar{y}})/(y - \bar{y})],
\label{eq:quad2}
\end{equation}
and the volatility function $\sigma(y)$ is obtained by insertion
into equation (\ref{eq:bes2}) while using the Bessel function of
order ${1\over 2}$,
\begin{equation}
    \sigma(y) = {\sigma_0 \over (\bar{y} - \bar{\bar{y}})}(y - \bar{y})(y - \bar{\bar{y}}).
\label{eq:quad3}
\end{equation}
Inserting the expression (\ref{eq:quad2}) into equation
(\ref{eq:bes3}), one obtains the pricing kernel
\begin{equation}
    U(y,t; y_0,0)
= {2e^{-\sigma_0^2t/8}\over \sigma(y)\sqrt{2\pi t}}\sqrt{(y_0 -
\bar{y})(y_0 - \bar{\bar{y}}) \over (y - \bar{y})(y -
\bar{\bar{y}})}e^{-(\phi(y)^2 +
\phi(y_0)^2)/2\sigma_0^2t}\sinh\bigg({\phi(y_0)\phi(y)\over
\sigma_0^2t}\bigg) \label{eq:quad5}
\end{equation}
where $\phi(y)\equiv \log((y - \bar{\bar{y}})/(y - \bar{y}))$. In
the special case of a volatility function with a double root, i.e.
\begin{equation}
    \sigma(y) = \sigma_0(y - \bar{y})^2
\end{equation}
the pricing kernel is computed by taking the limit as
$\bar{\bar{y}}\to \bar{y}$, and one finds
\begin{eqnarray}
    U(y,t; y_0,0) = {1\over \sigma_0\sqrt{2\pi t}}{(y_0 - \bar{y})\over (y - \bar{y})^3}
    \bigg[e^{-\big((y - \bar{y})^{-1} - (y_0 - \bar{y})^{-1}\big)^2/ 2\sigma_0^2t}
\nonumber \\
- e^{-\big((y - \bar{y})^{-1} + (y_0 - \bar{y})^{-1}\big)^2/
2\sigma_0^2t}
 \bigg].
\label{eq:quad6}
\end{eqnarray}
\end{example}

\begin{example}{\bf (Lognormal models)} \rm
The pricing kernel for the log-normal Black-Scholes model with
$\sigma(y) = \sigma_0 y$ is a particular case of the above formula
for the quadratic model. The derivative with respect to $y$ of the
quadratic volatility function in (\ref{eq:quad3}), evaluated at
$y=\bar{y}$, is $\sigma_0$.  Taking the limit $\bar{\bar{y}}\to
-\infty$ (or $\bar{\bar{y}} << \bar{y}$), while holding the other
parameters fixed, one obtains $\sigma(y) = \sigma_0 (y-\bar{y})$.
The pricing kernel in (\ref{eq:quad5}) gives the kernel for the
log-normal model in the limit $\bar{\bar{y}}\to-\infty$, i.e.
\begin{equation}
    U(y,t; y_0,0)
= {1 \over \ (y - \bar{y})\sigma_0 \sqrt{2\pi t}}
\exp\bigg[-\bigg(\log((y_0 - \bar{y})/(y - \bar{y})) -
{\sigma_0^2\over 2}t\bigg)^2\bigg/2\sigma_0^2t\bigg].
\label{eq:affine1}
\end{equation}
\end{example}

\begin{example}{\bf (CEV model)} \rm
The constant-elasticity-of-variance (CEV) model is recovered in the
limiting case as $\rho\to 0$. Assume $\lambda>2$ and let $\theta>0$
be defined so that $\lambda = \theta^{-1} + 2$. The transformation
$y=Y(x)$
\begin{equation}
    Y(x) = \bar{y} + (\sigma_0^2 x)^{-(2\theta)^{-1}}
\label{eq:cev1}
\end{equation}
has inverse $x=X(y)$ given by
\begin{equation}
    X(y) = \sigma_0^{-2} (y - \bar{y})^{-2\theta},
\label{eq:cev2}
\end{equation}
for any constant $\bar{y}$.  The volatility function for this model
is
\begin{equation}
    \sigma(y) = {\sigma_0\over \lvert\theta\rvert}(y - \bar{y})^{1 + \theta}.
\label{eq:cev3}
\end{equation}
In the limit $\rho\to 0$, the Laplace transform $\hat v(X(y), 0) =
1$, which implies that the numeraire change is trivial in this case.
The pricing kernel can be evaluated by substitution into the general
formula (\ref{eq:propagator}), and after collecting terms, it turns
out to be
\begin{eqnarray}
    U(y,t; y_0,0) = {\lvert\theta\rvert\over \sigma_0^2t}{(y_0 - \bar{y})^{1\over 2} \over (y - \bar{y})^{{3\over 2}+2\theta}}
e^{-\big((y - \bar{y})^{-2\theta} + (y_0 - \bar{y})^{-2\theta}\big)/
2\sigma_0^2t} \nonumber
\\
I_{1 \over 2\theta}\bigg({\big((y - \bar{y})(y_0 -
\bar{y})\big)^{-\theta}\over \sigma_0^2t}\bigg). \label{eq:cev5}
\end{eqnarray}
This formula was derived in the case $\theta>0$, for which the
limiting value $y=\bar{y}$ is not attained and the density is easily
shown to integrate to unity (i.e. no absorption occurs and the
density also vanishes at the endpoint $y=\bar{y}$). We note that the
same formula solves the propagator equation for $\theta < 0$,
leading to the same Bessel equation of order $\pm(2\theta)^{-1}$. In
the range $\theta < 0$, however, the properties of the above pricing
kernel are generally more subtle.  In particular, one can show that
the density integrates to unity for all values $\theta < -1/2$,
hence no absorption occurs for $\theta \in (-\infty,-1/2)$. The
boundary conditions for the density can be shown to be vanishing at
$y=\bar{y}$ (i.e. paths do not attain the lower endpoint) for all
$\theta <-1$. In contrast, for $\theta\in (-1,-1/2)$ the density
becomes singular at the lower endpoint $y=\bar{y}$ (hence this
corresponds to the case that the density has an integrable
singularity for which paths can also attain the lower endpoint, but
are not absorbed). For the special case of $\theta=-1/2$ the formula
gives rise to absorption. [Note that only for the range $\theta \in
(-1/2,0)$ the above pricing kernel is not useful since it gives rise
to a density that has a non-integrable singularity at $y=\bar{y}$.
In this case, however, another solution that is integrable is
obtained by only replacing the order $(2\theta)^{-1}$ by
$-(2\theta)^{-1}$ in the Bessel function. The latter solution for
the density does not integrate to unity and hence gives rise to
absorption which can be of use to price options in a credit
setting.] The special case of $\theta = -1$ gives a nonzero constant
value at the lower endpoint, and recovers the Wiener process with
reflection and no absorption on the interval $[\bar{y},\infty)$ with
\begin{eqnarray}
    U(y,t; y_0,0) = {1\over \sigma_0\sqrt{2\pi t}}\bigg(e^{-(y - y_0)^2/2\sigma_0^2t} + e^{-(y + y_0 - 2\bar{y})^2/2\sigma_0^2t}\bigg).
\label{eq:cev6}
\end{eqnarray}
\end{example}

\section{Stochastic Integrals for Diffusion Processes}\label{sec_feynman}

This section provides a new derivation of a theorem by Cameron,
Feynman, Girsanov, Ito, Kac and Martin, see  \cite{CameronMartin},
\cite{Ito1951}, \cite{Feynman1948} and \cite{Kac1950}. This result
is at the center of stochastic calculus and in the following
sections, we derive far reaching extensions and applications.

\begin{theorem} {\bf (Cameron-Feynman-Girsanov-Ito-Kac-Martin.)}
Consider a diffusion process on the simplicial sequence
$A_m \subset \RRR$ and of generator
\begin{equation}
\LLL_m  = \mu(x, t) \nabla_{h_m} + {\sigma(x, t)^2\over
2}\Delta_{h_m}.
\end{equation}
Consider also the process given by the integral
\begin{equation}
I_t = \int_0^t a(x_s, s) dx_s + b(x_s, s) ds
\end{equation}
where $a(x, t)$ and $b(x, t)$ are smooth functions in both
arguments. Let us introduce also the function $\phi(x, t) = \int^x_0
a(y, t) dy$. We have that
\begin{itemize}
\item[(i)] {\bf Ito's Lemma in integral form.}
\begin{equation}
I_t = \phi(x_t, t) - \phi(x_0, 0) + J_t + O(h).
\end{equation}
where
\begin{equation}
J_t = \int_0^t \bigg(b(x_s, s) - {1\over2} \sigma(x_s, s)^2 a'(x_s,
s) -{\dot \phi}(x_s, s) \bigg) ds.
\end{equation}
where $\dot \phi(x, t) = {\partial\over \partial t}\phi(x, t)$ and
$a'(x, t) = {\partial\over\partial x} a(x, t)$.
\item[(ii)] {\bf Feynman-Kac formula.} The characteristic
function of $J_t$ on the bridge leading from $x$ to $y$
is given by
\begin{equation}
E_0\big[e^{i p J_t} \delta(x_t = y) \lvert x_0 = x\big] =
P\exp\bigg(\int_0^t \bigg(\LLL(s) + i p b(s) - {ip\over 2}
\sigma^2(s) a'(s) - ip \dot \phi(s) + O(h)\bigg)\bigg)(x, y).
\end{equation}
\item[(iii)] {\bf Ito's Lemma in differential form.} Let $\phi(x, t)$ be a smooth
function in both arguments. Then we have that
\begin{align}
&\lim_{s\downarrow0} E_t\left[s^{-1}(\phi(x_{t+s},
t+s) - \phi(x_t, t)) \lvert x_t = x\right] \notag \\
&\hskip3cm = {\partial \phi\over\partial t}(x, t) + \mu(x,
t){\partial \phi\over\partial x}(x, t) + {\sigma(x, t)^2\over 2}
{\partial^2 \phi\over\partial x^2}(x, t) + O(h_m)
\end{align}
and
\begin{equation}
\lim_{s\downarrow0} E_t\left[s^{-1}(\phi(x_{t+s}, t+s) - \phi(x_t,
t))^2 \lvert x_t = x \right] = \sigma(x, t)^2 \left({\partial
\phi\over\partial t}(x, t)\right)^2 + O(h_m).
\end{equation}
For all $n\ge3$ instead we have that
\begin{equation}
\lim_{h\downarrow0} \lim_{s\downarrow0}
E_t\left[s^{-1}(\phi(x_{t+s}, t+s) - \phi(x_t, t))^n \lvert x_t =
x\right] = 0.
\end{equation}
\end{itemize}
\end{theorem}

\begin{proof}
Consider a discretization of the process $I_t = (\Delta I) n_t$.
This can be accomplished by introducing a simplicial complex of one
more dimension. Eventually in the proof,
we shall take the limit as $\Delta I \to 0$ while leaving $h>0$ constant.

Consider the lifted generator
\begin{eqnarray}
\tilde\LLL(x, n; x', n'; t) = \left( {\sigma(x)^2\over 2 h^2} +
{\mu(x, t)\over 2 h} \right) \delta_{x', x+h}\delta\bigg(n' - n
-\bigg[{a(x, t) h \over \Delta I}\bigg]\bigg) \notag  \\
+ \left( {\sigma(x, t)^2\over 2 h^2} - {\mu(x, t)\over 2 h} \right)
\delta_{x', x-h}\delta\bigg(n' - n +\bigg[{a(x, t) h \over \Delta
I}\bigg]\bigg) \notag \\
- {\sigma(x, t)^2\over h^2} \delta_{x' x} \delta(n-n') + {b(x,
t)\over\Delta I}\big( \delta_{n', n+1} - \delta_{n' n} \big)
\end{eqnarray}
where $[a]$ stands for the nearest integer to $a$. The partial
Fourier transform in the $n$ variable of this kernel is
\begin{eqnarray}
\tilde\LLL_p(x, x'; t) &=& \sum_{n} \tilde\LLL(x, 0; x', n; t)
e^{- i \Delta I p n } \notag \\
&=& \left( {\sigma(x, t)^2\over 2 h^2} + {\mu(x, t)\over 2 h}
\right) \exp\left(- i \left[{h a(x, t)\over \Delta I}\right] \Delta
I p \right)
\delta_{x', x+h} \notag  \\
\;\;\;&+& \left( {\sigma(x, t)^2\over 2 h^2} - {\mu(x, t)\over 2 h}
\right) \exp\left(i \left[{h a(x, t)\over \Delta I}\right] \Delta I
p \right) \delta_{x', x-h} \notag \\
\;\;\;&-& {\sigma(x, t)^2\over
h^2} \delta_{x' x} \delta(n-n') + {b(x, t)\over\Delta
I}\big(e^{ip\Delta I } - 1\big)
\end{eqnarray}
In the limit as $\Delta I \to 0$ we find
\begin{eqnarray}
\hskip-3cm \tilde\LLL_p(x, x'; t) = {\sigma(x, t)^2\over 2} \Delta_h
+ \left(
\mu(x, t) - i p \sigma(x, t)^2 a(x, t)\right) \nabla_h \notag \\
\hskip3cm - ip a(x,t) \mu(x, t) - {1\over2} p^2\sigma(x, t)^2 a(x,
t)^2 - i p b(x, t) +O(h).
\end{eqnarray}
Introducing the function $\phi(x, t)$ defined above, we find that
\begin{equation}
e^{ip\phi(x, t)} \LLL(x, x'; t) e^{-ip\phi(x', t)} = \tilde\LLL_p(x,
x'; t) + ip b(x, t)\delta_{x x'} -{ip \over 2} \sigma(x, t)^2 a'(x,
t) \delta_{x x'} + O(h).
\end{equation}
This equality can be rearranged as follows:
\begin{eqnarray}
e^{-ip\phi(x, t)} \tilde \LLL(x, x'; t) e^{ip\phi(x', t)} +
ip\dot\phi(x, t) \delta_{x x'} =& \notag \\
&\hskip-4cm \LLL_p(x, x'; t) - ip b(x, t)\delta_{x x'} + {ip \over
2} \sigma(x, t)^2 a'(x, t) \delta_{x x'} + ip\dot\phi(x, t)
\delta_{x x'} + O(h).
\end{eqnarray}
Hence
\begin{eqnarray}
&\hskip-10cmP\exp\bigg(\int_0^t \tilde\LLL_p(s) ds\bigg)(x, x') = \notag \\
&e^{ip(\phi(x, 0)-\phi(x', t))} P\exp\bigg(\int_0^t ds \big(\LLL(s)
- ip\big( b(s) -{1 \over 2} \sigma(s)^2 a'(s) -\dot\phi(s) +O(h)
\big)\bigg)(x, x').\notag \\
\end{eqnarray}
The joint kernel is thus given by
\begin{eqnarray}
&\hskip-9.5cm P\exp\bigg(\int_0^t ds \tilde \L(s)\bigg)(x, I; x', I') = \notag \\
&\int {dp\over 2\pi} e^{ip [ I' - I -\phi(x', t) + \phi(x, 0)]} P
\exp\bigg(\int_0^t ds \big(\LLL(s) - ip \big(b(s) -{1 \over 2}
\sigma(s)^2 a'(s) -\dot\phi(s) +O(h)\big)\bigg)(x, x').\notag \\
\end{eqnarray}
This formula proves both statements in the Theorem.
\end{proof}

When in Definition (\ref{def_measure_change}) we introduced the
notion of measure change, convergence and smoothness in the
continuum limit of the measure change function $G_y(y'; t)$ was not
emphasized. However, this concept is important, as it is stressed by
Girsanov's theorem below, of which we give two independent proofs.

\begin{theorem} {\bf (Girsanov.)} Consider two lattice diffusions
on the simplicial sequence $A_m$ with generators
\begin{equation}
{\LLL_m} = \mu(x, t) \nabla_h + {\sigma(x, t)^2\over 2} \Delta_h,
\;\;\; \bar{\LLL_h} = \bar\mu(x, t) \nabla_h + {\sigma(x, t)^2\over
2} \Delta_h,
\end{equation}
respectively. Here $\mu(x, t), \bar\mu(x, t), \sigma(x, t)$ are
assumed to be smooth functions. Then the Markov generators $\LLL_m$
and $\bar \LLL_m$ are related by the smooth measure change with
function
\begin{equation}
G_m^{xt}(x') =  \exp\left({ (\bar \mu(x, t) - \mu(x, t))\over
\sigma(x, t)^2}(x'-x)\right)
\end{equation}
and the Radon-Nykodym derivative is given by
\begin{eqnarray}
\rho(x_\cdot) = \exp\bigg\{ \int_T^{T'} {\bar\mu(x_t, t)-\mu(x_t,
t)\over \sigma(x_t, t)^2} dy_t - {\bar\mu(x_t, t)^2 - \mu(x_t,
t)^2\over 2 \sigma(x_t, t)^2} dt + O(h)\bigg\}.
\end{eqnarray}
\end{theorem}

\noindent{\it First Proof. } Let $v(x, t) = \bar\mu(x, t) - \mu(x,
t)$ and let $a(x, t) = {v(x, t) \over \sigma(x, t)^2}$. Consider the
measure change with
\begin{equation}
G_m^{xt}(x') =  e^{ a(x, t) (x'-x)}.
\end{equation}
The off-diagonal elements of the transformed Markov generator are
\begin{equation}
\bar \LLL(x, x\pm h; t) = \left({\sigma(x, t)^2\over 2 h^2} \; \pm
\; {\mu(x, t)\over 2 h} \right)e^{\pm {v(x, t) \over \sigma(x, t)^2}
h} = {\sigma(x, t)^2\over 2 h^2} \pm {\bar \mu(x, t)\over 2 h} +
{v(x, t)^2\over 4 \sigma(x, t)^2} + {\mu(y, t) v(x, t)\over 2
\sigma(x, t)^2} + O(h).
\end{equation}
The diagonal elements instead change as follows
\begin{equation}
\bar \LLL(x, x; t) = \LLL(x, x; t) - {v(x, t)^2\over 2 \sigma(x,
t)^2} - {\mu(x, t) v(x, t)\over \sigma(x, t)^2} + O(h).
\end{equation}

\noindent{\it Second Proof. } Let us
consider two generators differing by the drift term
\begin{equation}
\LLL(t)  = {\sigma(x, t)^2\over 2}\Delta_h + \mu(x, t) \nabla_h,
\hskip1cm \bar\LLL(t) = {\sigma(x, t)^2\over 2}\Delta_h + \bar\mu(x,
t) \nabla_h
\end{equation}
and consider the formula
\begin{eqnarray}
&\hskip-8cm E\big[e^{\int_0^t a(x_s, s) d x_s + b(x_s, s) ds }
\delta(x_t = y)
\lvert x_0 = x\big] = \notag \\
 &e^{(\phi(y, t) - \phi(x, 0))}
P\exp\bigg(\int_0^t ds \bigg(\LLL(s) + b(s) - {1\over 2} \sigma(s)^2
a'(s) + O(h)\bigg)\bigg)(x, y). \notag \\
\end{eqnarray}
To derive Girsanov's theorem, let us ask for what choice of the
functions $a(x, t), b(x, t)$ the right-hand side equals $e^{t
\bar\LLL}(x, y; t)$ up to terms of order $O(h)$. Since
\begin{equation}
e^{\phi} \bigg(\LLL + b - {1\over 2} \sigma^2 a' \bigg)e^{-\phi} =
{\sigma(x)^2\over 2}\Delta_h + (a(x) \sigma(x)^2 + \mu(x)) \nabla_h
 + {\sigma(x)^2\over 2} a(x)^2 + \mu(x) a(x) + b(x) +  O(h)
\end{equation}
we see that the right hand side equals $\bar\LLL$ up to terms of
order $O(h)$ if
\begin{equation}
a(x) = {\bar\mu(x)-\mu(x)\over\sigma(x)^2}, \hskip3cm b(x) =
{\mu(x)^2 - \bar\mu(x)^2 \over 2 \sigma(x)^2}.
\end{equation}

\section{Markov Bridges}

A first simple application of this general framework leads to an
extension of the results in the previous section to the case of a
general Markov generator.

Consider a Markov process on the simplicial sequence $A_m = h_m \ZZZ
\cap [-L, L] \subset \RRR$. According to theorem (\ref{LKrep}) and
assuming that hypothesis {\it MG1, MG2, MG3, MG4} hold, the symbol
of a Markov generator in canonical form can be expressed as follows:
\begin{equation}
\hat {\mathcal L}_m(x, p; t) =  i \tilde \mu(x; t) {\sin ph\over h}
+ \tilde \sigma(x; t)^2 {\cos ph -1\over h^2} + \sum_{y\in A_m}
\left(e^{ i p (y-x)} - 1 - i {\sin p h \over h} (y-x) \right) \tilde
\lambda(x, y; t). \notag \\
\end{equation}
Although not necessary for the validity of the calculation and the
convergence in the limit $h\to0$, we restrict to the special case
where the generator has the form
\begin{equation}
{\mathcal L}_m(x, y; t) =  i \mu(x; t) \nabla_h + \sigma(x; t)^2
\Delta_h + h_m \lambda_m(x, y; t)
\end{equation}
where the functions $\mu(x;t)$ and $\sigma(x;t)$ are smooth in both
arguments and $\lambda_m(x, y; t) \ge 0$ is smooth and non-negative
for $x\neq y$ and we have that
\begin{equation}
\lambda_m(x, x; t) = - \sum_{y\neq x \in A_m} \lambda_m(x, y; t).
\end{equation}

Consider a stochastic integral of the form
\begin{equation}
I_t = \int_0^t a(x_s, s) dx_s + b(x_s, s) ds
\end{equation}
where $a(x, t)$ and $b(x, t)$ are smooth functions in both
arguments. Let $\phi(x, t) = \int^x_0 a(y, t) dy$. The following
result holds:
\begin{theorem} {\bf (Markov Bridges.)}
The joint probability distribution function for $I_t$ and the
underlying process on the bridge leading from $x$ to $x'$ is given
by
\begin{eqnarray}
&\hskip-10cm P\exp\bigg(\int_0^t\tilde \L_m(s) ds \bigg) (x, I; x',
I') \notag
\\ &= \int {dp\over 2\pi} e^{ip [ I' - I -\phi(x') + \phi(x)]}
P\exp\bigg(\int_0^t \big(\LLL_m(s) - ip b(s) +{ip \over 2} \sigma^2
a'(s) -\dot \phi(s) + \kappa_{m, p}(s) + O(h) \big) ds \bigg)(x,
x'). \notag
\\ \label{eq_jointk}
\end{eqnarray}
where $\kappa_{m, p}(t)$ is the operator of matrix elements
\begin{equation}
\kappa_{m, p} (x, x'; t) = h_m \lambda_m(x, x'; t) \exp \bigg(- i p
\big[ a(x, t) (x'-x) + \phi(x, t) - \phi(x', t) \big] \bigg).
\end{equation}
\end{theorem}
\begin{proof} Consider again a discretization of the process $I_t = (\Delta
I) n_t$ and the lifted generator
\begin{eqnarray}
\tilde\LLL_m(x, n; x', n') = \LLL_m(x; x') \delta\bigg(n' - n
-\bigg[{a(x) (x'-x) \over (\Delta I)}\bigg]\bigg) +
{b(x)\over(\Delta I)}\big( \delta_{n', n+1} -
\delta_{n' n} \big) \notag \\
\end{eqnarray}
where $[a]$ stands for the nearest integer to $a$. The partial
Fourier transform in the $n$ variable of this kernel is
\begin{eqnarray}
\tilde\LLL_{m, p}(x, x') &=& \sum_{n} \tilde\LLL(x, 0; x', n)
e^{- i (\Delta I) p n } \notag \\
&=& \LLL_m(x; x') \exp\bigg(- i (\Delta I) p \bigg[{a(x) (x'-x)
\over (\Delta I)}\bigg] \bigg) + {b(x)\over(\Delta
I)}\big(e^{ip(\Delta I) } - 1\big)
\end{eqnarray}
In the limit as $(\Delta I) \to 0$ we find
\begin{eqnarray}
\tilde\LLL_{m, p}(x, x') = {\sigma(x)^2\over 2} \Delta_h + \left(
\mu(x) - i p \sigma(x)^2 a(x)\right) \nabla_h - ip a(x) \mu(x) -
{1\over2}
p^2\sigma(x)^2 a(x)^2 - i p b(x)  \notag \\
+ h_m \lambda_m (x, x'; t) e^{- i p a(x) (x'-x)}+O(h) \notag \\ .
\end{eqnarray}
Introducing the function $\phi(x, t)$ defined above, we find that
\begin{eqnarray}
e^{ip\phi(x, t)} \LLL_m(x, x'; t) e^{-ip\phi(x', t)} =
\tilde\LLL_{m, p}(x, x'; t) + ip b(x, t)\delta_{x x'} -{ip \over 2}
\sigma(x, t)^2 a'(x, t) \delta_{x x'} \notag
\\
+ h_m \lambda_m(x, x'; t) \exp \bigg(- i p \big[ a(x, t) (x'-x) +
\phi(x, t) - \phi(x', t) \big] \bigg) + O(h_m).
\end{eqnarray}
Hence
\begin{eqnarray}
&\hskip-11.5cm \exp\big(t \tilde\LLL_{m, p}\big)(x, x'; t) \notag \\
&= ^{ip(\phi(x, t)-\phi(x', t))} P\exp\bigg(\int_0^t \big(\LLL_m(s)
- ip b(s) +{ip \over 2} \sigma(s)^2 a'(s) -\dot \phi(s)  +
\kappa_{m, p}(s) +
O(h_m) \big) ds\bigg)(x, x'). \notag \\
\end{eqnarray}
The joint kernel is thus given by (\ref{eq_jointk}).
\end{proof}

\section{Abelian Processes in Continuous Time}{\label{abelian_contt}

This section is based on work in collaboration with Harry Lo and
Alex Mijatovic, see \cite{ALoMijatovic}.

Abelian processes are path-dependent processes which provide an
extension of the notion of stochastic integral for which one can
extend the Feynman-Kac theorem and Ito's formula in Section
(\ref{sec_feynman}).

Let $A_m, m\ge m_0$ be a simplicial sequence, consider a time
interval $[T, T']$ and a stochastic process described by the Markov
generator $\L_m(y_1; y_2; t)$. To numerically exponentiate, it is
crucial in most application examples to first reduce dimensionality
by means of block diagonalisations, or else the multiplication of
the lifted generators would not be feasible. It turns out that this
is possible for the payoffs of the type above as these two fall in
the general class outlined by the following definition:
\begin{definition} {\bf (Markov Generator.)}
 Let $\L(y_1; y_2; t)$ be a Markov generator
and let us consider a lifting of the form
\begin{equation}
\bar \L(y_1, k_1; y_2, k_2; t) = \L(y_1; y_2; t)\delta_{k_1 k_2} +
\A(y_1, k_1; y_2,  k_2 ; t).
\end{equation}
where $k = 0 ... K$. This lifting is called Abelian if the operators
$A(y_1; y_2; t)$ defined for each pair $y_1, y_2\in \Lambda$ and all
times $t\in[T, T']$ as the linear operators of matrix elements
\begin{equation}
A(y_1; y_2; t)_{k_1, k_2} \equiv \A(y_1, k_1; y_2,  k_2 ; t)
\end{equation}
are mutually commuting in the sense that
\begin{equation}
A(y_1, y_2) A(y_3, y_4) = A(y_3, y_4) A(y_1, y_2)
\end{equation}
for all $y_1, y_2, y_3, y_4\in\Lambda$ and all $t\in[T, T']$.
\end{definition}

\begin{lemma}
If all the matrices $A(y_1; y_2; t)_{k_1, k_2}$, $y_1,
y_2\in\Lambda$, $t\in[T, T']$ are mutually commuting and if
furthermore they are all diagonalizable, then they can be
diagonalized simultaneously at any given point in time, i.e. there
exists a time dependent matrix $V(k, i; t), i=0, ..K$ such that
$V(t)^{-1} \A(y_1; y_2; t) V(t) = \Lambda(y_1; y_2; t)$, where
$\Lambda(y_1; y_2; t)$ is a diagonal matrix of the form
$(\lambda(y_1; y_2; t)\delta_{i_1, i_2})$ for any $t\in[T, T']$.
\end{lemma}
\begin{proof} Fix a $t\in[T, T']$.
If $y_1, y_2, y_3, y_4 \in\Lambda$ and $ A(y_1; y_2; t) \psi(t) =
\lambda(y_1; y_2; t) \psi(t)$ for some vector $\psi(t)$, then
\begin{equation}
A(y_3; y_4; t) A(y_1; y_2; t) \psi = A(y_1; y_2; t) A(y_3; y_4; t)
\psi(t) = \lambda(y_1; y_2; t) A(y_3; y_4; t) \psi(t).
\end{equation}
Hence if $\psi(t)$ is an eigenvector of $ A(y_1; y_2; t) $ also
$A(y_3; y_4; t) \psi(t)$ is an eigenvector of $ A(y_1; y_2; t) $. If
$\psi(t)$ is an eigenvector of multiplicity one, this shows that
$\psi(t)$ is also an eigenvector of $A(y_3; y_4; t)$ for all $y_3,
y_4$. Otherwise, we conclude that the eigenspace of $A(y_1; y_2; t)$
of eigenvalue $\lambda$ is invariant under $A(y_3, y_4)$ for any
$y_3, y_4$. Iterating the argument above to this eigenspace one can
constructively obtain a common set of invariant eigenspaces shared
by all the operators $\A(y_1; y_2; t)$ for $y_1, y_2\in\Lambda$.
Hence, for any given $t\in[T, T']$, these operators can be
diagonalized simultaneously by the matrix whose column vectors span
bases for all the common eigenspaces.
\end{proof}

Let $V(k, i; t), i=0, ..K$ be a matrix which diagonalizes
simultaneously all members of the family of operators $T(k_1, k_2;
y_1 t)$, so that $V(t)^{-1} T(y_1; t) V(t) = \Lambda(y_1; t)$ is
diagonal. Consider the lifted matrix $\V(k, y_1; i, y_2; t) = V(k;
i; t) \delta_{y_1 y_2}$ and the transformed lifted Markov generator
\begin{equation} (\V(t)^{-1} \bar \L \V(t) )(y_1, i_1; y_2, i_2; t) =
\L(y_1; y_2; t)\delta_{i_1 i_2} + \Lambda(y_1; t)_{i_1} \delta_{y_1,
y_2}\delta_{i_1 i_2}.
\end{equation}
Since this matrix is block-diagonal, the problem of exponentiating
it is reduced to the problem of exponentiating $K$ blocks
separately. This reduces the overall numerical complexity and makes
it comparable with the complexity of evaluating propagators on $K$
different time points. As a further simplification, to exponentiate
this block-diagonal matrix it is necessary to hold in memory only
the blocks, and not the entire matrix. Notice that the blocks are in
general complex matrices. Hence to fast exponentiate them one needs
to invoke the complex matrix-matrix multiplication routine ${\tt
sgemm}$ or ${\tt cgemm}$ of Level-3 BLAS, depending on whether the
block-diagonalizing transformation is real or complex valued.

\begin{example} {\bf (Stochastic Integrals)} \rm Stochastic integrals over diffusion processes provide
examples of Abelian process. This was already used in the previous
section but it is worthwhile here to stress the link between the
computability of the characteristic functions for the path dependent
process on a bridge with the block-diagonalizability under partial
Fourier transforms of the lifted generator. Working directly with
Markov generators, one can also generalize the results in the
previous sections.

Consider the integral
\begin{equation}
I_{t} = I(y_\cdot, t) \equiv \int_T^{t} \bigg( \phi(y_{s}; s) +
\chi(y_{s-0}, y_{s}; s) \lim_{t\downarrow0} {\psi(y_{s}; s) -
\psi(y_{s-t}; s-t) \over t } \bigg) ds \label{eq_defi}
\end{equation}
where $\phi(y; s)$, $\psi(y; s)$ and $\chi(y_1; y_2; s)$ are
continuous functions. Consider the problem of finding its
distribution on a bridge for the underlying Markov process $y_t$.
Introducing the discretisation
\begin{equation}
I_t \approx (\Delta I) m_t
\end{equation}
where $m_t\in[0,..N]$ and $(\Delta I)>0$, the lifted generator for
the pair of processes $(y_t, (\Delta I) m_t)$ is
\begin{align}
&\bar \L(y_1, m_1; y_2, m_2; t) = \L(y_1; y_2; t) \delta\big(m_2 -
m_1 - [ (\Delta I)^{-1} \chi(y_1; y_2; t) (\psi(y_2; t)-\psi(y_1;
t)) ] \big) \notag \\
&+ {\delta_{y_1, y_2} \over (\Delta I)} \phi(y_1; t) \delta_{m_1+1,
m_2}. \notag \\
\end{align}

\begin{theorem} {\bf (Abelian Bridges.)}
The joint kernel of $I_t$ and the underlying process on the bridge
leading from $y_1$ to $y_2$ is given by
\begin{eqnarray}
Pe^{\int_0^t \tilde \L_m(s) }(y_1, I_1; y_2, I_2; s) = \int {dp\over
2\pi} e^{ip ( I_2 - I_1) } P\exp\bigg(\int_0^t \big(\LLL_m(s) - ip
\phi(s) + \kappa_{m,p}(s) + O(h) \big)\bigg)(x, x'). \notag
\\ \label{eq_jointk}
\end{eqnarray}
where $\kappa_{m, p}(t)$ is the operator with matrix elements
\begin{equation}
\kappa_{m, p} (y_1, y_2; t) = \L_m(y_1, y_2; t) \bigg( \exp \bigg(-
i p
 \chi(y_1, y_2, t) (\psi(y_2, t)-\psi(y_1, t))  \bigg) -1 \bigg).
\end{equation}
\end{theorem}

Notice that this theorem extends the result in the previous section
on Markov bridges as it includes multifactor processes such as
processes with stochastic volatility and, more generally, regime
switching.
\end{example}

\begin{example} {\bf (Double Liftings)} \rm There are situations that emerge in
practice where one has to track two integrals over paths, one of
form (\ref{eq_defi}) and a similar one of form
\begin{equation}
I_{t}' = I'(y_\cdot, t) \equiv \int_T^{t} \bigg( \phi'(y_{s}, s) +
\chi'(y_{s-0}, y_{s}, s) \lim_{t\downarrow0} {\psi'(y_{s}, s) -
\psi'(y_{s-t}, s-t) \over t } \bigg) ds \label{eq_defip}
\end{equation}
In this case one can introduce a similar approximation
\begin{equation}
I'_t \approx \bar I'_t = (\Delta I') n_t
\end{equation}
where $n_t\in[0,..N]$ and a double lifting for the generator
\begin{align}
& \bar \L(y_1, m_1, n_1; y_2, m_2, n_2; t) = \notag \\
& \L(y_1; y_2; t) \delta\bigg(m_2 - m_1 - \bigg[ { \chi(y_1; y_2; t)
(\psi(y_2; t)-\psi(y_1; t)) \over (\Delta I) }\bigg] \bigg)
\delta\bigg(n_2 - n_1 - \bigg[{ \chi'(y_1; y_2; t) (\psi'(y_2;
t)-\psi'(y_1; t)) \over (\Delta I')}\bigg]
\bigg)   \notag \\
& \hskip4cm + {\delta_{y_1, y_2} \over (\Delta I)} \phi(y_1; t)
\delta_{m_1+1, m_2} \delta_{n_1, n_2} + + {\delta_{y_1, y_2} \over
(\Delta I')} \phi'(y_1; t) \delta_{n_1+1, n_2} \delta_{m_1, m_2}.
\notag \\
\end{align}

\begin{theorem} {\bf (Multifactor Bridges.)}
The joint kernel of $I_t, I_t'$ and the underlying process on the
bridge leading from $y_1$ to $y_2$ is given by the double Fourier
transform
\begin{eqnarray}
&\hskip-5cm e^{t\tilde \L_m }(y_1, I_1; y_2, I_2) = \int {dp\over
2\pi} e^{ip (
I_2 - I_1) } \int {dp'\over 2\pi} e^{ip' ( I_2' - I_1') } \notag \\
&P\exp\bigg(\int_0^t \big(\LLL_m(s) - ip \phi(s) - ip' \phi'(s) +
\kappa_{m,p, p'} (s) + O(h) \big)\bigg)(x, x'). \notag
\\ \label{eq_jointk}
\end{eqnarray}
where $\kappa_{m, p, p'}(t)$ is the operator with matrix elements
\begin{equation}
\kappa_{m, p, p'} (y_1; y_2; t) = \L_m(y_1; y_2; t) \big( e^{ - i p
 \chi(y_1; y_2; t) (\psi(y_2; t)-\psi(y_1; t))} -1 \big)
 \big( e^{- i p'
 \chi'(y_1; y_2; t) (\psi'(y_2; t)-\psi'(y_1; t))} -1 \big).
\end{equation}
\end{theorem}
\end{example}

\begin{example} {\bf (The Max Process)} \rm Consider the path dependent quantity given by the
maximum attained by a given Markov process, i.e.
\begin{equation}
K_t = K(y_t, t) \equiv \max_{[T,t]} \chi(y_s; s) ds \label{eq_defk}
\end{equation}
Let's introduce the approximation
\begin{equation}
K_t \approx \bar K_t = \alpha k_t, \label{eq_idef}
\end{equation}
where $\alpha$ is a constant and $k_t$ is an integer value process
$m_t$ whose paths take values in $k_t \in\{ 0 ... M\}$. The dynamics
for the joint process $(y_t, k_t)$ is defined by the lifted
generator $\bar \L$ on $\bar \Lambda$ such that:
\begin{equation}
\bar \L(y_1, k_1; y_2, k_2; t) = \L(y_1; y_2; t) \delta_{k_1, k_2} +
\A(y_1)_{k_1, k_2}\delta_{y_1 y_2}
\end{equation}
where
\begin{equation}
\A(y_1)_{k_1, k_2} =
\begin{cases}
A\;\;\;\; {\rm if} \;\;\; \chi(y_1) > \alpha k_1 \;\;\;\; {\rm and }\;\;\;  k_2 = \left[\chi(y_1)/\alpha\right] \\
0 \;\;\;\; {\rm otherwise}.
\end{cases}
\end{equation}
where $[a]$ stands for the nearest integer to $a$ and $A>0$ is a
fixed number. Typically, $A$ is chosen to be a large number as the
approximation converges in the limit at $A\to\infty$ and $\alpha\to
0$. A direct calculation shows that the operators $\A(y_1)_{k_1,
k_2}$ commute and hence the maximum of the underlying process is
itself an Abelian path dependent process.
\end{example}

\section{Abelian Processes in Discrete Time and non-Resonance Conditions}{\label{abelian_contt}

This section is based on work in collaboration with Manlio Trovato
\cite{ATrovato2} and Paul Jones, see \cite{AJones}.

An important class of path-dependent options requires computing the
joint distribution of the underlying lattice process and of a
discrete sum of the following form:
\begin{equation}
J_t = J(y_t, t) \equiv \sum_{i=1}^n \psi(y_{t_{i-1}}, y_{t_{i}};
t_i) \label{eq_defi2}
\end{equation}
where $n$ is an integer, $t_i = i \Delta T$ and $t = n \Delta T$.
Suppose that the Markov generator $\L$ is time homogenous in the
interval $[t_0, t_n]$ and consider the elementary propagator
\begin{equation}
U_{i}(y_1, y_2) = P\exp\bigg(\int_{t_i}^{t_{i+1}}  \L(s) ds
\bigg)(y_1, y_2).
\end{equation}
To find the joint transition probability, one can again discretize
the variable $J_t$ so that
\begin{equation}
J_t \approx (\Delta J) n_t.
\end{equation}
As opposed to lifting the generator as in the continuous time case
discussed in the previous section, here we lift the elementary
propagator itself and form the joint propagator
\begin{equation}
\tilde U_{i}(y_1, k_1; y_2, k_2) = U_{i}(y_1, y_2) \delta\big(k_1 -
k_2 + [ \psi_i(y_1, y_2) (\Delta J)^{-1}]\big).
 \label{eq_defu}
\end{equation}
This operator can be block-diagonalized by means of a partial
Fourier transform \cite{ATrovato2}.

\begin{theorem} Consider the Fourier transform operator $\bar U_{i, p}$ of
matrix elements
\begin{equation}
\tilde U_{i, p}(y_1 ; y_2) = U_{i}(y_1, y_2) e^{ -i p \psi_i(y_1,
y_2)}.
 \label{eq_defup}
\end{equation}
Then we have that
\begin{equation}
\bigg( P \prod_{i=0}^{n-1} \tilde U_{i}\bigg)(y_1, J_1; y_2, J_2) =
\int {dp\over 2\pi} e^{ip(J_2-J_1)} \bigg( P \prod_{i=0}^{n-1}
\tilde U_{i, p}\bigg)(y_1, y_2) \label{discretefk1}
\end{equation}
\end{theorem}

This case also falls under a more general class of Abelian liftings.
\begin{definition} {\bf (Propagator Liftings.)}
 Let $U_{i}(y_1, y_2)$ be a family of Markov
propagators indexed by $i = 0 , 1, ... n-1$ defined over the time
intervals $[t_i, t_{i+1}]$. Consider a propagator lifting of the
form
\begin{equation}
\bar U_{i}(y_1, k_1; y_2, k_2) = U_{i}(y_1, y_2) Q_i(y_1, y_2)_{k_1
k_2}.
\end{equation}
where $k_1, k_2 = 0 ... K$. This lifting is called Abelian if the
operators $Q_i(y_1, y_2)$ satisfy the following commutation
condition:
\begin{equation}
Q_i(y_1, y_2) Q_{i}(y_1', y_2') - Q_{i}(y_1', y_2')Q_i(y_1, y_2) = 0
\end{equation}
for all $y_1, y_2, y_3, y_4 \in A_m$ and all time intervals
$i=0,...n-1$.
\end{definition}

Let $V(k, j), j=0, ..K$ be a matrix that simultaneously diagonalizes
all operators of the family  $Q(y_1, y_2)_{k_1 k_2}$ so that $V^{-1}
Q(y_1, y_2) V = \Lambda(y_1, y_2)$ is diagonal. Consider the lifted
matrix $\V(k, y_1; j, y_2) = V(k, j) \delta_{y_1 y_2}$ and the
transformed lifted Markov generator
\begin{equation}
(\V^{-1} \bar U_{i} \V )(y_1, j_1; y_2, j_2) = \bar U_{i}(y_1; y_2)
\Lambda_{i} (y_1, y_2)\delta_{j_1 j_2} .
\end{equation}
Since this matrix is block-diagonal, the evaluation of the
time-ordered product
\begin{equation}
P \prod_{i=0}^{n-1} \bar U_{i} = \V  \bigg( P \prod_{i=0}^{n-1}
(\V^{-1} \bar U_{i} \V ) \bigg) \V^{-1} \label{discretefk}
\end{equation}
involves only multiplying blocks whose dimension equals the size of
the lattice $A_m$. The reader will recognize that the formula in
(\ref{discretefk}) is yet another generalization and extension of
the Cameron-Feynman-Girsanov-Ito-Kac-Martin formulas discussed
above.

The non-singular linear transformation that accomplished the
block-diagonalization can be a Fourier transform or a more general
transformation. Several possibilities are open and the optimal
choice depends on the objective. When using transforms more general
than the Fourier transform, one has to keep present that the
simultaneous diagonalization of the matrices $Q(y_1, y_2)_{k_1k_2}$
above can possibly result in a numerically ill-conditioned problem.
An example of this phenomenon arises when one attempts to use a
non-homogeneous discretization of the path process coordinate. To
seize the benefit of inhomogeneous lattices, one needs to be careful
when discretizing and avoid resonances.

Consider a lattice for the path dependent process defined as follows
$\Omega = \{ \omega_1, \omega_1, .... \omega_K \}$ where
$\omega_1<\omega_2<...<\omega_K$. Consider the shift operator
${\mathcal R}$ of matrix elements
\begin{equation}
R_{k_1 k_2} = (1-p_{k_1}) \delta_{k_1 k_2} + p_{k_1}
\delta_{k_2,k_1+1},
\end{equation}
if $k_1<k_2$ and $R_{KK} = 0$. Here we assume that $0<p_k<1$, for
$k=1,..n$. The eigenvalues of ${\mathcal R}$ are given by the
diagonal elements $\rho_k = 1 - p_{k}$. Diagonalizing this sort of
matrix is potentially a seriously ill-conditioned especially if the
matrix is large and the eigenvalues are degenerate of nearly
degenerate. For practical application, one thus needs to achieve non
resonance conditions and keep the eigenvalues $\rho_i$ as far apart
from each other as possible.

Given a non-resonant choice of transition probabilities such as the
one above, one can then fix a $\omega_0$ and a $\Delta \omega > 0$
and determine the lattice $\Omega$ in such a way that
\begin{equation}
\sum_{k_2=1}^K  R_{k_1 k_2} \omega_{k_2} = \omega_{k_1} + \Delta
\omega
\end{equation}
for all $k_1 = 1, .... K-1$.

To obtain a non-resonant spectrum, one may choose the lattice
$\omega_k$ so that
\begin{equation}
\omega_k = \omega_0 \cdot Z^k
\end{equation}
where $\omega_0>0$ and $Z>1$. For this discretization to be
acceptable, the resulting transition probabilities
\begin{equation}
p_k = {\Delta \omega \over \omega_0 Z^k(Z-1)}
\end{equation}
ought to be between 0 and 1. Hence $Z$ must satisfy the additional
constraint
\begin{equation}
Z (Z-1) > {\Delta \omega \over \omega_0}.
\end{equation}

We interpolate between the transition kernels $\mathcal{R}$ to
obtain the relevant kernels of all possible values of $\psi(x_1,
x_2)$. We construct the relevant kernels $\Lambda_{\psi}$ as
follows:
\begin{eqnarray}
Q(y_1, y_2) = ( [\psi(y_1,y_2)]+1-\psi(y_1,y_2) ){\mathcal
R}^{[\psi(y_1,y_2)]}+ ( \psi(y_1,y_2) - [\psi(y_1,y_2)] ){\mathcal
R}^{[\psi(y_1,y_2)]+1}, \notag \\
\end{eqnarray}
for all $y_1$ and $y_2$ in $A_m$, where $\left[ \psi(y_1,y_2)
\right]$ represents the integer part of the functional $\psi$ at
$(y_1, y_2)$. In the representation in which the operator ${\mathcal
R}$ is diagonal, the operators $Q(y_1, y_2)$ are also diagonal and
the lifted propagator is block-diagonal.

\section{Univariate Moment Expansions on Bridges}\label{sec:moment1}

This section is based on work in collaboration with Adel Osseiran,
see \cite{AOsseiran}.

A second methodology that applies to most Abelian path dependent
options is is based on obtaining only a few moments of an Abelian
process on any given bridge with respect to the underlying Markov
process, as opposed to reconstructing the entire conditional
distribution.

Consider a time interval $[T, T']$ and a Markov generator $\L(y_1,
y_2; t)$. Consider again the integral $I_t$ in (\ref{eq_defi}).
Let's introduce the following one parameter family of deformed
Markov operators parameterized by $\epsilon\in\RRR$
\begin{equation}
\L_\epsilon(y_1, y_2; t) = \L(y_1, y_2; t) + \epsilon V(y_1, y_2; t)
\end{equation}
where
\begin{equation}
V(y_1, y_2; t) = \phi(y_1; t) \delta_{y_1, y_2} + \L(y_1, y_2; t)
\chi(y_1, y_2; t)(\psi(y_2; t)-\psi(y_1; t))
\end{equation}

\begin{theorem} {\bf (Dyson Formula.)}
We have that
\begin{equation}
\left({d\over d\epsilon}\right)^n\bigg\lvert_{\epsilon=0}
P\exp\bigg(\int_T^{t} \L_\epsilon(s) ds \bigg)(y_1, y_2) = E_T\big[
I_t^n \delta(y_{t} - y_2) \lvert y_{T} = y_1
 \big].
\end{equation}
\end{theorem}

\begin{proof} Consider Neper's formula for the propagator
\begin{equation}
P \exp\bigg(\int_T^t \L_\epsilon(s) ds \bigg) =
\lim_{N\to\infty} P \prod_{i=1}^N \left( 1 + {t - T\over N} (\L(t_i) +
\epsilon \phi(t_i)) \right)^N
\end{equation}
where $t_i = T + {i t \over N}$.  By collecting similar powers in
$\epsilon$, one finds Dyson's formula
\begin{align}
P \exp\bigg(\int_T^t \L_\epsilon(s) ds \bigg)& = \exp((t-T)\L + \\
& \epsilon \int_T^{t} ds_1
\bigg( e^{(s_1-T)\L} V(s_1) e^{(t-s_1)\L}\bigg) +\\
& \sum_{n=2}^\infty \epsilon^n \int_T^{t} ds_1 ...
\int_{s_{n-1}}^{t} ds_n  \bigg( e^{(s_1-T)\L} V(s_1) e^{(s_2-s_1)\L}
.... V(s_{n}) e^{(t-s_{n})\L}\bigg).
\end{align}
The time-ordered integrals above are proportional to conditional moments, i.e.
\begin{equation}
P\exp\bigg(\int_T^t \L_\epsilon(s) ds \bigg)(y_1, y_2) =
\sum_{n=0}^\infty {\epsilon^n\over n!} E_T\big[ I_t^n
\delta\big(y_{t} - y_2\big) \lvert y_T = y_1\big].
\end{equation}
Here, the factorials originate from the time ordering.
\end{proof}

A technique which is numerically stable in many situations of
practical relevance is to numerically differentiate with respect to
$\epsilon$ the deformed propagators $P \exp\bigg(\int_T^t \L_\epsilon(s) ds \bigg)(y_1, y_2)$
and evaluate at $\epsilon=0$. This technique can be used to obtain
the first moments of $I_t$ on any given bridge for the underlying
Markov process. In most applications, we find that two moments
suffice to extrapolate the probability distribution function to
sufficient accuracy. To do so, it is convenient to choose from among
the probability distribution functions which are analytically
tractable. For instance, starting from the first two moments only,
one can use either the log-normal or the chi-square distribution.

The standard chi-square distribution is given by
\[ f(x) = \frac{1}{2 \Gamma \left( \frac{a}{2} \right) }
\left( \frac{x}{2} \right)^{a/2 - 1} e^{-x/2} \]
where $a$ is the number of degrees of freedom. The first and second (raw)
 moments of this distribution are
\[ E[x] = a  \, , \,\,\,\,\,\,\,\,\,\,\,\,\, E[x^2] = a(a+2) \]
To match the pre-assigned first and second moments $m_1$, $m_2$,
respectively, one can pass to the new variable
\[ \xi = \frac{m_1}{a} x \]
and chose
\[ a = \frac{2m_1^2}{m_2 - m_1^2} \]

Let $\xi$ be a log-normally distributed random variable with
probability distribution function
\[ f(x;\mu,\sigma) = \frac{1}{x \sigma \sqrt{2 \pi}}
e^{-(\ln x - \mu)^2 / 2 \sigma^2} \] The first two moments are given
by
\[ E[x] = e^{\mu + \sigma^2/2} \,\,\,\,\,\,\,\,\,\,\,
\textrm{and} \,\,\,\,\,\,\,\,\,\,\, E[x^2] = e^{2 \mu + 2 \sigma^2}
\] The parameters
$\mu$ and $\sigma$ can be reconstructed from two pre-assigned first
and second moments, $m_1$ and $m_2$ respectively, by setting
\[ \mu = \log \left( \frac{m_1^2}{\sqrt{m_2}} \right)
\,\,\,\,\,\,\,\,\,\, \textrm{and} \,\,\,\,\,\,\,\,\,\,\, \sigma^2 =
\log \left( \frac{m_2}{m_1^2} \right). \]

The log-normal and the chi-square distributions allow one to match
the first two moments. More general distributions available in
closed form allow one to match higher moments also. The Pearson Type
III distribution for instance has a probability distribution
function given by
\[ f(x) = \frac{1}{b \Gamma(p)} \left( \frac{x - a}{b}
\right)^{p-1} e^{-(x- a)/b} \] and extends over the half line $[a,
+\infty)$. The special case of this distribution when $a = 0$, $b =
2$ and $p$ is half of an integer, gives the Chi-Squared
distribution. In general, the moments are given by
\[ E[x] = a + b p \]
\[ E[x^2] = (a + b p)^2 + b^2 p \]
\[ E[x^3] = (a + b p)^3 + 3 b^2 p (a + b p)
+ 2 b^3 p \] and matching these with our pre-assigned moments
$m_1,m_2$ and $m_3$ and computing the values of $a,b$ and $p$ in
terms of these moments we get
\[ a = m_1 - \frac{2(m_2 - m_1^2)^2}{m_3 + 2m_1^3 - 3m_1m_2} \]
\[ b = \frac{m_3 + 2m_1^3 - 3m_1m_2}{2(m_2 - m_1^2)} \]
\[ p =  \frac{4(m_2 - m_1^2)^3}{(m_3 + 2m_1^3 - 3m_1m_2)^2} \]

\begin{example} {\bf (Volatility contracts)} \rm As an example, consider variance and
volatility swaps. Variance swaps of maturity $T'$ and time of
issuance $T$ have a payoff given by
\[ \mathrm{min} \left( \int_{T}^{T'} v(y_{s-0}, y_{s})  ds,  f \cdot SR^2 \right) - SR^2
\] where $SR$ is the swap rate and $f$ is a factor, a typical value
being $f = 6.2$. Here, $v(y_1, y_2, t)$ is the instantaneous
variance defined as follows:
\begin{equation}
v(y_1, y_2) = {\mathcal{L}} (y_1,y_2;t) \log^2 \left(
\frac{S(y_2)}{S(y_1)} \right)
\end{equation}
if $y_1$ is such that $x(y_1) \neq 0$. Otherwise, if $x(y_1) = 0$
then $v(y_1, t) = \infty$. The variance swap is said to be at
equilibrium if its price is $0$. The payoff of a volatility swap is
\[ \mathrm{min} \left( \sqrt{\int_{T}^{T'} v(y_s,s) ds},SR \right) -
SR \] Here, $SR$ is the volatility swap rate.

To price these contracts it suffices to evaluate the distribution of
realized variance on a bridge, i.e. of the functional
\begin{equation}
RV(y_2) = \delta \left( y_{T'} - y_2 \right) \int_{T}^{T'} v(y_s,s)
ds
\end{equation}
By approximating the distribution of $RV(y)$ with the chi-squared
distribution, and using the CDF
 \[ F(x;a) = \frac{\gamma \left(\frac{a}{2},\frac{x}{2} \right)}
{\Gamma \left(\frac{a}{2} \right)} \]
where $\Gamma(z)$ and $\gamma(z,a)$ are the gamma and incomplete
gamma functions respectively, i.e:
\[ \Gamma(z) = \int_{0}^{\infty} s^{z-1} e^{-s} ds \,,\,\,\,\,\,
\gamma(z,a) =  \int_{0}^{a} s^{z-1} e^{-s} ds \] we find
\[ E_t \left[ \mathrm{min} (RV,RV_{\mathrm{max}})
\, \delta(y_{T'} - y_2)  \right] =
\frac{m_1}{a}  \left[ K(1 - F(K;a)) + a F(K;a +2) \right] \]
and
\[ E_t \left[ \mathrm{min} (\sqrt{RV},\sqrt{RV_{\mathrm{max}}})
\, \delta(y_{T'} - y_2)  \right] =
\sqrt{\frac{m_1}{a}} \left[ \sqrt{K} (1- F(K;a))
+ \frac{\sqrt{2} \gamma \left(\frac{a+1}{2},\frac{K}{2} \right)}
{\Gamma \left(\frac{a}{2} \right)} \right] \]
where
\[ a = a(y_1,y_2) = \frac{2m_1 (y_1,y_2)^2}{m_2 (y_1,y_2) - m_1(y_1,y_2)^2}
\,,\,\,\,\, K = K(y_1,y_2) = \frac{a(y_1,y_2)}{m_1(y_1,y_2)}
RV_{\mathrm{max}} \] Since the dependency on the swap rate in both
cases is non-linear, an exact calculation requires using a root
finder.

Using instead the log-normal distribution to extrapolate, we find
\begin{eqnarray}
&\hskip-8cm E_t [ \mathrm{min} (RV,RV_{\mathrm{max}}) \,
\delta(y_{T'} -
y_2)] \notag \\
&= e^{\mu + \sigma^2/2} {\mathcal{N}} \left(
\frac{\log(RV_{\mathrm{max}}) - \mu - \sigma^2/2} {\sigma} \right) +
RV_{\mathrm{max}} \, {\mathcal{N}} \left(
\frac{\log(RV_{\mathrm{max}}) - \mu}{\sigma} \right). \notag \\
\end{eqnarray}
For volatility swaps instead, we find
\begin{eqnarray}
&\hskip-8cm  E_t \left[ \mathrm{min}
(\sqrt{RV},\sqrt{RV_{\mathrm{max}}}) \, \delta(y_{T'} - y_2) \right]
\notag \\
&= e^{\frac{1}{8} (4 \mu + \sigma^2)} {\mathcal{N}} \left(
\frac{\log(RV_{\mathrm{max}}) - \mu - \sigma^2/2} {\sigma} \right) +
\sqrt{RV_{\mathrm{max}}} \, {\mathcal{N}} \left(
\frac{\log(RV_{\mathrm{max}}) - \mu} {\sigma} \right) \notag \\
\end{eqnarray}
where $\mu$ and $\sigma$ are specified above in terms of $m_1$ and
$m_2$.

Finally, in case the Pearson Type 3 distribution is used, we can
express these expectations in terms of the chi-square cumulative
distribution function $F(x;a)$ as follows:
\[ E_t [ \mathrm{min} (RV,RV_{\mathrm{max}})
\, \delta(y_{T'} - y_2)] = (a+bp - RV_{\mathrm{max}})
F_{\mathrm{Chi}} \left(2p, 2\frac{RV_{\mathrm{max}}-a}{b} \right) \]
\begin{equation}
- b \left( \frac{RV_{\mathrm{max}}-a}{b} \right)^p
\frac{e^{-\frac{RV_{\mathrm{max}}-a}{b}}}{\Gamma(p)} + (2
RV_{\mathrm{max}} - a - bp).
\end{equation}
\end{example}

\section{Bivariate Moment Expansions on Bridges}\label{sec:moment1}

This section is based on work in collaboration with Adel Osseiran,
see \cite{AOsseiran}.

Consider a time interval $[T, T']$ and a Markov generator $\L(y_1,
y_2; t)$. We are interested in the joint distribution on any given
bridge for the underlying process of two stochastic integrals
\begin{equation}
I_{1, t} = I(y_\cdot, t) \equiv \int_T^{t} \bigg( \phi_1(y_{s}; s) +
\chi_1(y_{s-0}, y_{s}; s) \lim_{t\downarrow0} {\psi_1(y_{s}; s) -
\psi_1(y_{s-t}; s-t) \over t } \bigg) ds
\end{equation}
and
\begin{equation}
I_{2, t} = I(y_\cdot, t) \equiv \int_T^{t} \bigg( \phi_2(y_{s}; s) +
\chi_2(y_{s-0}, y_{s}; s) \lim_{t\downarrow0} {\psi_2(y_{s}; s) -
\psi_2(y_{s-t}; s-t) \over t } \bigg) ds.
\end{equation}
To handle this problem using the moment method, we introduce the
following two-parameter family of deformed Markov operators
parameterized by $\epsilon, \epsilon'\in\RRR$
\begin{equation}
\L_{\epsilon_1, \epsilon_2} (y_1, y_2; t) = \L(y_1, y_2; t) +
\epsilon_1 V_1(y_1, y_2; t) + \epsilon_2 V_2(y_1, y_2; t)
\end{equation}
where
\begin{equation}
V_1(y_1, y_2; t) = \phi_1(y_1; t) \delta_{y_1, y_2} + \L(y_1, y_2;
t) \chi_1(y_1, y_2; t)(\psi_1(y_2; t)-\psi_1(y_1; t))
\end{equation}
and
\begin{equation}
V_2(y_1, y_2; t) = \phi_2(y_1; t) \delta_{y_1, y_2} + \L(y_1, y_2;
t) \chi_2(y_1, y_2; t)(\psi_2(y_2; t)-\psi_2(y_1; t))
\end{equation}
\begin{theorem} {\bf (Dyson Formula.)}
We have that
\begin{equation}
\left({\partial^{n_1+n_2} \over
\partial \epsilon_1^{n_1} \partial \epsilon_2^{n_2} }\right)\bigg\lvert_{\epsilon_1=0, \epsilon_2 = 0}
P\exp\bigg(\int_T^{t} L_{\epsilon_1, \epsilon_2} (s) ds \bigg)(y_1,
y_2) = E_T\big[ I_{1, t}^{n_1} I_{2, t}^{n_2} \delta(y_{t} - y_2)
\lvert y_{t} = y_1
 \big].
\end{equation}
\end{theorem}

\begin{proof} The proof is a simple extension of the proof in the univariate
case and will be left to the reader.
\end{proof}

Consider a bi-variate log-normally distributed variable
\begin{displaymath} \left(
\begin{array}{c}
Y_1 \\
Y_2 \\
\end{array} \right)
= \left( \begin{array}{c}
\log(X_1) \\
\log(X_2) \\
\end{array} \right)
\sim N \left( \left( \begin{array}{c}
\mu_1 \\
\mu_2 \\
\end{array} \right) ,
\left( \begin{array}{c}
\sigma_1 \\
\sigma_2 \\
\end{array} \right) \right)
\end{displaymath}
whereby $X_1$ and $X_2$ are correlated normally distributed
variables of joint distribution
\begin{equation}
f(x_1,x_2) = \frac{1}{2 \pi \sigma_1 \sigma_2 \sqrt{1- \rho^2} x_1
x_2} \cdot
\end{equation}
\[  \exp \left\{ -\frac{1}{2(1- \rho^2)} \left[
\left( \frac{\log x_1 - \mu_1}{\sigma_1} \right)^2 - 2 \rho
\frac{(\log x_1 - \mu_1)(\log x_2 - \mu_2)}{\sigma_1 \sigma_2} +
\left( \frac{\log x_2 - \mu_2}{\sigma_2} \right)^2 \right] \right\},
\]
where $\rho$ is a correlation parameter. Both $X_1$ and $X_2$ are
log-normally distributed with
\[ E[X_i] = e^{\mu_i + \sigma_i^2/2} \,,\,\,\,\,\,\,\,
\mathrm{and} \,\,\,\,\,\, E[X_i^2] = e^{2 \mu_i + 2 \sigma_i^2}
\,,\,\,\,\, i = 1,2. \] Matching these with the pre-assigned
moments, and solving for $\mu_i$ and $\sigma_i$ we find
\[ \mu_i = \log \left(
E \left[ I_{t}^{(i)} \right]^2 \Big/ \, \sqrt{E
\left[\left(I_{t}^{(i)} \right)^2 \right]} \right) \,\,\,\,\,\,\,
\textrm{and} \,\,\,\,\,\,\,\,\, \sigma_i^2 = \log \left( E \left[
\left( I_{t}^{(i)}\right)^2 \right] \Big/ \, E \left[ \left(
I_{t}^{(i)} \right) \right]^2 \right) \] Moreover, the mixed term is
\[ E \left[X_1 X_2 \right]
= E \left[ e^{Y_1 + Y_2} \right] = e^{\mu_1 + \mu_2 + \frac{1}{2}
(\sigma_1^2 + \sigma_2^2 + 2 \rho \sigma_1 \sigma_2)} = E[X_1]E[X_2]
\, e^{\rho \sigma_1 \sigma_2} \] which gives $\rho$ (in terms of the
pre-assigned moments):
\begin{equation}
\rho = \frac{1}{\sigma_1 \sigma_2} \log \frac{E \left[ I_{t}^{(1)}
I_{t}^{(2)} \right]} {E \left[ I_{t}^{(1)} \right] E
\left[I_{t}^{(2)} \right]}
\end{equation}

\begin{example} {\bf (Conditional Variance Swaps)} \rm
The payoff of a conditional variance swap is given by the ratio
\begin{equation}
\frac{\int_{T}^{T'} v(y_s,s) \,\mathbf{1}(L < S(y_s) < H) ds}{
\int_{T}^{T'} \mathbf{1}(L < S(y_s) < H) ds} - SR^2
\end{equation}
To apply the moment method to this case, the first thing to note is
that essentially we are modelling the two integrals appearing in the
payoff at the same time. We are going to need a Bi-variate
distribution to do this. Firstly let's write:
\[ I_{t}^{(1)} = \int_{T}^{T'} v(y_s,s) \,\mathbf{1}(L < S(y_s) < H) ds \]
\[ I_{t}^{(2)} = \int_{T}^{T'} \mathbf{1}(L < S(y_s) < H) ds \]
and in order to compute the expectation
\begin{equation}
\label{eq:exp1} E \left[ \frac{I_{t}^{(1)}}{I_{t}^{(2)}} \right]
\end{equation}
we'll need the following expectations:
\[ E \left[ I_{t}^{(1)} \right] \,,\,\,\,\,
E \left[ I_{t}^{(2)} \right] \,,\,\,\,\, E \left[ \left( I_{t}^{(1)}
\right)^2 \right] \,,\,\,\,\, E \left[ \left( I_{t}^{(2)} \right)^2
\right] \,,\,\,\,\, \mathrm{and} \,\,\,\,\,\, E \left[ I_{t}^{(1)}
I_{t}^{(2)} \right] \] To tackle this problem consider the operator
\begin{equation}
{\mathcal{L}}_{\epsilon_1,\epsilon_2} (y_1,y_2) = {\mathcal{L}}
(y_1,y_2) + \epsilon_1 \phi(y_1) \delta_{y_1 y_2} + \epsilon_2
\,\psi(y_1) \delta_{y_1 y_2}
\end{equation}
where
\begin{equation}
\phi(y_1, t) = \sum_{y_2} {\mathcal{L}} (y_1,y_2;t) \log^2 \left(
\frac{S(y_2)}{S(y_1)} \right) \mathbf{1}(L < S(y_1) < H)
\end{equation}
and
\begin{equation}
\psi(y_1) = \mathbf{1}(L < S(y_1) < H).
\end{equation}
We have that
\[ \left. \frac{\partial}{\partial \epsilon_1}
\right\vert_{\epsilon_1 = 0} e^{((t'-t)
{\mathcal{L}}_{\epsilon_1,\epsilon_2})} (y_1,y_2) = E \left[ \left.
I_{t}^{(1)} \right\vert y_t = y_1, y_{t'} = y_2  \right] \] and
\[ \left. \frac{\partial^2}{\partial \epsilon_1^2}
\right\vert_{\epsilon_1 = 0} e^{((t'-t)
{\mathcal{L}}_{\epsilon_1,\epsilon_2})} (y_1,y_2) = E \left[ \left.
\left( I_{t}^{(1)} \right)^2 \right\vert y_t = y_1, y_t' = y_2
\right] \] similarly (but with respect to $\epsilon_2$) for
\[ E \left[ \left. I_{t}^{(2)} \right\vert
y_t = y_1, y_{t'} = y_2 \right] \,\,\,\,\,\,\, \mathrm{and}
\,\,\,\,\,\,\, E \left[ \left. \left( I_{t}^{(2)} \right)^2
\right\vert y_t = y_1, y_{t'} = y_2 \right] \] The joint expectation
will involve the mixed derivative:
\begin{equation}
E \left[ I_{t}^{(1)} I_{t}^{(2)} \right] = \left.
\frac{\partial^2}{\partial \epsilon_1 \partial \epsilon_2}
\right\vert_{\epsilon_1,\epsilon_2 = 0} e^{((t'-t)
{\mathcal{L}}_{\epsilon_1,\epsilon_2})} (y_1,y_2)
\end{equation}
and once computed, we make use of these expectations to match a
bivariate distribution. For simplicity of notation we leave out the
conditional part of these expectations, noting that all the moments
we obtain are conditional to the initial and final points.

To evaluate the expectation $E \left[ I_{t}^{(1)} \Big/ \,
I_{t}^{(2)} \right]$, let us notice that
\[ E \left[ \frac{X_1}{X_2} \right]
= E \left[ e^{Y_1 - Y_2} \right] = E \left[ e^{Y_1 + Y^{'}_2}
\right]
\]
where $Y^{'}_2 = -Y_2$ is also be normally distributed $\sim
N(-\mu_2,\sigma_2^2)$. Hence
\[ E \left[ \frac{X_1}{X_2} \right]
= e^{\mu_1 + \frac{1}{2} \sigma_1^2 } \, e^{- \mu_2 + \frac{1}{2}
\sigma_2^2 } \, e^{-\rho \sigma_1 \sigma_2} \] and the expectation
(\ref{eq:exp1}) is given by
\begin{equation}
E \left[ \frac{I_{t}^{(1)} }{ I_{t}^{(2)} } \right] = E \left[
I_{t}^{(1)} \right] \cdot \frac{E \left[ \left( I_{t}^{(2)}
\right)^2 \right]} {E \left[ I_{t}^{(2)} \right]^3} \cdot \frac{E
\left[ I_{t}^{(1)} \right] E \left[I_{t}^{(2)} \right]} {E \left[
I_{t}^{(1)} I_{t}^{(2)} \right]} = \frac{E \left[ I_{t}^{(1)}
\right]^2 E \left[ \left( I_{t}^{(2)} \right)^2 \right]} {E \left[
I_{t}^{(2)} \right]^2 E \left[ I_{t}^{(1)} I_{t}^{(2)} \right] }
\end{equation}
moreover
\[ E \left[ \left( \frac{X_1}{X_2} \right)^2 \right]
= e^{2\mu_1 + 2 \sigma_1^2 } \, e^{- 2\mu_2 + 2 \sigma_2^2 } \,
e^{-4 \rho \sigma_1 \sigma_2} = E \left[ \frac{X_1}{X_2} \right]^4
e^{-2 \mu_1} e^{2 \mu_2} \] so
\begin{equation}
E \left[ \left( \frac{I_{t}^{(1)} }{ I_{t}^{(2)} } \right)^2 \right]
= \frac{E \left[ I_{t}^{(1)} \right]^4 E \left[ \left( I_{t}^{(2)}
\right)^2 \right]^3 E \left[ \left(I_{t}^{(1)}\right)^2 \right] } {E
\left[ I_{t}^{(2)} \right]^4 E \left[ I_{t}^{(1)} I_{t}^{(2)}
\right]^4 }
\end{equation}
To compute a payoff we will need the expectation $E\left[
\mathrm{min} \left( \frac{X_1}{X_2}, CV_{\mathrm{Max}}  \right)
\right]$ for the payoff of the conditional variance swap.
\begin{equation}
E\left[ \mathrm{min} \left( \frac{X_1}{X_2}, CV_{\mathrm{Max}}
\right) \right] = \int_{0}^{\infty}  \int_{0}^{X_2
CV_{\mathrm{Max}}} \left( \frac{X_1}{X_2} - CV_{\mathrm{Max}}
\right) f(X_1,X_2) dX_1 dX_2 + CV_{\mathrm{Max}}
\end{equation}
This double integral will need to be evaluated numerically. As the
first one has finite bounds and the second has an infinite upper
bound it would make sense to use two types of Gaussian quadratures:
a Gauss-Legendre quadrature on the inner integral and a
Gauss-Laguerre quadrature on the outer one.

\end{example}

\section{Block Factorizations}\label{dcond}

Although most path dependent processes emerging in applications to
Finance are Abelian, some aren't. If the Abelian property fails, the
propagator cannot be block-diagonalized by the methods discussed
above and also moment methods generally fail. In several relevant
situations one can still achieve a numerically viable framework by
using block-factorizations instead of diagonalizations.

Let $U_{i}(y_1, y_2)$ be a family of Markov propagator indexed by $i
= 0 , 1, ... n-1$ defined over the time intervals $[t_i, t_{i+1}]$.
Consider a propagator lifting of the form
\begin{equation}
\bar U_{i}(y_1, k_1; y_2, k_2) = U_{i}(y_1, y_2) Q_i(y_1, y_2)_{k_1
k_2}. \label{equibf}
\end{equation}
where $k_1, k_2 = 0 ... K$.

\begin{definition} {\bf (Block-Factorization.)}
We say that the propagators in (\ref{equibf})
admit a block-factorization if they can be represented in the form:
\begin{equation}
\bar U_{i} = \Pi_i \cdot \big(U_i \otimes {\mathbb I}\big)
\label{equibf2}
\end{equation} Here $\Pi_i, i=0,...n-1$ is a family
of permutation matrices, i.e. matrices with the property that for
each pair $(y_1, k_1)$ we have that
\begin{equation}
\Pi_i(y_1, k_1; y_2, k_2) = 0
\end{equation}
for all pairs $(y_2,k_2)$ except for one single pair $(Y(y_1, k_1),
K(y_1, k_1))$ for which we have instead
\begin{equation}
\Pi_i(y_1, k_1; Y(y_1, k_1), K(y_1, k_1)) = 1.
\end{equation}
Furthermore, the tensor product notation in equation (\ref{equibf2})
stands for the operator with the following matrix elements:
\begin{equation}
\big(U_i \otimes {\mathbb I}\big)(y_1, k_1; y_2, k_2) = U_{i}(y_1,
y_2) \delta_{k_1 k_2}.
\end{equation}
\end{definition}

If a block-factorization exists, then an efficient backward
induction algorithm can be setup. Namely, if $v$ is a vector, then
we can value the following path ordered products
\begin{equation}
v_i = P \prod_{j=i}^{n-1} \bar U_{j} v_n = P \prod_{j=i}^n  \Pi_j
\cdot \big(U_j \otimes {\mathbb I}\big) v_n
\end{equation}
iteratively in $i$ from $i = n-1$ down to $i=0$. At the first step
one applies the operator $\big(U_{n-1} \otimes {\mathbb I}\big)$ to
$v_{n}$, followed by the permutation $\Pi_{n-1}$ to obtain $v_{n-1}$
and then iterate. Due to the tensor product structure of the first
operator, it is convenient to apply by regarding the vector $v_i$ as
a matrix of elements $v_i(y, k)$. This representation makes it
possible to leverage on numerical efficiencies and to re-interpret a
high dimensional BLAS level 2 matrix-vector multiplication as a
lower dimensional BLAS level 3 matrix-matrix multiplication.
Applying a permutation operator is then quite straightforward from
the numerical viewpoint and is an operation whose complexity scales
linearly with respect to the dimension of the vector $v$.

\begin{example} {\bf (Snowballs)} \rm Here we follow
\cite{AIRmodel2007}.  Consider the case of a valuing a snowballs for
which the structured coupon at time $T_i = T + (\Delta T) i$ has the
following form:
\begin{equation}
C_{T_i} = f C_{T_{i-1}} + \Phi_i(y_{T_{i-1}})
\end{equation}
where the factor $f$ is a fixed parameter and $\Phi_i(y_{T_{i-1}})$
is a given function. Since the coupon amount at a given time affects
the coupon amou/nt in the next period, the process is not Abelian,
in fact it is path-dependent. However, block-factorizations are
still possible.

Consider discretizing the coupon dimension in intervals $\Delta C$,
so that a generic coupon can be approximated as follows
\begin{equation}
C_{T_i} = (\Delta C) k_{T_i}
\end{equation}
where $ k_{T_i} = 0, 1, .... N-1$ is a discrete time, integer value
process. A strategy to implement the backward induction scheme is to
organize the payoff function $v(y, k)$ at maturity in a matrix with
$N$ columns, each one indexed by the state variable $y$ and
representing the price conditional to the discretized value of the
previous coupon paid. One can then iterate backward by applying the
pricing kernel to this matrix of conditional pricing functions.
After each step in the iteration, one needs to reshuffle the pricing
functions in such a way that the conditioning relation on each
column is satisfied. More precisely, we set
\begin{equation}
v_{i-1}(y_1, k) = \sum_{y_2} U_{i}(y_1, y_2) v_{i}(y_2, [ (f (\Delta
C) \cdot k_1 + \Phi_i(y_1))/ (\Delta C)]) = \bigg[\Pi_{i} \cdot
\big(U_i \otimes {\mathbb I}\big) v_i\bigg](y_1, k)
\end{equation}
where
\begin{equation}
\Pi_{i}(y_1, k_1; y_2, k_2) = \delta(y_2 - y_1) \delta(k_2 - K(y_1,
k_1))
\end{equation}
and
\begin{equation}
K(y_1, k_1) = [ (f (\Delta C) \cdot k_1 + \Phi_i(y_1))/ (\Delta C)].
\end{equation}
Notice that this backward induction scheme can accommodate
callability provisions.
\end{example}

\begin{example} {\bf (Soft Calls)} \rm Soft calls provide another interesting example
of a non-Abelian process. In this case, a call decision depends on
the fraction of time the process spends above a given barrier level
during a fixed time period prior to the decision itself. By
restricting times at which one makes observations to the discrete
sequence of time points $t_i = T + i \delta t, i = 0,...M$, the
pricing function takes the form
\begin{equation}
v_i(y, \overarrow \sigma)
\end{equation}
where
\begin{equation}
\overarrow \sigma = (\sigma_1, ..., \sigma_N) \in \{0,1\}^N .
\end{equation}
Here the variable $\sigma_j$ equals 1 if $y_{t_{i-N+j}}$ satisfies
the barrier condition and 0 otherwise. We model this dependency by
means of a generic function notation
\begin{equation}
\sigma_j = \Sigma_i(y_{t_{i-N+j}}).
\end{equation}
The sequences $\overarrow \sigma$ can be arranged in an ordered
sequence of $2^N$ integers so that the data structure $v_i(y,
\overarrow \sigma)$ appears as a matrix with $2^N$ columns. A
backward induction scheme involves evaluating
\begin{equation}
v_{i-1}(y_1, \overarrow \sigma) = \sum_{y_2} U_{i}(y_1, y_2)
v_{i}(y_2, \Phi_i(\overarrow \sigma, y_1))
\end{equation}
where
\begin{equation}
\Phi_i(\overarrow \sigma, y_1) = \{\sigma_2, ... \sigma_{N-1},
\Sigma(t_i, t_{i-1}, y_1) \}
\end{equation}
This form is block-factorized as
\begin{equation}
v_{i-1}(y_1, \overarrow \sigma) = \bigg[\Pi_{i} \cdot \big(U_i
\otimes {\mathbb I}\big) v_i\bigg](y_1, k)
\end{equation}
where $\Pi_{i}$ is the permutation operator of matrix
\begin{equation}
\Pi_{i}(y_1, \overarrow \sigma_1; y_2, \overarrow \sigma_2) =
\delta(y_2 - y_1) \delta\big((\overarrow\sigma_2)_N - \Sigma_i(
y_1)\big) \prod_{j=1}^{N-1} \delta\big((\overarrow\sigma_2)_j -
(\overarrow\sigma_1)_{j+1}\big).
\end{equation}
\end{example}

More general types of useful block-factorizations can be imagined,
although they might be numerically less efficient. An example is set
forth by the following definition:

\begin{definition} {\bf (State Dependent Block-Factorization.)}
We say that the propagators in (\ref{equibf})
admit a state dependent block-factorization if they can be
represented in the form:
\begin{equation}
\bar U_{i} = \Pi_i \cdot \bigoplus_{k=1}^K U_{i,k} \label{equibf3}
\end{equation}
Here $\Pi_i, i=0,...n-1$ is a family of permutation matrices as
above. The direct sum notation in equation (\ref{equibf3}) stands
for the operator with the following matrix elements:
\begin{equation}
\big(\bigoplus_{k_0 = 1}^K U_{i,k_0} \big)(y_1, k_1; y_2, k_2) =
U_{i, k_1}(y_1, y_2) \delta_{k_1 k_2}.
\end{equation}
\end{definition}

A backward induction in this case is still a lower dimensional
problem but since the lower dimensional propagator is state
dependent this reduces to a sequence of matrix-vector
multiplications as opposed to a single matrix-matrix multiplication.
The scheme is thus numerically less efficient. Still it would be
useful in cases for instance such as GARCH type models and
extensions, see \cite{Engle1982}.

\section{Dynamic Conditioning}\label{dcond}

This section is based on work in collaboration with Alicia Vidler,
see \cite{AVidler}.

Modeling correlations between several processes has sometimes been
considered a problem affected by the so-called curse of
dimensionality which causes the numerical complexity to explode
exponentially, see \cite{Bellman1957}. This motivates the recourse
to Monte Carlo methods while calibrating using closed form
solutions.

Here we deviate substantially from this framework and aim at
building models viable from the engineering deployment viewpoint and
specified semi-parametrically or even non-parametrically without any
restriction imposed by analytical tractability. The method of fast
exponentiation allows one to calibrate single name models without
recurring to analytic solvability thanks to providing smooth
sensitivities. Although Monte Carlo schemes could in principle be
employed in a model constructed non-parametrically and calibrated
with operator methods, in many circumstances important for
applications this can be avoided by using the technique of dynamic
conditioning discussed in this section. This technique can be
regarded as a dynamic copula whereby the single factor marginal
distributions are preserved and the numerical complexity grows
linearly with the number of factors, similarly to what happens with
Monte Carlo methods. But, unlike Motecarlo methods, dynamic
conditioning is numerically noiseless as there is no sampling, no
variance reduction scheme is needed and even features such as
callability are numerically treatable.

Dynamic conditioning is based on constructing a hierarchy of
conditioning relations. At the base of the hierarchy we have
continuous time lattice models for single factor marginal
distributions. Next, one introduces a binomial process for each risk
factor to condition the corresponding Markov chain. Next one finds a
hierarchy of binomial processes to condition the former conditioning
lattices. See figure (\ref{figdyncond}).

To describe the technique of dynamic conditioning in detail,
consider a particular risk factor described by a lattice process
whose filtered probability space is engendered by a Markov
propagator $U(y_1, t_1; y_2, t_2)$. For each such single factor, we
introduce a conditioning binomial process $h_{t}\in\ZZZ$, which is
constant over the time intervals $[T_i, T_{i+1})$ where $T_i = T + i
\Delta T$, $i = 0, 1, 2, ... N $ and $\Delta T = (T'-T)/N$. At
initial time we set $h_{T} = 0$ while for all $i>0$ we have that
$h_{T_i} - h_{T_{i-1}} = \pm 1$. The elementary propagator for the
process $h_t$ across neighboring time points $T_i$ is
\begin{equation}
V(h_i, T_i; h_{i+1}, T_{i+1}) = q_1(h_i, i) \delta_{h_i+1, h_{i+1}}
+ q_{-1}(h_i, i) \delta_{h_i-1, h_{i+1}}
\end{equation}
where $q_1(h_i, i), q_{-1}(h_i, i)>0$ and $q_1(h_i, i)+q_{-1}(h_i,
i) =1$. On a general time interval $[T_i,T_j]$, we have that
\begin{equation}
V(h_i, T_i; h_j, T_j)  = \sum_{h_t: h_{T_i} = h_i, h_{T_j} = h_j}
\prod_{k=i}^{j-1} V(h_{T_k}, T_k; h_{T_{k+1}}, T_{k+1}).
\end{equation}

\begin{figure}
    \includegraphics[width = 16cm]{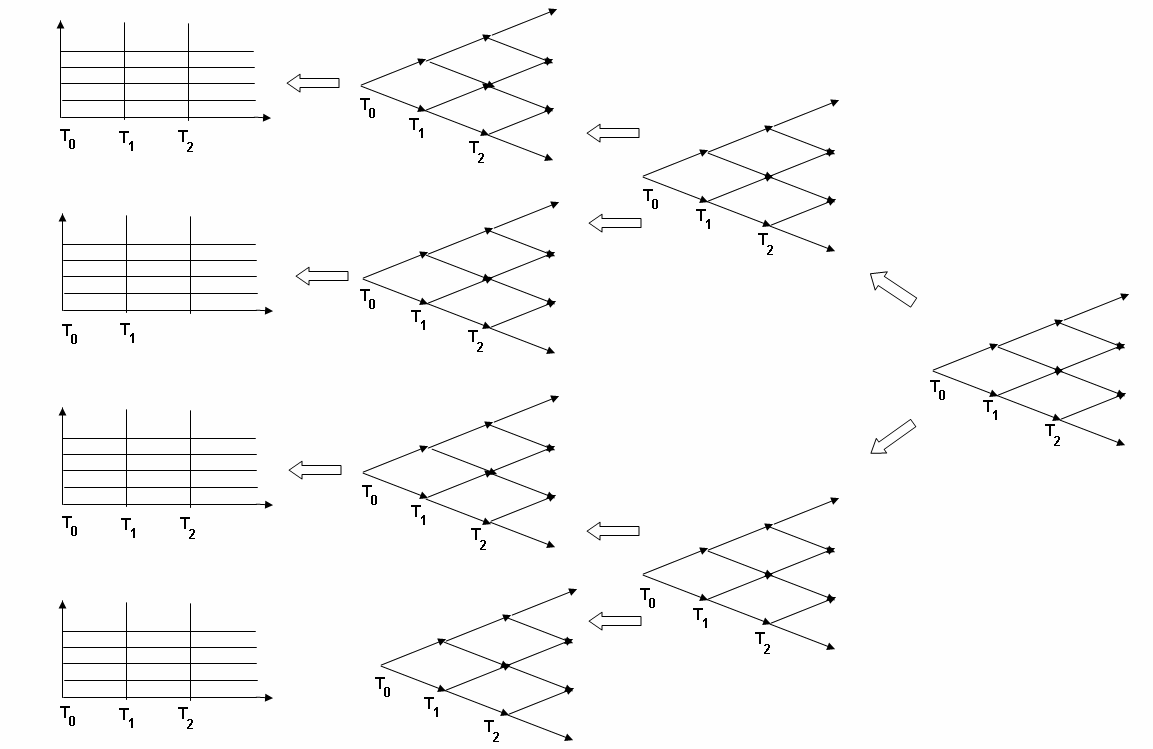}
    \caption{Dynamic conditioning scheme: a binomial process conditions each risk factor and
    is in turn conditioned by another binomial process. The latter is also conditioned by a hierarchy of
    binomial processes.}
    \label{figdyncond}
\end{figure}

Next we define a lifted conditional propagator $\bar U(y_i, h_i,
T_i; y_j, h_j, T_j)$ in such a way to preserve unconditional
transition probabilities from the starting time at $T$, i.e. so that
\begin{equation}
\sum_{ h_j } \bar U(y_0, 0, T; y_j, h_j, T_j) = U(y_0, T; y_j, T_j).
\label{eq_condu}
\end{equation}
To do so, one strategy is to define two propagators for each node
$(h, T_i)$, namely to choose a pair of operators
 $U_{1}(y_i, h_i, T_i; y_{i+1}, h_i+1, T_{i+1})$
and $U_{-1}(y_i, h_i, T_i; y_{i+1}, h_i-1, T_{i+1})$ so that
\begin{align*}
U(y_i, T_i; y_{i+1}, T_{i+1}) &= q_1(h, i) U_{1}(y_i, h, T_i;
y_{i+1}, h+1,
T_{i+1}) + \\
&\hskip2cm q_{-1}(h, i) U_{-1}(y_i, h, T_i; y_{i+1}, h-1, T_{i+1}).
\end{align*}
The operator $\bar U(y_0, 0, T; y_j, h_j, T_j)$ satisfying
(\ref{eq_condu}) is defined as the following sum:
\begin{align*}
\bar U(y_0, h_0, T_0; y_j, h_j, T_j) &= \sum_{h_t: h_T = 0, h_{T_j}
= h_j}
\sum_{y_1,..y_{j-1}} \prod_{k=1..j} q_{h_{T_k}-h_{T_{k-1}}}(h_{T_{k-1}}) \times \\
&\hskip2cm U_{h_{T_k}-h_{T_{k-1}}} (y_{k-1}, h_{T_{k-1}}, T_{k-1};
y_k, h_{T_{k}}, T_{k}).
\end{align*}

Since the operators $U_{1}$ and $U_{-1}$ do not commute with each
other, if not in trivial situations, each path $h_t$ in the
summation on the right hand side of this equation represents a
different operator. In many situations one can however avoid the
numerical complexities of a non-recombining tree by finding modified
versions of such operators so that the kernels
\begin{align*}
U_{h_\cdot} (y_0, y_{k}) =  \sum_{y_1,..y_{j-1}} \prod_{k=1..j}
q_{h_{T_k}-h_{T_{k-1}}}(h_{T_{k-1}}) \bar U(y_{k-1}, h_{T_{k-1}},
T_{k-1}; y_k, h_{T_{k}}, T_{k})
\end{align*}
are all equal to each other, for all paths $h_t: h_T = 0, h_{T_j} =
h_j$ and at least one single fixed starting point $y_0 = \bar y_0$.
This will be referred to as {\it conditional recombination
property}.

The reasons why we focus on the initial point $y_0 = \bar y_0$ only
are manifold. Firstly, in applications one needs to condition
marginal distributions only to spot values, as the price of options
in the hypothetical case asset prices were different is not known.
Secondly, if we insisted on the same property being valid for all
initial starting points we would end up with a seriously ill posed
problem of difficult solution. Because of these reasons, we settle
for the more modest objective of conditional recombination.

Namely, on each node $(h_i, T_i)$ with $i>0$ we define
\begin{equation}
\bar U(y_{i-1}, h_{i-1}, T_{i-1}; y_{i}, h_{i}, T_{i}) =
U_{h_{i}-h_{i-1}} (y_{i-1}, h_{i-1}, T_{i-1}; y_{i}, h_{i}, T_{i})
\label{eq_margu}
\end{equation}
in case $h_i = \pm i$. Otherwise, we determine this operator so that
\begin{align*}
\bar U&(\bar y_0, h_0, T_0; y_{i}, h_{i}, T_{i}) = \\
&\hskip2cm \sum_{y_{i-1}\in\Lambda} \bar U(\bar y_0, h_0, T_0;
y_{i-1}, h_{i-1}, T_{i-1}) \bar U(y_{i-1}, h_{i-1}, T_{i-1}; y_{i},
h_{i}, T_{i})
\end{align*}
for all $y_{i}\in\Lambda$ and a fixed $\bar y_0$. This can be
achieved in more than one way. As a guideline, it is advisable to
deviate the least possible from the definition (\ref{eq_margu}).

As a next step in the construction, consider a second binomial
process $c_t$
which is piecewise constant across neighboring time points $T_i$.
The dynamics of $c_t$ is by construction correlated to that of
$h_t$. More precisely, the propagator for the joint process is
\begin{align*}
W(h_i, & c_i, T_i; h_{i+1}, c_{i+1}, T_{i+1}) = \\
& q_{++}(h_i, c_i, i) \delta_{h_i+1, h_{i+1}}\delta_{c_i+1, c_{i+1}}
+ q_{+-}(h_i, c_i, i) \delta_{h_i+1, h_{i+1}}\delta_{c_i-1,
c_{i+1}}\\
&+ q_{-+}(h_i, c_i, i) \delta_{h_i-1, h_{i+1}}\delta_{c_i+1,
c_{i+1}} + q_{--}(h_i, c_i, i) \delta_{h_i-1, h_{i+1}}\delta_{c_i-1,
c_{i+1}}
\end{align*}
where $q_{\pm\pm}(h_i, c_i, i) \ge 0$ and
\begin{equation}
q_{++}(h_i, c_i, i)+q_{+-}(h_i, c_i, i)+q_{-+}(h_i, c_i,
i)+q_{--}(h_i, c_i, i) =1.
\end{equation}
Due to the conditional recombination property, if $y_0 =\bar y_0$,
then the conditional propagators resum with a simple formula
\begin{align*}
& \sum_{h_t: h_T = 0, h_{T_j} = h_j} \sum_{y_1,..y_{j-1}}
\prod_{k=1..j} W(h_i, c_i, T_i; h_{i+1}, c_{i+1}, T_{i+1})
 \times \\
&\hskip1cm q_{h_{T_k}-h_{T_{k-1}}}(h_{T_{k-1}}) \bar
U_{h_{T_k}-h_{T_{k-1}}} (y_{k-1}, h_{T_{k-1}}, T_{k-1}; y_k,
h_{T_{k}}, T_{k})\\
& \hskip1cm = \sum_{h_j} W(0, 0, T_0; h_{j}, c_{j}, T_{j}) \bar
U(y_0, 0, T_0; y_j, h_j, T_j)\\
& \hskip1cm \equiv \tilde U(y_0, 0, T_0; y_j, c_j, T_j).
\end{align*}

As a next step, the construction above is repeated for a
number of different risk factors associated to $N$ lattice processes
$y_t^{(\alpha)}$, where $\alpha = 1,..N$ all correlated to the
process $c_t$, possibly in different ways. Then, conditioned to
starting all processes at fixed lattice points $y_0^{(\alpha)} =
\bar y_0^{(\alpha)}$ and conditioned to $c_{T_j} = c_j$, the
multi-factor propagator factorizes into the product of conditional
single factor propagators
\begin{equation}
\prod_{\alpha = 1...N} \tilde U(y_0^{(\alpha)}, 0, T_0; y_j^{(\alpha)},
c_j, T_j)
\end{equation}
This is the key formula which we use to correlate
processes while ensuring that numerical complexity increases only
linearly with the number of factors $N$.

The construction can obviously be iterated and we can have a whole
hierarchy of binomial processes conditioning each other to model
multi-factor correlations.

\section{Conclusion}

We reviewed a new framework for Mathematical Finance and the theory
of stochastic processes based on operator methods. The framework
grew over the years through applied research and solving specific
concrete problems in derivative pricing. As wrong ideas were weeded
out, a coherent framework emerged and is now summarized in this
review paper. From the numerical viewpoint, the methods rely on the
ability to multiply efficiently large matrices as the main
computational engine. Technologies for matrix manipulation are
currently being developed at great pace. The emerging multi-core GPU
chipsets will soon provide affordable teraflop performance for
algorithms based on matrix-matrix manipulations such as ours. To
both set the theory of stochastic processes on firm theoretical
grounds and justify the empirically observed convergence rates, we
derive pointwise kernel convergence estimates of a new type which
explain the observed smoothing properties of the method. We
introduce the notion of Abelian process to generalize the concept of
stochastic integral and extend the classic
Cameron-Feynman-Girsanov-Ito-Kac-Martin theorem in multiple
directions. This theorem is here reinterpreted as a
block-diagonalization scheme for large matrices corresponding to
Abelian processes. We outline solution methods for Abelian path
dependent options, a class we identified and which comprises the
great majority of the path-dependent options currently traded.
Important non-Abelian processes are also considered and discussed by
means of a weaker but still effective method of
block-factorizations.  Furthermore, we also discuss a method for
dynamic conditioning which applies to correlation products such as
baskets and hybrid derivatives.

\bibliographystyle{giwi}
\bibliography{abelian}

\end{document}